\documentclass[12pt]{amsart}

\usepackage{epsf,amssymb,times,overpic,amscd,diagrams}

\newtheorem{thm}{Theorem}[section]
\newtheorem{prop}[thm]{Proposition}
\newtheorem{lemma}[thm]{Lemma}
\newtheorem{cor}[thm]{Corollary}
\newtheorem{conj}[thm]{Conjecture}

\newtheorem{fact}[thm]{Fact}

\theoremstyle{definition}
\newtheorem{defn}[thm]{Definition}

\theoremstyle{remark}
\newtheorem{q}[thm]{Question}
\newtheorem{rmk}[thm]{Remark}
\newtheorem{example}[thm]{Example}

\newcommand{\C}{\mathbb{C}}

\newcommand{\R}{\mathbb{R}}
\newcommand{\Z}{\mathbb{Z}}
\newcommand{\Q}{\mathbb{Q}}
\newcommand{\N}{\mathbb{N}}

\newcommand{\bdry}{\partial}
\newcommand{\s}{\vskip.1in}
\newcommand{\n}{\noindent}

\newcommand{\F}{\mathbb{F}}

\newcommand{\m}{\underline}
\newcommand{\op}{\operatorname}
\newcommand{\ora}{\overrightarrow}

\newcommand{\be}{\begin{enumerate}}
\newcommand{\ee}{\end{enumerate}}

\begin{document}

\title[Sutures and contact homology]{Sutures and contact homology I}

\author{Vincent Colin}
\address{Universit\'e de Nantes, 44322 Nantes, France}
\email{Vincent.Colin@math.univ-nantes.fr}

\author{Paolo Ghiggini}
\address{Universit\'e de Nantes, 44322 Nantes, France}
\email{paolo.ghiggini@univ-nantes.fr}
\urladdr{http://www.math.sciences.univ-nantes.fr/\char126 ghiggini}

\author{Ko Honda}
\address{University of Southern California, Los Angeles, CA 90089}
\email{khonda@math.usc.edu} \urladdr{http://math.usc.edu/\char126
khonda}

\author{Michael Hutchings}
\address{UC Berkeley, Berkeley, CA 94720}
\email{hutching@math.berkeley.edu}
\urladdr{http://math.berkeley.edu/\char126 hutching}




\thanks{VC supported by the Institut
Universitaire de France, ANR Symplexe, and ANR Floer Power. PG
supported by ANR Floer Power. KH supported by NSF Grant DMS-0805352.
MH supported by NSF Grant DMS-0806037.}

\begin{abstract}
We define a relative version of contact homology for contact
manifolds with convex boundary, and prove basic properties of this
relative contact homology.  Similar considerations also hold for
embedded contact homology.
\end{abstract}

\maketitle

\section{Introduction and main results}

The goal of this paper is to define relative versions of contact
homology and embedded contact homology for contact manifolds with
convex boundary and to prove basic properties of these {\em relative
  contact homology} theories.  {\em Contact homology}, due to
Eliashberg-Hofer and part of the symplectic field theory (SFT) package
of Eliashberg-Givental-Hofer \cite{EGH}, is a Floer-type invariant of
a (closed) contact manifold.  It is the homology of a differential
graded algebra whose differential counts genus zero holomorphic curves
in the symplectization with one positive puncture and an arbitrary
number of negative punctures.  Contact homology has been quite
successful at distinguishing contact structures, as can be seen for
example from the works of Bourgeois-Colin~\cite{BC} and
Ustilovsky~\cite{U}.
{\em Embedded contact homology} (ECH) is a variant of contact
homology/SFT for three-dimensional contact manifolds, defined in
\cite{Hu1,HS,HT1,HT2}, which is the homology of a chain complex whose
differential counts certain embedded holomorphic curves, possibly of
higher genus, in the symplectization.
Although ECH is defined in terms of a contact form, it is actually a
topological invariant of the underlying $3$-manifold, i.e.\ it does
not depend on the contact structure (up to a possible grading shift, see
Section~\ref{subsection: contact homology for sutured}).  This
invariance follows from a theorem of Taubes \cite{T2} identifying ECH
with Seiberg-Witten Floer cohomology, which also implies the Weinstein
conjecture in dimension three \cite{T1}.

Let $M$ be a compact, oriented $(2n+1)$-dimensional manifold with
boundary.  A natural boundary condition for an oriented contact
structure $\xi$ on $M$ to satisfy is that $\bdry M$ be {\em
  $\xi$-convex}. The notion of convexity in contact geometry was
introduced by Eliashberg-Gromov~\cite{EG}, and developed by
Giroux~\cite{Gi1}. A thorough discussion will be given in
Section~\ref{subsection: convex submanifolds}, but we briefly give
definitions here: A $2n$-dimensional submanifold $\Sigma\subset M$ is
{\em $\xi$-convex} if there is a contact vector field $X$ transverse
to $\Sigma$. To a $\xi$-convex submanifold $\Sigma$ and a transverse
contact vector field $X$ we can associate the {\em dividing set}
$\Gamma=\Gamma_X\subset \Sigma$, namely the set of points $x\in
\Sigma$ such that $X(x)\in \xi(x)$. By the contact condition,
$(\Gamma,\xi\cap T\Gamma)$ is a $(2n-1)$-dimensional contact
submanifold of $(M,\xi)$; the isotopy class of $(\Gamma,\xi\cap
T\Gamma)$ is independent of the choice of $X$. The set of points $x\in
\Sigma$ where $X$ is positively (resp.\ negatively) transverse to
$\xi$ will be denoted by $R_+(\Gamma)$ (resp.\ $R_-(\Gamma)$). We
denote by $(M,\Gamma,\xi)$ the contact manifold $(M,\xi)$ with convex
boundary and dividing set $\Gamma=\Gamma_X\subset \bdry M$ with
respect to some transverse contact vector field $X$.  We emphasize
that, in this paper, $\Gamma$ is a submanifold of $\Sigma$, {\em not
  an isotopy class of submanifolds of $\Sigma$.}

\subsection{Invariants of sutured contact manifolds}
\label{subsection: contact homology for sutured}

\subsubsection{Sutured contact homology and sutured ECH}

Our first result is that the contact homology algebra and, in the
three-dimensional case, embedded contact homology can be defined for a
contact manifold $(M,\Gamma,\xi)$ with convex boundary, extending the
usual definitions. A slight subtlety is that the actual boundary
condition we want to use is not that $\partial M$ be $\xi$-convex, but
rather that $(M,\Gamma,\xi)$ be a {\em sutured contact manifold\/}.
Roughly speaking this is a sutured manifold, essentially as defined by
Gabai \cite{Ga}, with a contact structure adapted to the sutures.  The
precise definition of sutured contact manifold is given in
Section~\ref{subsection: sutured contact manifolds}, and
Section~\ref{subsection: from convex to sutured} explains how to pass
between the convex and sutured boundary conditions.  For now we write
$(M,\Gamma,\xi)$ to indicate either of these boundary conditions, and
we refer to $\Gamma$ interchangeably as a ``suture'' or a ``dividing
set''.

\begin{thm} \label{thm: well-definition} Let $(M,\Gamma,\xi)$ be a
  $(2n+1)$-dimensional sutured contact manifold. Then:
\begin{enumerate}
\item The contact homology algebra $HC(M,\Gamma,\xi)$ is defined
and independent of the choice of contact $1$-form $\alpha$ with
$\ker\alpha=\xi$, adapted almost complex structure $J$, and abstract
perturbations.
\item Suppose $\dim{M}=3$. Then the embedded contact homology
  $ECH(M,\Gamma,\alpha,J)$ is defined.
\end{enumerate}
\end{thm}

Here contact homology is
defined over $\Q$.  One reason for this is that multiply covered Reeb orbits 
force one to use coefficients in $\Q$ or some extension thereof. On the other hand, 
ECH is defined over $\Z$. 

The definitions of these versions of contact homology, as well as the
proof of Theorem~\ref{thm: well-definition}, are be given in
Section~\ref{section: defn of sutured CH}.  The basic idea is to copy
the definitions from the closed case, and to argue that the relevant
Gromov compactness carries over.

Note that already in the closed case, the definition and proof of
invariance of contact homology require some abstract perturbations of
the moduli spaces of holomorphic curves (due to the presence of
multiply covered holomorphic curves of negative index).  This
construction is still in progress, using the polyfold technology being
developed by Hofer-Wysocki-Zehnder, see \cite{Ho3}.  The proof of
Theorem~\ref{thm: well-definition}(1) assumes that the machinery
needed to construct contact homology in the closed case works, see
Section~\ref{section: defn of sutured CH} for details. 

The differentials in both contact homology and ECH depend also on the 
choice of a coherent orientation of the moduli spaces, see \cite{BM} 
for contact homology and \cite[Section 9]{HT2} for ECH. Since the 
construction of the coherent orientation is local, it carries over unchanged 
in the sutured case. Different choices of coherent orientations yield different,
but canonically isomorphic, chain complexes.

If $A$ is a homology class in $H_1(M)$, then we write
$HC(M,\Gamma,\xi,A)$ for the homology of the subcomplex generated by
monomials $\gamma_1\dots\gamma_k$, where $\gamma_i$ is a closed orbit
of the Reeb vector field $R_\alpha$ corresponding to $\alpha$, and
$\sum_{i=1}^k[\gamma_i]=A$.  Also write $ECH(M,\Gamma,\alpha,J,A)$ for
the homology of the subcomplex generated by orbit sets
$\{(\gamma_i,m_i)\}_{i=1}^k$ where $\sum_{i=1}^k m_i[\gamma_i]=A$.

\subsubsection{Conjectural topological invariance of sutured ECH}

In the closed case, ECH is a topological invariant of the underlying
3-manifold in the following sense: If $M$ is a closed 3-manifold, if
$\alpha_i$ is a contact form on $M$ and $J_i$ is a generic
$\alpha_i$-adapated almost complex structure as needed to define the
ECH chain complex for $i=1,2$, and if $A_1\in H_1(M)$, then
\[
ECH(M,\alpha_1,J_1,A_1) \simeq ECH(M,\alpha_2,J_2,A_2),
\]
as relatively graded $\Z$-modules, where
\begin{equation}
\label{eqn:A2-A1}
A_2-A_1=\operatorname{PD}(\frak{s}_{\xi_1}-\frak{s}_{\xi_2}).
\end{equation}
Here $\frak{s}_{\xi_i}$ denotes the ${\rm Spin}^c$ structure determined by
$\xi_i=\operatorname{Ker}(\alpha_i)$, and
$\frak{s}_{\xi_1}-\frak{s}_{\xi_2}\in H^2(M;\Z)$ denotes the
difference between the two ${\rm Spin}^c$ structures.  The above invariance
follows from the theorem of Taubes \cite{T2} identifying
$ECH_*(M,\alpha_i,J_i,A_i)$ with the Seiberg-Witten Floer cohomology
$\widehat{HM}^{-*}(M,\frak{s}_{\xi_i}+\operatorname{PD}(A_i))$, up to
a possible grading shift\footnote{Both Seiberg-Witten Floer homology
  and ECH have absolute gradings by homotopy classes of oriented
  2-plane fields on $M$, see \cite{KM3,Hu2}, and it is natural to
  conjecture that Taubes's isomorphism between them respects these
  gradings.}.

This motivates the following conjecture in the sutured case:

\begin{conj} \label{conj for ECH}
The sutured embedded contact homology $ECH(M,\Gamma,\alpha,J)$ does
not depend on the choice of contact form $\alpha$, contact structure
$\xi=\ker \alpha$, or almost complex structure $J$. More precisely,
\[
ECH(M,\Gamma,\alpha_1,J_1,A_1)\simeq
ECH(M,\Gamma,\alpha_2,J_2,A_2)
\]
as relatively graded $\F$-modules, when $A_1$ and $A_2$ are related
by \eqref{eqn:A2-A1}.
\end{conj}

\begin{rmk}
  We need to explain why equation \eqref{eqn:A2-A1} still makes sense
  in the sutured case.  The difference between two Spin$^c$-structures
  on $M$ is an element of $H^2(M;\Z)=H_1(M,\partial M)$.  However for
  a sutured manifold one has a fixed $2$-plane field on $\partial M$
  determined by the sutures, which determines a canonical
  Spin$^c$-structure $\mathfrak{s}_0$ in a neighborhood of $\partial
  M$.  A contact structure $\xi$ compatible with the sutures then
  determines a relative Spin$^c$-structure relative to
  $\mathfrak{s}_0$, which means a Spin$^c$-structure
  $\mathfrak{s}_\xi$ on $M$ together with an isomorphism of
  $\mathfrak{s}_\xi|_{\partial M}$ with $\mathfrak{s}_0$. These
  relative ${\rm Spin}^c$ structures comprise
  an affine space over $H^2(M,\partial M;\Z)=H_1(M)$.
\end{rmk}

\subsubsection{Invariants of Legendrian submanifolds}

Let $(M,\xi)$ be a closed $(2n+1)$-dimensional contact manifold.  Then
we can define an invariant $HC(M,\xi,L)$ of a Legendrian submanifold
$L$ in $(M,\xi)$ as follows: Let $N(L)$ be a Darboux-Weinstein
neighborhood of $L$.  Then $\bdry (M- N(L))$ is a convex submanifold
of $M$ with dividing set $\Gamma_{\bdry(M-N(L))}$.  We now define
$$HC(M,\xi,L)=HC(M-N(L),\Gamma_{\bdry(M-N(L))},\xi|_{M-N(L)}).$$
Similarly, in dimension three, if $\xi=\ker \alpha$, then we can
define
\[
ECH(M,\alpha,J,L)=ECH(M-N(L),\Gamma_{\bdry(M-N(L))},\alpha',J'),
\]
where $\alpha',J'$ are obtained from $\alpha,J$ by a modification near
$\bdry N(L)$. If Conjecture~\ref{conj for ECH} is true, then
$ECH(M,\alpha,J,L)$ depends only on the ambient manifold $M$ and the
framing of the knot $L$, as a relatively graded $\F$-module. The
details of the Legendrian knot invariants are given in
Section~\ref{subsection: invts of leg submanifolds}.

\subsection{Comparison with sutured Floer homology}

In this section $\dim{M}=3$.

The definition of the sutured versions of contact homology
theories has been known at least since the work \cite{CH}.  However,
additional impetus for the current work came from the recent
foundational work of Juh\'asz~\cite{Ju1,Ju2} on the sutured version
of Heegaard Floer homology. Juh\'asz' work also motivated the
definition of a sutured version of Seiberg-Witten Floer homology by
Kronheimer and Mrowka~\cite{KM}.

\begin{defn}
A sutured 3-manifold $(M,\Gamma)$ (see Section~\ref{subsection: sutured
  contact manifolds}) is called {\em balanced} if $M$ has no
closed components, the map $\pi_0(\Gamma)\to \pi_0(\bdry M)$ is surjective,
and $\chi(R_+(\Gamma))=\chi(R_-(\Gamma))$ on the boundary of each
component of $M$.
\end{defn}

To a balanced sutured 3-manifold $(M,\Gamma)$, Juh\'asz assigned the
{\em sutured Floer homology} module $SFH(M,\Gamma)$, which generalizes
the ``hat'' version of Heegaard Floer homology and link Floer homology
as follows. Let $M$ be a closed oriented $3$-manifold. If we define
the sutured manifold $M(1)$ to be the pair consisting of $M-B^3$ and
suture $S^1$ on $\bdry B^3$, then one has
\begin{equation}
\label{eqn:SFHM(1)}
SFH(M(1))\simeq
\widehat{HF}(M),
\end{equation}
where the right hand side is the ``hat'' version of Heegaard Floer homology.
Next, if $L\subset M$ is a link, define the sutured manifold $M(L)$ to
be the pair consisting of $M-N(L)$ and suture which consists of two
meridian curves on each component of $\bdry N(L)$.  Juh\'asz then
showed that $SFH(M(L))$ is isomorphic to the link Floer homology of
$L$.

If $(M,\Gamma,\xi)$ is a sutured contact 3-manifold with no closed
components, then the sutured manifold $(M,\Gamma)$ is automatically
balanced. To see this, recall the Euler class formula
\[
\langle e(\xi),\Sigma\rangle
=\chi(R_+(\Gamma))-\chi(R_-(\Gamma))
\]
for a $\xi$-convex surface $\Sigma$ with dividing set $\Gamma$.  Since
each component $\Sigma$ of $\bdry M$ is homologically trivial, the
claim follows.  (The $\pi_0$ surjectivity holds because each component
of $\partial M \setminus \Gamma$ is an exact symplectic manifold, see
Section~\ref{subsection: sutured contact manifolds}.)  Conversely, if
$(M,\Gamma)$ is a balanced sutured 3-manifold, then there is a contact
structure $\xi$ so that $\bdry M$ is convex with dividing set
$\Gamma$.  (Moreover, according to \cite{HKM2}, there is a tight (or
universally tight) $\xi$ with convex boundary and dividing set
$\Gamma$ on $\bdry M$ if and only if $(M,\Gamma)$ is a {\em taut}
sutured manifold, which means roughly that $R_\pm(\Gamma)$ is
incompressible and genus-minimizing in its homology class in
$H_2(M,\Gamma)$.)  In this paper we will assume without further
mention that our sutured $3$-manifolds are balanced.

If $M$ is closed, it is conjectured that ECH is isomorphic to Heegaard
Floer homology, namely $ECH(M,\xi,A)\simeq
HF^+(-M,\frak{s}_\xi+\operatorname{PD}(A))$ as relatively graded
$\Z$-modules.  Extending this to the sutured case, we conjecture the
following, which is a strengthening of Conjecture~\ref{conj for ECH}:

\begin{conj}
\label{conj:SECH}
If $(M,\Gamma,\xi)$ is a sutured contact 3-manifold, then
\[
ECH(M,\Gamma,\xi,A)
\simeq
SFH(-M,-\Gamma,\mathfrak{s}_\xi + \operatorname{PD}(A))
\]
as relatively graded $\F$-modules, where $\mathfrak{s}_\xi$ denotes
the relative Spin$^c$-structure determined by $\xi$.
\end{conj}

Calculations due to Golovko~\cite{Go1,Go2} confirm this conjecture in
some examples, e.g., when $M=S^1\times D^2$ and $\Gamma$ is arbitrary, for a 
universally tight contact structure.

In the closed case, it is further conjectured that the isomorphism
between $ECH$ and $HF^+$ intertwines the $U$-maps on both sides.
Assuming this conjecture, we can confirm Conjecture~\ref{conj:SECH}
for the sutured contact 3-manifold $M(1)$, where $M$ is closed, as
follows.  On the Heegaard Floer side, the map $U\colon HF^+(M)\to
HF^+(M)$ fits into an exact triangle with $\widehat{HF}(M)$ in the
third position.  To obtain an analogue of this on the ECH side, define
$\widehat{ECH}(M)$ to be the homology of the mapping cone of the
$U$-map on the ECH chain complex.  We then have the following analogue
of \eqref{eqn:SFHM(1)}:

\begin{thm} \label{thm: mapping cone}
If $M$ is a closed oriented 3-manifold, then $ECH(M(1))$
is independent of choices (as a relatively graded $\F$-module), and
\[
ECH(M(1))\simeq \widehat{ECH}(M).
\]
\end{thm}
Arguments in Section 8.4 show that $ECH(M(1))$ depends only on the contact structure.  
The rest of Theorem 1.6 will be proved in the sequel \cite{CGHH2}

We also have some evidence for Conjecture~\ref{conj:SECH} for the
sutured manifold $M(K)$, where $K$ is a nullhomologous knot in a
closed oriented 3-manifold $M$.  Namely, in Section~\ref{section: knot
  filtration} we define a filtration on the chain complex whose
homology gives $ECH(M(1))$; the associated graded complex gives
$ECH(M(K))$.  This is analogous to the Heegaard Floer story, where the
knot Floer homology, (identified with $SFH(M(K))$), is the homology of the
associated graded complex for a filtration on the chain complex
computing $\widehat{HF}(M)$, (identified with $SFH(M(1))$).

\subsection{The simplest sutured contact manifold}

Let $(W,\beta)$ be a Liouville manifold. (See
Section~\ref{subsection: Liouville manifolds} for a definition and
discussions.) Then the simplest contact manifold with convex/sutured
boundary is the {\em product sutured contact manifold}
\[
(M,\Gamma,\alpha)=(W\times[-1,1], \bdry W\times\{0\},
\operatorname{Ker}(dt+\beta)),
\]
where $t$ denotes the $[-1,1]$ coordinate on $W\times[-1,1]$.

\begin{lemma} \label{lemma: product sutured manifold}
Suppose $(M,\Gamma,\xi)$ is a product sutured manifold. If
$\alpha=dt+\beta$ is the $[-1,1]$-invariant contact form for $\xi$ as
above, then
\begin{enumerate}
\item $HC(M,\Gamma,\xi)=\Q;$
\item $ECH(M,\Gamma,\alpha,J)=\F$, if $\dim{M}=3$.
\end{enumerate}
\end{lemma}

\begin{proof}
  The Reeb vector field of $\alpha$ is $R_\alpha=\bdry_t$, which has
  no closed orbits. The algebra $HC(M,\Gamma,\alpha)=\Q$ is generated
  by the unit $1$, and the vector space $ECH(M,\Gamma,\alpha,J)=\F$ is
  generated by the empty set.
\end{proof}

\subsection{Gluing theorems} $\mbox{}$

\subsubsection{Connected sums}
The simplest gluing result describes the behavior of contact homology and
ECH under connected sum. Given a $(2n+1)$-dimensional closed contact
manifold $(M,\xi)$, let us write
$\widehat{HC}(M,\xi)=HC(M-B^{2n+1},\Gamma=S^{2n-1},\xi|_{M-B^{2n+1}})$,
where $(B^{2n+1},\Gamma=S^{2n-1},\xi)$ is the standard Darboux ball
with convex boundary. Then:

\begin{thm}
\label{thm: connected sum} Let $(M_1,\xi_1)$ and $(M_2,\xi_2)$ be
$(2n+1)$-dimensional closed contact manifolds. If $(M_1\#
M_2,\xi_1\#\xi_2)$ is the contact manifold obtained by removing
standard Darboux balls from each $(M_i,\xi_i)$ and gluing, then:
\begin{enumerate}
\item $\widehat{HC}(M_1\#
M_2,\xi_1\#\xi_2)=\widehat{HC}(M_1,\xi_1)\otimes
\widehat{HC}(M_2,\xi_2)$.
\item 
If $\dim M_1=\dim M_2=3$ and we take ECH with coefficients in a field, then
\[
\widehat{ECH}(M_1\#
M_2,\xi_1\#\xi_2)=\widehat{ECH}(M_1,\xi_1)\otimes
\widehat{ECH}(M_2,\xi_2).
\]
\end{enumerate}
\end{thm}

The proof of Theorem~\ref{thm: connected sum} is given in
Section~\ref{subsection: proof of connected sum theorem}. We remark
that, in Theorem~\ref{thm: connected sum}(2), we have a tensor product
of homologies since the ground ring is the field $\F$.  With $\Z$
coefficients one would need to modify the right hand side according to
the K\"unneth formula for the homology of a tensor product of chain
complexes.  Note also that Theorem~\ref{thm: connected sum}(b) is
consistent with the conjectural equivalence of ECH and Heegaard Floer
homology (and their respective $U$ maps), because the analogous
property holds for $\widehat{HF}$.

Before stating the next two theorems we need to make the following:

\s\n {\bf Disclaimers.}  Theorems~\ref{thm: sutured gluing}(2) and
Theorem~\ref{thm: convex gluing}(2) for ECH presuppose part of
Conjecture~\ref{conj for ECH}, namely that sutured ECH depends only on
the contact structure and not on the contact form or almost complex
structure.  They also assume a slightly stronger conjecture, namely
that a suitable ``exact symplectic cobordism'' between sutured contact
3-manifolds induces a map on sutured ECH satisfying certain basic
properties, see Section~\ref{subsection: ECH case} for details.
Analogous maps on ECH induced by exact symplectic cobordisms between
{\em closed\/} contact 3-manifolds are constructed by Hutchings and
Taubes~\cite{HT3}, using Seiberg-Witten theory.

\subsubsection{Sutured manifold gluing}
Let $(M',\Gamma',\xi')$ be a sutured contact manifold.  Suppose
there exist codimension zero Liouville submanifolds $P_+\subset
R_+(\Gamma')$ and $P_-\subset R_-(\Gamma')$ which are
symplectomorphic with respect to $d\alpha$, where $\alpha$ is a
contact $1$-form for $\xi'$, and the symplectomorphism takes
$\alpha|_{P_+}$ to $\alpha|_{P_-}$.  Then we can glue $P_+$ and
$P_-$ to obtain a new sutured contact manifold $(M,\Gamma,\xi)$ with
a properly embedded surface $P$ which is transverse to the Reeb flow.
Details of this {\em sutured manifold gluing} --- the inverse
procedure of a {\em sutured manifold decomposition}, as defined by
Gabai~\cite{Ga} in dimension $3$ --- will be given in
Section~\ref{subsection: gluing}. We then have the following:

\begin{thm} \label{thm: sutured gluing}
If $(M,\Gamma,\xi)$ is obtained from performing a sutured manifold
gluing on $(M',\Gamma',\xi')$, then there are canonical injections:
\begin{enumerate}
\item $\Phi\colon HC(M',\Gamma',\xi')\hookrightarrow
HC(M,\Gamma,\xi);$
\item $\Phi\colon ECH(M',\Gamma',\xi')\hookrightarrow
ECH(M,\Gamma,\xi)$, when $\dim{M}=3$.
\end{enumerate}
Moreover, the map (1) is a $\Q$-algebra homomorphism. In both cases the image 
coincides with the subgroup of (E)CH generated by Reeb orbits which do not intersect $P$.
\end{thm}

Theorem~\ref{thm: sutured gluing} is analogous to a theorem of
Juh\'asz in the context of sutured Floer homology~\cite{Ju1,Ju2},
namely that there is an injection
$$\Phi\colon SFH(M',\Gamma')\hookrightarrow SFH(M,\Gamma)$$ of sutured Floer
homology modules. Its proof will be given in Section~\ref{subsection:
  proof of connected sum theorem}. 

\subsubsection{Convex gluing}
A more general type of gluing is that of gluing along a
closed convex submanifold.  Postponing the precise procedure for
gluing along a convex submanifold $S$ until Section~\ref{subsection:
gluing along convex submanifolds}, we have the following results:

\begin{thm} \label{thm: convex gluing}
If $(M,\Gamma,\xi)$ is obtained from $(M',\Gamma',\xi')$ by gluing
along a closed convex submanifold $S$, then there are canonical maps:
\begin{enumerate}
\item
$\Phi\colon HC(M',\Gamma',\xi')\rightarrow HC(M,\Gamma,\xi);$
\item
$\Phi\colon ECH(M',\Gamma',\xi')\rightarrow ECH(M,\Gamma,\xi)$, when
$\dim{M}=3$.
\end{enumerate}
Moreover, the map (1) is a $\Q$-algebra homomorphism.
\end{thm}

The proof of Theorem~\ref{thm: convex gluing} will be given in
Section~\ref{section: convex gluing}. Theorem~\ref{thm: convex
gluing} is analogous to a theorem of Honda-Kazez-Mati\'c~\cite{HKM}
for sutured Floer homology.

Unlike the case of a sutured manifold gluing, the convex gluing does
not necessarily give an injection of the corresponding contact
homology algebras.  However, we still have the following:

\begin{cor}
If $HC(M,\Gamma,\xi)\not=0$, then $HC(M',\Gamma',\xi')\not=0$.
\end{cor}

\begin{proof}
This is due to the fact that the gluing map
$$\Phi\colon  HC(M',\Gamma',\xi')\to HC(M,\Gamma,\xi)$$
is a $\Q$-algebra homomorphism.
\end{proof}

For example, if $M$ is closed and if $L$ is a Legendrian submanifold of
$M$, then $M$ is obtained by gluing along the convex submanifold
$\partial N(L)$.  Thus we obtain:

\begin{cor}\label{cor: nonvanishing Legendrian invariant}
Let $L$ be a closed Legendrian submanifold of a closed contact
manifold $(M,\xi)$.  If $HC(M,\xi)\not=0$, then $HC(M,\xi,L)\not=0$.
\end{cor}

In contrast to Corollary~\ref{cor: nonvanishing Legendrian invariant},
the Legendrian contact homology --- due to Chekanov~\cite{Ch} and
Eliashberg in dimension three, and Ekholm-Etnyre-Sullivan~\cite{EES}
in higher dimensions --- of a stabilized Legendrian submanifold always
vanishes.  On the other hand, let $\mathcal{A}$ be the contact
homology differential graded algebra (DGA) for some choice of contact
form $\alpha$ for $(M,\xi)$, almost complex structure $J$, and
abstract perturbation.  If $\mathcal{A}$ admits an {\em augmentation},
i.e., a chain map $\mathcal{A}\to \Q$ with the trivial differential
for $\Q$, for example if $(M,\xi)$ has an exact symplectic filling, then
$HC(M,\xi)\not=0$.

\s In a sequel, we plan to prove gluing theorems for
contact homology and embedded contact homology for the initial step
in a sutured manifold hierarchy.

\s\n {\em Outline of the paper.} Sections~\ref{section:
sutured}--\ref{section: operations on sutured contact manifolds}
present the basic material on sutured contact manifolds. In
Section~\ref{section: sutured} we introduce Liouville manifolds,
convex submanifolds, and sutured contact manifolds, and in
Section~\ref{section: almost complex} we introduce almost complex
structures which are ``tailored'' to sutured contact manifolds.
Section~\ref{section: operations on sutured contact manifolds}
collects the various operations that can be done with sutured
contact manifolds --- in particular we discuss switching between the
sutured and convex boundary conditions, and explain the sutured
manifold gluing and convex gluing procedures. Then in
Section~\ref{section: compactness results} we prove the necessary
compactness results for holomorphic curves in completions of sutured
contact manifolds. In Section~\ref{section: defn of sutured CH} we
define sutured contact homology and sutured ECH and prove
Theorem~\ref{thm: well-definition}. Section~\ref{section:
variations} is devoted to the various invariants that can be defined
via sutured contact homology: the ``hat'' versions of contact
homology and ECH, Legendrian knot invariants, and a transverse knot
filtration.  Finally, after some preliminary considerations on
neck-stretching in Sections~\ref{section: warm-up 1}
and~\ref{section: warm-up 2}, we prove Theorem~\ref{thm: sutured
gluing} in Section~\ref{section: proof} and Theorem~\ref{thm: convex
gluing} in Section~\ref{section: convex gluing}.

\section{Sutured contact manifolds}
\label{section: sutured}

In this paper, when we refer to a $(2n+1)$-dimensional contact
manifold $(M,\xi)$, it is assumed that the ambient manifold $M$ is
oriented, and the contact structure $\xi$ is cooriented by a global
$1$-form $\alpha$ which is positive, i.e., satisfies $\alpha\wedge
(d\alpha)^n>0$.

\subsection{Liouville manifolds}
\label{subsection: Liouville manifolds}

\begin{defn}
A {\em Liouville manifold} (often also called a {\em Liouville
domain}) is a pair $(W,\beta )$ consisting of a compact, oriented
$2n$-dimensional manifold $W$ with boundary and a $1$-form $\beta$
on $W$, where $\omega =d\beta$ is a positive symplectic form on $W$
and the {\it Liouville vector field} $Y$ given by $\imath_Y \omega
=\beta$ is positively transverse to $\partial W$ (i.e., exits from
$W$ along $\bdry W$). It follows that the $1$-form $\beta_0 =\beta
\vert_{\partial W}$ (this notation means $\beta$ pulled back to
$\bdry W$) is a positive contact form on $\partial W$, whose kernel we
denote by $\zeta$.
\end{defn}

There is a neighborhood $N(\bdry W)$ of $\bdry W$ which can be
written as $(-\varepsilon,0]\times \bdry W$, with coordinates
$(\tau,x)$, where $Y=\bdry_\tau$, $\beta=e^\tau\beta_0$, and $\bdry
W=\{0\}\times \bdry W$. In other words, $(N(\bdry W),d\beta)$ is
locally symplectomorphic to the symplectization of $\beta_0$, with
$Y=\bdry_\tau$.

We briefly give a proof of this fact: Since $Y$ is transverse to
$\bdry W$, we take $\bdry W=\{0\}\times \bdry W$ and $Y=\bdry_\tau$.
Then we can write $\beta=\beta_\tau+fd\tau$, where
$\beta_\tau=\beta|_{\{\tau\}\times \bdry W}$ does not contain
any $d\tau$-term. Then $d\beta= d_x \beta_\tau + d\tau\wedge
{d\beta_\tau\over d\tau}+d_xf\wedge d\tau$, where $d_x$ means $d$ in
the $\bdry W$-direction. The Liouville condition $\imath_Y d\beta=\beta$
implies that ${d\beta_\tau\over d\tau}-d_xf= \beta_\tau+fd\tau$.
Hence $f=0$ and ${d\beta_\tau\over d\tau}=\beta_\tau$, implying
$\beta_\tau=e^\tau\beta_0$.

We write $(\widehat W,\widehat\beta)$ to denote the completion of
$(W,\beta)$, obtained by attaching the positive symplectization
$([0,\infty)\times \bdry W,e^\tau\beta_0)$.

Two Liouville $1$-forms $\beta^0$ and $\beta^1$ on $W$ are {\em
homotopic} if there is a $1$-parameter family of Liouville $1$-forms
$\beta^t$, $t\in[0,1]$, such that the corresponding Liouville vector
field $Y^t$ on $N(\bdry W)= (-\varepsilon,0]\times \bdry W$ is
$\bdry_\tau$.  We can then complete the homotopy $\beta^t$ to
$\widehat W$ by setting $\widehat\beta^t= e^\tau\beta^t_0$ on
$[0,\infty)\times \bdry W$, where $\beta^t_0=\beta^t|_{\bdry W}$.

\subsection{Convex submanifolds}
\label{subsection: convex submanifolds}

Let $(M,\xi)$ be a $(2n+1)$-dimensional contact manifold. Following
Giroux~\cite{Gi1}, we say that a closed, oriented $2n$-dimensional
submanifold $\Sigma$ of $M$ is {\it $\xi$-convex} if there is a
contact vector field $X$ transverse to $\Sigma$. (Recall that a
contact vector field is generated by a contact Hamiltonian function.
Hence any contact vector field $X$ which is defined in a neighborhood
of $\Sigma$ can be extended to a contact vector field on all of $M$,
and thus convexity is a local condition.) 
Given $X$ as above, one defines the {\em dividing set} $\Gamma$ to be
$\{x\in\Sigma~|~X(x)\in \xi(x)\}$.

To understand the dividing set more explicitly, let
$N(\Sigma)=[-\varepsilon,\varepsilon]\times\Sigma$ be a
neighborhood\footnote{In this paper, a ``neighborhood'' is not
  necessarily an open neighborhood.} of $\Sigma=\{0\}\times\Sigma$,
such that $X=\bdry_t$, where $t$ denotes the
$[-\varepsilon,\varepsilon]$ coordinate.  By changing the sign of $X$
if necessary, we may assume that $\bdry_t$ gives the normal
orientation of $\Sigma$.  We can now find a $1$-form $\alpha$ for
$\xi$ which in $N(\Sigma)$ is given by $\alpha=fdt+\beta$, where $f$
and $\beta$ do not depend on $t$ and $\beta$ has no $dt$-term.  The
dividing set is then $\Gamma=\{f=0\}$.  Since $\alpha$ is a contact
form,
\begin{equation}\label{eqn: contact}
\alpha\wedge (d\alpha)^{n}=
fdt(d\beta)^n+ndfdt\beta(d\beta)^{n-1}>0.
\end{equation}
It follows that (i) $df\not=0$ along $\Gamma$, and
hence $\Gamma$ is a codimension $1$ submanifold of $\Sigma$, and
(ii) $\beta$ is a contact form on $\Gamma$. In particular, (iii)
$\xi=\ker \alpha$ is transverse to $\Gamma$. The dividing set
$\Gamma$ is not necessarily connected.

\begin{lemma} \label{lemma: two defns of convex}
A closed, oriented, codimension one submanifold $\Sigma\subset
(M,\xi)$ is $\xi$-convex if and only if there is an oriented,
codimension one submanifold $\Gamma$ of $\Sigma$ and a (cooriented)
contact form $\alpha$ for $\xi$ such that:
\begin{enumerate}
\item[(A)] $\Gamma$ decomposes $\Sigma$ into alternating positive
and negative open regions $R_\pm(\Gamma)$ so that $(R_+
(\Gamma),d\alpha|_{R_+ (\Gamma)})$ and
$(R_-(\Gamma),d\alpha|_{R_-(\Gamma)} )$, endowed with the
orientation of $\Sigma$ on $R_+ (\Gamma )$ and its opposite on $R_-
(\Gamma )$,  are positive symplectic manifolds;
\item[(B)]  the form $\alpha \vert_\Gamma$ is a positive contact form
on $\Gamma$ for the boundary orientation of $R_+ (\Gamma )$.
\end{enumerate}
\end{lemma}

A contact form $\alpha$ satisfying (A) and (B) above is said to be
{\it adapted} to $(\Sigma,\Gamma )$.   When $M$ has dimension three,
$\Gamma$ is an oriented multicurve on the surface $\Sigma$ which is
positively transverse to $\xi$. 

\begin{rmk} \label{rmk: transverse Reeb vector field}
Let $R_\alpha$ be the Reeb vector field associated with $\alpha$.
The condition that $d\alpha $ be symplectic on $R_\pm (\Gamma)$ is
equivalent to the condition that $R_\alpha$ be positively transverse
to $\Sigma$ along $R_+ (\Gamma)$ and negatively transverse to
$\Sigma$ along $R_- (\Gamma)$.
\end{rmk}

\begin{rmk}
  If $(\Sigma ,\Gamma)$ is a convex hypersurface of $(M,\xi )$, then
  the proof of Lemma~\ref{lemma: two defns of convex} shows that the
  closures $\overline{R_\pm (\Gamma)}$ are Liouville manifolds with a
  Liouville form obtained from the restriction of an adapted contact
  form by a slightly modification near $\Gamma$.  Also, one can choose
  an adapted contact form $\alpha$ so that $(\Sigma,d\alpha|_\Sigma)$
  is a {\em folded symplectic manifold\/}, as defined in \cite{CGW}.
\end{rmk}

The proof of Lemma~\ref{lemma: two defns of convex} uses the following
notion: Given a codimension one submanifold $\Sigma$ of $(M,\xi)$, the
{\em characteristic line field} $L$ is the singular line field in
$\zeta=\xi\cap T\Sigma$ such that $\imath_L(d\alpha|_\zeta)=0$ for any
contact form $\alpha$ for $\xi$. The line field $L$ is singular where
$\xi=T\Sigma$.

\begin{proof}[Proof of Lemma~\ref{lemma: two defns of convex}]
  ($\Rightarrow$) Suppose $\Sigma$ is a convex submanifold. Let
  $\alpha=fdt+\beta$ be the contact form on
  $N(\Sigma)=[-\varepsilon,\varepsilon]\times\Sigma$ as above. By
  Equation~(\ref{eqn: contact}), $\Gamma$ can be oriented so that
  $\alpha|_\Gamma$ is a positive contact form on $\Gamma$.  With this
  orientation of $\Gamma$, the normal orientation of $\Gamma$ in
  $\Sigma$ is given by the direction in which $f$ is decreasing.  We
  then define $R_+(\Gamma)$ (resp.\ $R_-(\Gamma)$) to be the region
  $\{f>0\}$ (resp.\ $\{f<0\}$).  This proves (B).

In order to prove (A) we further normalize the contact form.  Let
$N(\Gamma)=[-1,1]\times[-\varepsilon,\varepsilon]\times\Gamma$ be a
sufficiently small neighborhood of $\Gamma$ with coordinates
$(\tau,t,x)$ so that $\beta$ is a contact form on all
$\{(\tau,t)\}\times\Gamma$.  Here we take $\bdry_t$ for $N(\Gamma)$
to agree with $\bdry_t$ for $N(\Sigma)$.  By possibly multiplying
$\alpha$ by a positive function, we may assume that $f=1$ for
$\tau\geq {1\over 2}$, $f=-1$ for $\tau\leq -{1\over 2}$, $f$ is
constant outside of $N(\Gamma)$, and $f=f(\tau)$ inside $N(\Gamma)$.
Wherever $f$ is locally constant, $(d\beta)^n$ is $>0$ or $<0$ as
appropriate, by Equation~(\ref{eqn: contact}).

Next, let $L$ be the line field on $N(\Gamma)$ which agrees with the
characteristic line field on each level set $A_{t_0}=\{t=t_0\}$ of
$N(\Gamma)$. Take a $t$-invariant vector field $Y$ that directs $L$ so
that the component of $Y$ in the $\tau$-direction is exactly
$\bdry_\tau$. This is possible since $\xi\pitchfork \Gamma$ and
$d\alpha$ is nondegenerate on $\xi\cap T\Gamma$; hence $L$ must have a
component transverse to $\Gamma$. Flowing along $Y$ gives us a new
coordinate function $x$ on $N(\Gamma)$ so that
$\alpha=fdt+\beta$, where $\beta$ only has $dx$-terms and no
$d\tau$-term, $f=f(\tau)$ as before, and $\alpha$ is
$t$-invariant. Now
$$d\alpha=f'(\tau)d\tau dt+d\tau {d\beta\over d\tau}+d_x\beta,$$
where $d_x$ is the derivative in the $x$-direction.  Since
$\imath_{\bdry_\tau} d\alpha|_\zeta=0$, it follows that ${d\beta\over
d\tau}(v)=0$ for all $v\in \ker \beta|_{T\Sigma}$, or, equivalently,
${d\beta\over d\tau}$ is a function times $\beta$. Hence, on
$N(\Gamma)$, we can write
\begin{equation}
\label{eqn: alpha}
\alpha=f(\tau)dt+g(\tau,x)\beta_0,
\end{equation}
where $\beta_0=\beta|_{\{(0,0)\}\times\Gamma}$ and $g>0$.  Also, the
contact condition implies that ${\bdry g\over \bdry \tau}<0$ when
$f=1$ and ${\bdry g\over \bdry \tau}>0$ when $f=-1$.

Finally, let $h$ be a positive $t$-invariant function on $N(\Gamma)$
so that:
\begin{itemize}
\item[(i)] $h=g$ for $\tau\geq {1\over 2}$ and $\tau\leq -{1\over 2}$;
\item[(ii)] ${\bdry h\over \bdry \tau}<0$ for $\tau>0$; and
\item[(iii)] ${\bdry h\over \bdry \tau}>0$ for $\tau<0$.
\end{itemize}
We claim now that condition (A) is fulfilled by a contact form that
agrees with $(h/g)\alpha$ on $N(\Sigma)$.  We need to check that
$d((h/g)\alpha|_{\Sigma})$ is a positive symplectic form on
$R_+(\Gamma)$ and a negative symplectic form on $R_-(\Gamma)$.  On the
complement of $N(\Gamma)$, this follows from equation \eqref{eqn:
  contact} since $f$ is constant there. On $N(\Gamma)$, we have $(h/
g)\alpha|_{\Sigma}=h\beta_0$, and $d(h\beta_0)$ is symplectic on each
of $R_\pm(\Gamma)$ by (ii) and (iii).

\s\n ($\Leftarrow$) Suppose now that there is a contact $1$-form
$\alpha$ which is adapted to $(\Sigma,\Gamma)$. Let
$\beta=\alpha|_\Sigma$. We first normalize $\beta$ on $N(\Gamma)\cap
\Sigma=\{t=0,-1\leq \tau\leq 1\}$: Let $X$ be the characteristic
vector field on $N(\Gamma)\cap\Sigma$ so that its $\tau$-component
is $\bdry_\tau$. Flowing along $X$ (starting at $\tau=0$) gives us
new coordinates $(\tau,x)$ so that $\beta(\tau,x)=g(\tau,x)\beta_0$,
where $\beta_0=\beta|_{\tau=0}$ and $g$ is a positive function.
Moreover, since $d\beta>0$ is a positive symplectic form for
$\tau>0$, it follows that ${\bdry g\over \bdry \tau}<0$ on $\tau>0$;
similarly, ${\bdry g \over \bdry\tau}>0$ on $\tau<0$.

Next we construct a $1$-form $\widetilde\alpha$ on $N(\Sigma)$ of
the form:
$$\widetilde\alpha(\tau,t,x)=\widetilde f dt+ \widetilde\beta.$$
where $\widetilde f$ and $\widetilde\beta$ do not depend on $t$. The
function $\widetilde f\colon \Sigma\to\R$ is constant outside of
$N(\Gamma)$ and can be written as $\widetilde f(\tau)$ on
$N(\Gamma)$ so that $\widetilde{f}(\tau)=1$ for $\tau\geq {1\over
2}$, $\widetilde{f}(\tau)=-1$ for $\tau\leq -{1\over 2}$,
$\widetilde{f}(0)=0$, and $\widetilde{f}'(\tau)>0$ for $-{1\over 2}<
\tau < {1\over 2}$. The $1$-form $\widetilde\beta$ equals $\beta$
outside of $N(\Gamma)$ and equals $\widetilde{g}\beta_0$ on
$N(\Gamma)$, where $\widetilde{g}(\tau,x)=g(\tau,x)$ near $\tau=-1,
1$, $\widetilde{g}(\tau,x)$ only depends on $\tau$ for $-{1\over
2}\leq \tau \leq {1\over 2}$,  $\widetilde{g}>0$ for all $\tau$,
${\bdry\widetilde{g}\over \bdry \tau}<0$ on $\tau>0$, and
${\bdry\widetilde{g}\over \bdry\tau}>0$ on $\tau<0$.

The $1$-form $\widetilde\alpha$ is clearly contact outside of
$N(\Gamma)$. Inside $N(\Gamma)$ we compute that:
\begin{eqnarray}\label{eqn: wedging}
\widetilde\alpha\wedge (d\widetilde\alpha)^n&=& n
\left({\bdry\widetilde{f}\over
\bdry\tau}\widetilde{g}-\widetilde{f}{\bdry \widetilde{g}\over \bdry
\tau}\right) \widetilde{g}^{n-1}d\tau dt \beta_0 (d\beta_0)^{n-1}
>0.
\end{eqnarray}

Since $\alpha, \widetilde\alpha$ pull back to $1$-forms $\beta$,
$\widetilde\beta$ which differ by a conformal factor on $\Sigma$,
there is a local diffeomorphism which fixes $\Sigma$ and sends
$\ker \alpha$ to $\ker\widetilde\alpha$. Since $\Sigma$ is clearly
convex with respect to $\widetilde\alpha$, the same holds for
$\alpha$.
\end{proof}

The following is a corollary of the above proof:

\begin{cor} \label{cor: cor lemma two definitions of convex}
The contact structure in a neighborhood of a convex submanifold
$\Sigma$ can be normalized so that it is given by a contact form
$\alpha_0=f dt +\beta$, where $f$ and $\beta$ do not depend on $t$,
$f=1$ on $R_+(\Gamma)-N(\Gamma)$, $f=-1$ on $R_-(\Gamma)-N(\Gamma)$,
$f=f(\tau)$ on $N(\Gamma)$, the zero set of $f$ is $\tau=0$,
$\beta=g(\tau)\beta_0$ on $N(\Gamma)$, $g(\tau)>0$, $\beta_0$ is a
contact form on $\Gamma$, and ${\bdry f\over \bdry\tau} g-f{\bdry
g\over \bdry \tau}>0$.
\end{cor}

\begin{example}
Let $(K,\theta)$ be a supporting open book for a closed $(M,\xi)$
and let $\alpha$ be a contact form for $\xi$ adapted to $(K,\theta)$
(as in Giroux~\cite{Gi2}). Let $\Sigma$ be the submanifold of $M$
which is the union of (closures of) two pages of the open book that
match up smoothly.  Then $\Sigma$ is $\xi$-convex with dividing set
$K$ and adapted form $\alpha$.
\end{example}

\subsection{Sutured contact manifolds}
\label{subsection: sutured contact manifolds}

\begin{defn}
A compact oriented manifold $M$ of dimension $m$ with boundary
and corners is a {\it sutured manifold} if it comes with an
oriented, not necessarily connected submanifold $\Gamma\subset \bdry
M$ of dimension $m-2$ (called the {\em suture}), together with a
neighborhood $U(\Gamma) =[-1,0]\times [-1,1] \times \Gamma$ of
$\Gamma=\{(0,0)\}\times \Gamma$ in $M$, with coordinates $(\tau
,t)\in [-1,0]\times [-1,1]$, such that the following holds:
\begin{itemize}
\item $U\cap \partial M =(\{0 \} \times [-1,1]\times \Gamma) \cup
([-1,0]\times \{ -1\} \times \Gamma) \cup ([-1,0]\times \{ 1\}
\times \Gamma$);
\item $\partial M -(\{ 0\} \times (-1,1) \times \Gamma)$ is the disjoint
union of two submanifolds which we call $R_- (\Gamma )$ and $R_+
(\Gamma )$,\footnote{At the risk of some confusion, we will use this
definition of $R_\pm(\Gamma)$ when we view $(M,\Gamma)$ as a sutured
manifold, and the definition of $R_\pm(\Gamma)$ given in
Section~\ref{subsection: convex submanifolds} when we think of
$\bdry M$ as being smooth.} where the orientation of $\bdry M$
agrees with that of $R_+ (\Gamma)$ and is opposite that of $R_-
(\Gamma )$, and the orientation of $\Gamma$ agrees with the boundary
orientation of $R_\pm(\Gamma)$.
\item The corners of $M$ are precisely $\{ 0\} \times \{ \pm 1\} \times
\Gamma$.
\end{itemize}
\end{defn}

The notion of a sutured manifold was introduced by Gabai
in~\cite{Ga} for $3$-manifolds.
The definition above is slightly different from the usual one; in
particular the neighborhoods $U(\Gamma)$ do not appear in Gabai's
definition.

By analogy with the theory of branched surfaces, the submanifold
$\partial_h M =R_+ (\Gamma )\cup R_- (\Gamma )$ is often called the
{\it horizontal boundary} and $\partial_v M =\{ 0\} \times [ -1,1]
\times \Gamma$ the {\it vertical} boundary of $M$.\footnote{Strictly
  speaking, the orientation of $U(\Gamma)$ is that of the product
  $[-1,1]\times[-1,0]\times\Gamma$.  However we write the first two
  factors in the opposite order because we want to visualize $[-1,0]$ as the
  horizontal direction and $[-1,1]$ as the vertical direction.}

\begin{defn}
\label{def:SCM}
Let $(M,\Gamma ,U(\Gamma ))$ be a $(2n+1)$-dimensional sutured
manifold. If $\xi$ is a contact structure on $M$, we say that $(M
,\Gamma ,U(\Gamma ),\xi )$ is a {\em sutured contact manifold} if
$\xi$ is the kernel of a positive contact $1$-form $\alpha$ such
that:
\begin{itemize}
\item $(R_+ (\Gamma ),\beta_+ =\alpha\vert_{R_+ (\Gamma )} )$ and
$(R_- (\Gamma ),\beta_- =\alpha\vert_{R_- (\Gamma )} )$ are
Liouville manifolds;
\item $\alpha =Cdt+\beta$ inside $U(\Gamma )$, where $C$ is a positive
constant and $\beta$ is independent of $t$ and does not have a
$dt$-term;
\item $\partial_\tau =Y_\pm$ inside $U(\Gamma)$, where $Y_\pm$ is the
Liouville vector field for $\beta_\pm$.
\end{itemize}
Such a contact form $\alpha$ is said to be {\it adapted} to $(M,\Gamma
,U(\Gamma ))$.  (This is analogus to, but different from, the notion
of a contact form adapted to a convex submanifold as discussed in
Section~\ref{subsection: convex submanifolds}.)  We sometimes denote
the sutured contact manifold by $(M,\Gamma,U(\Gamma),\alpha)$.
\end{defn}

We note two immediate consequences of the above definition.  First,
the Reeb vector field $R_\alpha$ of $\alpha$ equals
$\frac{1}{C} \partial_t$ on $U(\Gamma )$ and is positively transversal
to all of $R_\pm (\Gamma )$, i.e., enters $M$ along $R_- (\Gamma )$
and exits $M$ along $R_+ (\Gamma )$.  Second, on
$U(\Gamma')=[-1,0]\times[-1,1]\times\Gamma'$, with coordinates
$(\tau,t,x)$, we have $\alpha'=Cdt+e^\tau\beta_0(x)$, where $\beta_0$
is a contact form on $\Gamma'$.

\begin{example}
Let $(W,\beta )$ be a Liouville manifold and let $N(\partial
W)=(-\varepsilon,0]\times \bdry W$ be the neighborhood of $\partial
W$ with coordinates $(\tau,x)$, so that the Liouville vector field
$Y$ equals $\bdry_\tau$. Then the manifold
\[
(W\times [-1,1] , \partial W\times \{ 0\} ,N(\partial W) \times
[-1,1] , dt+\beta )
\]
is a sutured contact
manifold, called a {\it product} sutured contact manifold.
\end{example}

\begin{example}
\label{example: interval fibered extension}
Let $(M',\Gamma',U(\Gamma'),\xi')$ be a $(2n+1)$-dimensional sutured
contact manifold with adapted contact form $\alpha'$.  Let
$\Gamma_0'\subset \Gamma'$ be a union of connected components of
$\Gamma'$.  Also let $(W,\beta)$ be a $2n$-dimensional {\em Liouville
  cobordism\/} from $\bdry_+W$ to $\bdry_-W$. By this we mean that
$\bdry W = \bdry_+W - \bdry_-W$ and $d\beta$ is a symplectic form on
$W$, such that the Liouville vector field $Y$ satisfying $\imath_Y
d\beta=\beta$ points into $W$ along $\bdry_- W$ and out of $W$ along
$\bdry_+W$.  Suppose there is a diffeomorphism
\[
\phi\colon  (\bdry_-W,\beta|_{\bdry_-W})\stackrel\sim\to
(\Gamma_0',\beta_0|_{\Gamma_0'}).
\]
We can then define a new sutured contact
manifold $(M,\Gamma,U(\Gamma),\xi)$, called an {\em
interval-fibered extension} of
$(M',\Gamma',U(\Gamma'),\xi')$, by
\[
M=M'\sqcup (W\times[-1,1])/\sim,
\]
where $(0,t,\phi(y))\sim(y,t)$ for all $y\in \bdry_-W$.  Here
$\xi=\operatorname{Ker}(\alpha)$ where $\alpha$ is
obtained by gluing $\alpha'$ and $Cdt+\beta$.  Also
$\Gamma=(\Gamma'-\Gamma_0')\sqcup (\bdry_+W\times\{0\})$, and
$R_\pm(\Gamma)= R_\pm(\Gamma')\cup (W\times\{\pm 1\}).$
\end{example}

\subsection{Completion of a sutured contact manifold}
\label{sec:completion}

Let $(M,\Gamma ,U(\Gamma ),\xi)$ be a sutured contact manifold with an
adapted contact form $\alpha$.  We now explain how to extend
$(M,\alpha)$ to a ``complete'' noncompact contact manifold
$(M^*,\alpha^*)$.

The Reeb flow of $\alpha$ defines
neighborhoods $[1-\varepsilon,1]\times R_+(\Gamma)$ and
$[-1,-1+\varepsilon]\times R_-(\Gamma)$ of $R_+(\Gamma) =\{1\}\times
R_+ (\Gamma)$ and $R_-(\Gamma)=\{-1\}\times R_-(\Gamma)$ respectively,
in which $\alpha = Cdt+\beta_\pm$, where $t\in [-1,-1+\varepsilon]\cup
[1-\varepsilon,1]$ extends the $t$-coordinate on $U(\Gamma)$.
The first step is to extend $\alpha$ ``vertically'' by gluing
$[1,\infty)\times R_+(\Gamma)$ and $(-\infty,-1]\times R_-(\Gamma)$
with the forms $Cdt+\beta_+$ and $Cdt+\beta_-$ respectively. The
boundary of this new manifold is $\{ 0\} \times \R \times \Gamma$.  In
the neighborhood $[-1,0]\times \R\times\Gamma$ of the boundary with
coordinates $(\tau,t,x)$, we have $\alpha=Cdt+e^\tau\beta_0(x)$ where
$\beta_0$ is a contact form on $\Gamma$.

To complete the construction of $(M^*,\alpha^*)$, we then extend
``horizontally'', similarly to the construction of an interval-fibered
extension, by gluing $[0,\infty)\times \R \times \Gamma $ with the
form $Cdt +e^\tau \beta_0$.

For convenience, we extend the coordinates $(\tau,t)$, which are so
far defined only on the ends of $M^*$, to functions on all of $M^*$ so
that $t(M) \subset [-1,1]$ and $\tau(M) \subset [-1,0]$.  We then
refer to $t> 1$ as the Top (T), to $t<-1$ as the Bottom (B), and to
$\tau>0$ as the Side (S).  Consistently with our notation for the
completion of Liouville manifolds in general, we let
$(\widehat{R_\pm(\Gamma)},\widehat{\beta}_\pm)$ denote the completion
of $(R_\pm(\Gamma),\beta_\pm)$ obtained by extending to (S).

\section{Almost complex structures}
\label{section: almost complex}

\subsection{Adapted and tailored almost complex structures}
\label{subsection: almost complex structure}

Let $(Y,\xi)$ be a contact manifold.  Then an almost complex
structure $J$ on the symplectization $\R\times Y$, with
$\R$-coordinate $s$, is {\em adapted to the symplectization} if the
following hold:
\begin{enumerate}
\item $J$ is $s$-invariant;
\item $J$ takes $\xi$ to itself on each $\{s\}\times Y$;
\item $J$ maps $\bdry_s$ to the Reeb vector field $R_{\alpha}$
associated to a contact $1$-form $\alpha$ for $\xi$;
\item $J|_\xi$ is $d\alpha$-positive, i.e., $d\alpha(v,Jv)>0$ for all
nonzero $v\in\xi$.
\end{enumerate}
We remark that Condition (4) does not depend on the choice of
$\alpha$. If we also want to specify the contact $1$-form $\alpha$,
then we say that $J$ is {\em $\alpha$-adapted}.

Let $(W,\beta)$ be a Liouville manifold and let $\zeta$ be the
contact structure given on $\partial W$ by $\ker \beta_0$, where
$\beta_0=\beta|_{\bdry W}$. Recall the completion
$(\widehat{W},\widehat{\beta})$ of $(W,\beta)$, where
$\widehat{W}=W\cup ([0,\infty)\times \bdry W)$ and
$\widehat\beta|_{[0,\infty)\times\bdry W}= e^\tau\beta_0$.  Here
$\tau$ is the $[0,\infty)$-coordinate.  An almost complex structure
$J_0$ on $\widehat{W}$ is {\em $\widehat\beta$-adapted} if it is:

\begin{enumerate}
\item $\beta_0$-adapted on $[0,\infty)\times \bdry W$;
\item $d\beta$-positive on $W$, i.e.,
$d\beta(v,J_0v)>0$ for all nonzero tangent vectors $v$.
\end{enumerate}

Let $(M,\Gamma ,U(\Gamma ),\xi )$ be a sutured contact manifold,
$\alpha$ an adapted contact form and $(M^* ,\alpha^* )$ its
completion.  We consider the symplectization $(\R \times M^* ,d(e^s
\alpha^* ))$ of $(M^* ,\alpha^* )$, where $s$ is the coordinate on
$\R$. We say that an almost complex structure $J$ on $\R\times M^*$
is {\em tailored to $(M^*,\alpha^*)$} if the following hold:
\begin{enumerate}
\item[(A$_0$)] $J$ is $\alpha^*$-adapted;
\item[(A$_1$)] $J$ is $\bdry_t$-invariant in a neighborhood of
$M^* \setminus int(M)$;
\item[(A$_2$)] The projection of $J$ to $T\widehat{R_\pm(\Gamma)}$
is a $\widehat\beta_\pm$-adapted almost complex structure $J_0$ on
the completion $(\widehat{R_+(\Gamma)}, \widehat\beta_+)\sqcup
(\widehat{R_-(\Gamma)}, \widehat\beta_-)$ of the Liouville manifold
$(R_+(\Gamma), \beta_+)\sqcup (R_-(\Gamma), \beta_-)$.
\end{enumerate}
Note that, by the above conditions, $J_0$ uniquely determines $J$ on
$M^*\setminus int(M)$.  Moreover, the flow of $\bdry_t$ identifies
$J_0 \vert_{\widehat{R_+(\Gamma)}-R_+(\Gamma)}$ and $J_0
\vert_{\widehat{R_-(\Gamma)}-R_-(\Gamma)}$.

\subsection{Integrable complex structures $J_0$ for Stein domains.}

We now discuss the integrable complex structure on a Stein domain,
which, as we will see in Section~\ref{subsection: Stein case}, is
often more convenient for calculations.  The slight drawback is that
the integrable complex structure is usually not adapted to the
symplectization.

Let $(W,J_0)$ be a Stein domain.  Then there exists a Morse function
$f \colon  W\rightarrow \R$ which is strictly plurisubharmonic
and for which $\partial W$ is a regular level set. If $\beta
=-d^{\C}f=-df \circ J_0$, then we claim that $(W,\beta)$ is a
Liouville manifold and that the symplectic form $\omega=d\beta$ is
$J_0$-compatible. Indeed, $\omega$ is symplectic since $\omega
(v,J_0 v)>0$ (i.e., $\omega$ is tamed by $J_0$) for all nonzero
tangent vectors $v$ of $W$ by the strict plurisubharmonicity of $f$.
Moreover, $\omega (\cdot,J_0 \cdot)$ is symmetric by the
integrability of $J_0$: Writing
$(*)=-\omega(X,J_0Y)+\omega(Y,J_0X)$, we compute, using the Cartan
formula, that
\begin{eqnarray*}(*)&=& dd^{\C} f (X,J_0Y)+dd^{\C} f(J_0X,Y)\\ &=&
X(d^{\C}f(J_0Y))-J_0Y(d^{\C} f(X)) -d^{\C} f([X,J_0Y])  \\ &&
+J_0X(d^{\C} f(Y))  -Y(d^{\C} f(J_0X))-d^{\C}
f([J_0X,Y]).\end{eqnarray*} Now, the integrability of $J_0$ is
equivalent to the vanishing of the Nijenhuis tensor, i.e.,
$$J_0[X,J_0Y] +J_0[J_0X,Y]=[J_0X,J_0Y]-[X,Y].$$
Thus $(*)=-X(df(Y))+Y(df(X))+df([X,Y])
+J_0X(df(J_0Y))-J_0Y(df(J_0X))-df([J_0X,J_0Y])= -d^2 f(X,Y)+d^2
f(J_0X,J_0Y)=0$, and we have proved that $\omega(\cdot,J_0\cdot)$ is
symmetric. Now let $\zeta$ be the contact structure on $\bdry W$
given by $\ker \beta|_{\bdry W}$.  If $v \in \zeta$, then $\beta
(J_0 v) =df (v) =0$, and thus $J_0$ fixes $\zeta$. Let
$g(X,Y)=\omega(X,J_0Y)$ be the compatible metric on $W$. Then the
Liouville vector field $X$ satisfying $\imath_X \omega= \beta$ is given
by $X=\nabla f=J_0X_f$, where the gradient $\nabla$ is with respect
to $g$ and $X_f$ is the Hamiltonian vector field of $f$ with respect
to $\omega$. Hence the Liouville vector field $X$ is positively
transverse to $\bdry W$ and $(W,\beta)$ satisfies the conditions of
a Liouville manifold.

When $W$ is a compact surface with nonempty boundary, there is a
complex structure $J_0$ which makes $(W,J_0)$ into a Stein domain.
Thus $W$ has the structure of a Liouville manifold with a compatible
almost complex structure $J_0$.

One subtlety that we address in Subsection~\ref{subsection: interp}
is that, in a neighborhood of $\bdry W$, the integrable $J_0$ is
often slightly different from an almost complex structure $J_0'$
which is $\beta_0=\beta|_{\bdry W}$-adapted. If
$(-\varepsilon,\varepsilon)\times \bdry W$ is a piece of the
symplectization of $\bdry W$ with coordinates $(\tau,x)$ and $\bdry
W= \{\tau=0\}$ so that the Liouville vector field $X=\bdry_\tau$,
then the level sets of $f$ differ slightly from the level sets of
$\tau$. Also, while $J_0'$ can be made to agree with $J_0$ on $\ker
\beta_0$, and $J_0$ maps $\bdry_\tau\mapsto g_0R_{\beta_0}$, the
function $g_0$ is usually not constant. The following is an example
of the above-mentioned issues, which the authors learned from Jian
He.

\begin{example}
Consider $\C^n$ with coordinates $z_i=x_i+\sqrt{-1}~y_i$ and the
standard integrable complex structure $J_0$. Let $M$ be an
ellipsoid in $\C^n$ which is a level set of
$$f={1\over 2} \sum_i (x_i^2 +\lambda_i y_i^2).$$

We compute that
\[
df=\sum_i (x_idx_i+\lambda_i y_idy_i),
\]
\[
\beta=-df\circ J_0=\sum_i (-\lambda_i y_idx_i+ x_idy_i),
\]
\[
\omega=d\beta= \sum_i (1+\lambda_i)dx_idy_i,
\]
\[
X_f=\sum_i{1\over
1+\lambda_i}(-x_i\bdry_{y_i}+\lambda_iy_i\bdry_{x_i}),
\]
where $X_f$
is the Hamiltonian vector field of $f$ with respect to $\omega$, and
$$X=\nabla f= J_0X_f=\sum_i {1\over 1+\lambda_i}(x_i\bdry_{x_i}+\lambda_i
y_i\bdry_{y_i}).$$   Hence, we have
$$df(X)= \sum_i{1\over 1+\lambda_i}(x_i^2+\lambda_i^2y_i^2).$$
It follows that if not all the $\lambda_i$ are the same, then
$df(X)=df(\bdry_\tau)$ is not constant on the level sets of
$f$, and so the level sets of $\tau$ are different from the level sets
of $f$.
\end{example}

\subsection{Interpolation of almost complex structures on
symplectizations.} \label{subsection: interp0}

Let $Y$ be an odd-dimensional manifold and let $\beta_0, \beta_0'$ be
homotopic contact $1$-forms on $Y$, i.e.\ suppose there is a $1$-parameter
family of contact $1$-forms from $\beta_0$ to $\beta_0'$.  Consider
$\R\times Y$ with coordinates $(\tau,x)$. We then have the following
lemma, which is used to prove that the sutured contact homology
algebras are independent of the choice of almost complex structure:

\begin{lemma} \label{lemma: interpol}
There is a constant $C> 0$ and an almost complex structure $J$ on
$\R\times Y$ which is $C\cdot\beta_0'$-adapted for $\tau\geq 1$ and
$\beta_0$-adapted for $\tau\leq 0$, and such that $\tau$ is
plurisubharmonic with respect to $J$, i.e.,
$-dd^\C\tau(v,Jv)=-d(d\tau\circ J)(v,Jv)\geq 0$ for all tangent
vectors $v\in T(\R\times Y)$.
\end{lemma}

\begin{proof}
This is just a modification of the usual proof of the
plurisubharmonicity of $\tau$ with respect to a $\beta_0$-adapted
almost complex structure.

By Gray's theorem, the homotopy from $\beta_0$ to $\beta_0'$ gives
rise to a diffeomorphism isotopic to the identity, which takes
$\beta_0'$ to $f\beta_0$ for some positive function $f$ on $Y$.
Hence, after composing with a diffeomorphism of $\R\times Y$ of type
$(\tau,x)\mapsto (\tau,\phi_\tau(x))$, where
$\phi_\tau\colon Y\stackrel\sim\to Y$ is a diffeomorphism, we may assume
that $\beta_0'=f\beta_0$. We then define a $1$-form
$\beta(\tau,x)=g(\tau,x)\beta_0(x)$ on $\R\times Y$ such that
$g(\tau,x)$ is a smooth function which satisfies the following:
\begin{enumerate}
\item[(i)] $g(\tau,x)=1$ for $\tau\leq 0$;
\item[(ii)] $g(\tau,x)=C\cdot f(x)$ for
$\tau\geq 1$, where $C$ is a constant greater than $\max(1/f)$;
\item[(iii)] ${\bdry g(\tau,x)\over \bdry\tau}\geq 0$.
\end{enumerate}

Let $R_\tau$ be the Reeb vector field for $\beta(\tau)$. Then we
choose $J$ so that $J(\tau,x)$ sends $\ker\beta(\tau)=\ker\beta_0$
to itself and $\bdry_\tau$ to $R_\tau$, and satisfies
$d_Y\beta(\tau)(X,JX)>0$ for all nonzero $X\in \ker\beta(\tau)$, where
$d_Y$ denotes the exterior derivative on $Y$.

We claim that $d\tau\circ J=-\beta$.  Indeed, $d\tau\circ J$ sends
$\ker\beta(\tau)\mapsto 0$, $\bdry_\tau\mapsto 0$, and
$R_\tau\mapsto -1$, agreeing with the evaluation of $-\beta$ on
these tangent vectors.

We now have
\begin{equation} \label{equation: ddctau}
-dd^\C\tau=d\beta=d(g\beta_0)={\bdry g\over \bdry\tau}d\tau\wedge
\beta_0+d_Y(g\beta_0),
\end{equation}
If we write $v\in T(\R\times Y)$ as $X+a\bdry_\tau+bR_\tau$, where
$X\in\ker\beta(\tau)$, then $Jv=JX+aR_\tau-b\bdry_\tau$.  Evaluating
the pair $(v,Jv)$ on the right-hand side of Equation~(\ref{equation:
  ddctau}), we obtain:
\begin{equation}
\label{equation: eq2} g^{-1} {\bdry g\over \bdry
\tau}(a^2+b^2)+d_Y(g\beta_0)(X,JX)\geq 0.
\end{equation}
This proves the plurisubharmonicity of $\tau$.
\end{proof}

By rescaling in the $\tau$-direction we obtain the following:

\begin{cor} \label{cor: plurisub}
There is an almost complex structure $J$ on $\R\times Y$ which is
$\beta_0'$-adapted for sufficiently positive $\tau$ and
$\beta_0$-adapted for sufficiently negative $\tau$, so that some increasing
function $u$ of $\tau$ is $J$-plurisubharmonic.  In particular, no
holomorphic map from a Riemann surface with punctures into
$(\R\times Y,J)$ attains a local maximum in the $\tau$-direction.
\end{cor}

\subsection{Interpolation between the adapted and integrable 
almost complex structures.} \label{subsection: interp}

Let $(W,J_0)$ be a Stein domain with a strictly plurisubharmonic
function $\tau$ and a corresponding Liouville $1$-form $\beta$.
(Unlike our previous notation, $\tau$ now denotes the plurisubharmonic
function and not the coordinate near the boundary given by the
Liouville vector field.)
Without loss of generality, we may assume that $\bdry W=\{\tau=0\}$.
Writing $Y=\bdry W$, let $N(\bdry W)=[-\varepsilon,0]\times Y$ be a
neighborhood of $\bdry W=\{0\}\times Y$ with coordinates $(\tau,x)$.
Extend this to $[-\varepsilon,\infty)\times Y$, also with coordinates
$(\tau,x)$. Write $\beta_\tau=\beta|_{\{\tau\}\times Y}$ and
$\zeta_\tau=\ker \beta_\tau$.

\begin{lemma} \label{lemma: interpol2}
Suppose $\beta'_0$ is a contact $1$-form which is homotopic to
$\beta_0$. On $[-\varepsilon,\infty)\times Y$, there exist an almost
complex structure $J$ and a $J$-plurisubharmonic function $u(\tau)$
such that:
\begin{itemize}
\item[(i)] $J$ is $\beta_0'$-adapted for sufficiently positive $\tau$;
\item[(ii)] $J$ agrees with $J_0$ on $N(\bdry W)$.
\end{itemize}
\end{lemma}

We thank Yasha Eliashberg for suggesting that something like the
above lemma might be true.

\begin{proof}
By applying Corollary~\ref{cor: plurisub} above, we may assume that
$\beta'_0=\beta_0$.

Let us first consider the Liouville $1$-form $\beta=-d\tau\circ J_0$
on $N(\bdry W)$.  By changing the identification of $N(\partial W)$
with $[-\varepsilon,0]\times Y$, we can arrange for the vector field
$\bdry_\tau$ to be parallel to, but not necessarily a constant
multiple of, the Liouville vector field $\nabla \tau$ which satisfies
$\imath_{\nabla\tau}d\beta=\beta$.  It then follows that $\beta$ has
no $d\tau$-terms.  Hence $\beta(\tau,x)=\beta_\tau(x)$.  We also
observe that, if $R_\tau$ is the Reeb vector field for $\beta_\tau$ on
$\{\tau\}\times Y$, then it is parallel to the Hamiltonian vector
field $X_\tau$ for $\tau$, which satisfies $$\imath_{X_\tau}
d\beta=\imath_{X_\tau}(d_Y\beta_\tau+d\tau\wedge
\dot\beta_\tau)=d\tau,$$ where $\dot\beta_\tau={d\beta_\tau\over
  d\tau}$. Moreover, we claim that $J_0(\bdry_\tau)=R_\tau$. Indeed,
since $\nabla\tau$ is parallel to $\bdry_\tau$, $X_\tau$ is parallel
to $R_\tau$, and $J_0(\nabla\tau)=-X_\tau$, we have $J_0(\bdry_\tau)$
is a function times $R_\tau$. The function can be determined from the
equation $\beta_\tau(R_\tau)=-d\tau\circ J_0(R_\tau)=1$.

Next define a smooth function $u\colon [-\varepsilon,0]\to \R$ so that it
satisfies the following:
\begin{itemize}
\item $u(\tau)=\tau$ on $[-\varepsilon,-{\varepsilon\over 2}]$;
\item ${d^2u\over d\tau^2}\geq 0$ on $[-\varepsilon,0]$; and
\item ${d^2u\over d\tau^2}(0)\gg {du\over d\tau}(0)$.
\end{itemize}
The function $u(\tau)$ is $J_0$-plurisubharmonic on $N(\bdry W)$.
This follows from the general fact that the composition of a
plurisubharmonic function with a smooth, increasing, convex function
$u$ from a subset of $\R$ to $\R$ is plurisubharmonic. Here ``convex''
means $u''\geq 0$ at all points in the domain. 
To see this explicitly, if we set $\beta'=-du\circ J_0$, then
\[
\begin{split}
\beta'&=-{du\over d\tau}(d\tau\circ J_0)={du\over d\tau}\beta,\\
d\beta'&={d^2u\over d\tau^2}d\tau\wedge \beta+{du\over
d\tau}d\beta\\
&= {d^2u\over d\tau^2} d\tau\wedge (-d\tau\circ J_0) + {du\over
d\tau}d\beta.
\end{split}
\]
The conditions on $u(\tau)$ then imply that $d\beta'(v,J_0v)>0$ for
all nonzero $v$.

It is useful below to write $g(\tau)={du\over d\tau}$, and to
rewrite the above equation as
\begin{equation} \label{eqn: for d beta prime}
d\beta'={dg\over d\tau} d\tau\wedge (-d\tau\circ J_0) + g (d_Y
\beta_\tau+d\tau\wedge \dot\beta_\tau),
\end{equation}
where $g(\tau)$ satisfies ${dg\over d\tau}\gg g$ near $\tau=0$.

We now extend $\beta'=g\beta$ and $J=J_0$ over $[0,\infty)\times Y$. First
choose $g\colon [0,\infty)\to \R$ so that ${dg\over d\tau}\gg g$ on
$[0,1]$ and ${dg\over d\tau}>0$ elsewhere. We then extend $\beta$ so
that:
\begin{itemize}
\item $\beta(\tau,x)=\beta_\tau$, i.e., $\beta$ has no $d\tau$-term;
\item $\beta_\tau$ are contact forms on $Y$;
\item $\beta_\tau=\beta_0$ for $\tau\geq 1$.
\end{itemize}
(The only reason we cannot set $\beta_\tau=\beta_0$ for all $\tau\geq
0$ is that we require $\beta$ to be smooth.)  Let
$\zeta_\tau=\ker\beta_\tau$ and $R_\tau=R_{\beta_\tau}$. Since $J_0$
maps $\zeta_\tau$ to itself and $\bdry_\tau\mapsto R_\tau$ on
$N(\partial W)$, we can extend $J_0$ to $J$ so that $\zeta_\tau$ is
mapped to itself and $\bdry_\tau\mapsto R_\tau$.

Now let
$u=u(\tau)$ be the extension of $u|_{N(\bdry W)}$ to
$[-\varepsilon,\infty)\times Y$ so that ${du\over d\tau}=g(\tau)$.
To show that $u$ is $J$-plurisubharmonic, first observe that
\[
-du\circ J={du\over d\tau} (-d\tau\circ J) = g \beta =\beta'.
\]

Thus we need to verify the nonnegativity condition
$d\beta'(v,Jv)\geq 0$. Write $v=X+a\bdry_\tau+b R_\tau$, where
$X\in\zeta_\tau$, so that $Jv= JX+aR_\tau-b \bdry_\tau$. Then
Equation~(\ref{eqn: for d beta prime}) gives
\[
\begin{split}
d\beta'(v,Jv)&={dg\over d\tau}(a^2+b^2)+g (d_Y
\beta_\tau(X,JX))\\
& +  g(a\dot\beta_\tau(JX+aR_\tau)+b\dot\beta_\tau(X+bR_\tau)).
\end{split}
\]
The nonnegativity is immediate for $\tau\geq 1$ since
$\dot\beta_\tau=0$. The nonnegativity for $\tau\in[0,1]$ follows
from ${dg\over d\tau}\gg g$ and is based on the inequality
$$K \sum_ix_i^2+ k \sum_i y_i^2\geq \sum_{ij} a_{ij} x_iy_j,$$ where
$k>0$ and $a_{ij}$ are given, and $K\gg 0$ is chosen in response to
$k,a_{ij}$.
\end{proof}

\section{Operations on sutured contact manifolds}
\label{section: operations on sutured contact manifolds}

\subsection{Switching between convex and sutured boundary 
conditions}
\label{subsection: from convex to sutured}

In this subsection we describe how to pass between the convex and
sutured boundary conditions.

When $(M, \Gamma ,U(\Gamma ),\xi )$ is a sutured contact manifold,
it is easy to smooth the corners of $M$ inside $U(\Gamma )
=[-1,0]\times [-1,1] \times \Gamma$, so that the resulting manifold
$M'$ has boundary $\bdry M'$ which is transversal to the Reeb vector
field $R=\frac{1}{C}
\partial_t$ except at $\Gamma=\{(0,0)\}\times \Gamma$.  More
precisely, the portion of $\bdry M'$ for which $t> 0$ (resp.\ $t<0$)
is positively (resp.\ negatively) transverse to $R$. Hence the
slight retract $(M', \Gamma ,\xi \vert_{M'} )$ of $M$ has
$\xi$-convex boundary by Lemma~\ref{lemma: two defns of convex}.

On the other hand, the following lemma explains how to pass from
convex to sutured boundary.

\begin{lemma}\label{l:suture}
Let $(M,\xi )$ be a $(2n+1)$-dimensional contact manifold with
$\xi$-convex boundary $(\partial M, \Gamma )$, and let
$N(\Gamma)\subset M$ be a tubular neighborhood of $\Gamma$. Then
there exists a codimension $0$ sutured contact submanifold $(M',
\Gamma', U(\Gamma'), \xi \vert_{M'} )$ of $M$, together with a
contact form $\alpha$ on $M$, such that $\alpha \vert_{M'}$ is adapted to
$(M', \Gamma', U(\Gamma'), \xi \vert_{M'} )$, $M-M'\subset N(\Gamma)$,
$U(\Gamma')\subset N(\Gamma )$, and $(\Gamma' ,\xi\cap T\Gamma')$ is
isotopic to $(\Gamma ,\xi\cap T\Gamma)$ through $(2n-1)$-dimensional
contact submanifolds of $(M, \xi )$.
\end{lemma}

\begin{proof}
Since $\Sigma=\bdry M$ is $\xi$-convex, there is a neighborhood
$N(\Sigma)=[-\varepsilon,0]\times\Sigma$ of $\Sigma=\{0\}\times
\Sigma$ with first coordinate $t$ and a contact form
$\alpha_0=fdt+\beta$ as given by Corollary~\ref{cor: cor lemma two
definitions of convex}.  In particular, on
$N(\Gamma)=[-1,1]\times[-\varepsilon,0]\times \Gamma$, the form
$\alpha_0$ can be written as
$$\alpha_0=f(\tau)dt+g(\tau)\beta_0=g(\tau)(\beta_0+\widetilde{f}(\tau)dt);$$
we may assume that $\widetilde{f}(\tau)=\tau$ for $-{1\over 4}\leq
\tau\leq {1\over 4}$, ${\bdry g\over \bdry \tau}>0$ for $\tau<0$,
${\bdry g\over \bdry\tau}<0$ for $\tau>0$, and $g(\tau)=g(-\tau)$.
Then
\begin{itemize}
\item[($\pitchfork$)]$R_{\alpha_0}$ is
positively transverse to $\bdry M$ along $R_+(\Gamma)$ and
negatively transverse to $\bdry M$ along $R_-(\Gamma)$.
\end{itemize}

Consider cylindrical coordinates $(r,\theta,x)$ on $N(\Gamma)$ so
that
$$(\tau,t)=(r\cos(\theta),r\sin(\theta))$$ and the portion contained
in $M$ is $\pi\leq \theta\leq 2\pi$. Let
$$U=\{\pi\leq \theta\leq 2\pi, 0\leq r\leq \delta\}\subset
N(\Gamma).$$ Along $t=0$, $-{1\over 4}\leq \tau\leq {1\over 4}$, the
contact forms
\begin{eqnarray*}
\alpha_0&=&g(\tau)(\beta_0+ \widetilde{f}(\tau) dt)\\
\alpha_1&=&g(r) (\beta_0 + r^2 d\theta )
\end{eqnarray*} agree and the interpolation
$\alpha_s=(1-s)\alpha_0+s\alpha_1$ is contact. Hence, by the usual
Moser-Weinstein technique, there is a $1$-parameter family of local
diffeomorphisms $\phi_s$, $s\in[0,1]$, near $\Gamma$ so that
$\phi_0=id$, $\phi_s=id$ along $\Sigma$, and $(\phi_1)_*$ takes
$\xi_{\alpha_0}$ to $\xi_{\alpha_1}$.  In other words, after a
change of coordinates we may write
$$\alpha_0=h_0(r,\theta,x) (\beta_0+r^2d\theta)$$
on $U$, for some positive function $h_0\colon U\to \R$ and sufficiently
small $\delta$.  Note that we {\em have not} modified $\alpha_0$ by
a conformal factor, and $R_{\alpha_0}$ still satisfies
($\pitchfork$).

Now let $h\colon U\to\R$ be any positive function. We claim that the Reeb
vector field $R_\alpha$ for the contact form $\alpha = h (\beta_0 +
r^2 d\theta )$ is positively transverse to the surfaces $\{ \theta
=const \} \subset U-\Gamma$ if and only if $\frac{\partial
h}{\partial r} <0$. Indeed, by plugging $R_\alpha$ into the equation
$\alpha=h(\beta_0+r^2d\theta)$, we obtain
$$\beta_0 (R_\alpha)={1\over h}-r^2 d\theta (R_\alpha).$$
Also, the coefficient of $dr$ in the equation $\imath_{R_\alpha}
d\alpha=0$ gives
$$\frac{\partial h}{\partial r} \beta_0 (R_\alpha) + \left(r^2 \frac{\partial
h}{\partial r} +2r h  \right)d\theta (R_\alpha)=0.$$ Putting the two
identities together, we obtain
$$\frac{\partial h}{\partial r}=-2r h^2 d\theta (R_\alpha)$$ and the
conclusion follows.

Now we take a function $h$ on $U$ with the following properties:
\begin{itemize}
\item $h=h_0$ on $\partial U\cap \{r=\delta\}$;
\item $\frac{\partial h}{\partial r} <0$;
\item $h=\frac{C_0}{r^2}$ when ${\varepsilon\over 2} \leq r\leq \varepsilon$.
(Here $C_0>0$ is a large constant and $\varepsilon>0$ is a small
constant $<\delta$.)
\end{itemize}
If we define $\alpha$ to be $h(\beta_0 + r^2 d\theta )$ on $U$ and
$\alpha_0$ on $M-U$, then the Reeb vector field $R_\alpha$ is
transverse to $R_{\pm} (\Gamma )$.  On ${\varepsilon\over 2}\leq
r\leq \varepsilon$, since $\alpha= {C_0\over r^2}\beta_0+C_0d\theta$
we have $R_\alpha={1\over C_0}\bdry_\theta$. We then take $M' =M-\{
r<{\varepsilon\over 2} \}$, $\Gamma'=\{ r={\varepsilon \over
2},\theta ={3\pi \over 2}\}$ and $U(\Gamma') =M\cap \{
{\varepsilon\over 2} \leq r\leq \varepsilon \}$. The
$\theta$-coordinate becomes the $t$-coordinate on $U(\Gamma')$ and
the contact form $\alpha$ gives this modified manifold $(M',
\Gamma',U(\Gamma'))$ the structure of a sutured contact manifold.

Finally, $\Gamma$ is isotopic to $\Gamma'$ through contact
submanifolds of type $(\Gamma_{a,b},\ker\beta_0)$, where
$\Gamma_{a,b}=\{r=a,\theta=b\}$.
\end{proof}

\subsection{From concave to convex boundary}
\label{subsection: concave to convex}

\begin{defn}
Let $M$ be a compact $(2n+1)$-dimensional manifold with
$3\pi/2$-corners and let $\Gamma\subset \bdry M$ be a
$(2n-1)$-dimensional submanifold. We call $(M,\Gamma,V(\Gamma))$ is a
{\em concave sutured manifold} with {\em suture} $\Gamma$, if
$V(\Gamma)\subset M$ is a neighborhood of
$\Gamma=\{(0,0)\}\times\Gamma$ of the form
$$([-1,1] \times [-2,2]- (0,1] \times (-1,1)) \times \Gamma$$
with coordinates $(\tau,t,x)$, and all the corners of $M$ lie in the
interior of $V(\Gamma)$.
\end{defn}

Let $R_+ (\Gamma)\sqcup R_- (\Gamma ) =\partial M -int (\{ 0\}
\times [-1,1] \times \Gamma )$ be the horizontal boundary and $\{ 0
\} \times [-1,1] \times \Gamma$ be the vertical boundary of $M$.
Here the orientation of $R_+(\Gamma)$ (resp.\ $R_-(\Gamma)$) agrees
with (resp.\ is opposite of) the boundary orientation of $M$, and
the orientation of $\Gamma$ is the boundary orientation of
$R_\pm(\Gamma)$.

\begin{defn}
$(M,\Gamma, V(\Gamma),\xi )$ is a {\it concave sutured contact
manifold} if $\xi$ is contact structure on $M$ and there exists a
contact form $\alpha$ for $\xi$ so that $(R_\pm
(\Gamma),\alpha|_{R_\pm(\Gamma)})$ are Liouville manifolds, $\alpha
=C dt +\beta$ in $V(\Gamma)$, and the Reeb orbits along the vertical
boundary go from $R_+(\Gamma)$ to $R_-(\Gamma)$ (instead of from
$R_-(\Gamma)$ to $R_+(\Gamma)$, which is the case for convex
sutures). Here $C>0$ and $\beta$ is independent of $t$ and has no
$dt$-term.
\end{defn}

\begin{example}
Let $(M,\xi =\ker \alpha )$ be a contact $3$-manifold and let
$S\subset M$ be a compact, oriented surface which is transversal to
the Reeb vector field $R_\alpha$ and whose boundary $\partial S$ is
positively transversal to $\xi$. Now, if $N(S)=S\times
[-\varepsilon,\varepsilon ]$ is a collar neighborhood of $S$ whose
$[-\varepsilon,\varepsilon]$-coordinate $t$ satisfies
$R_\alpha=\bdry_t$, then $M-int(N(S))$ is naturally a concave
sutured contact manifold with respect to the form $\alpha$.  In
particular, $\Gamma =\{0\}\times\partial S$, the vertical boundary
is $\partial S \times [-\varepsilon ,\varepsilon]$, $R_+ (\Gamma ) =
S\times \{ -\varepsilon \}$, and $R_- (\Gamma ) = S\times \{
+\varepsilon \}$.
\end{example}

\begin{example} \label{example: Legendrian2}
Let $(M,\xi)$ be a closed $(2n+1)$-dimensional contact manifold and let
$L\subset M$ be a closed Legendrian submanifold.  By the
Darboux-Weinstein neighborhood theorem, there is a sufficiently
small neighborhood $N(L)$ of $L$ which is contactomorphic to a small
neighborhood of the zero section $\{z=p_1=\dots=p_n=0\}$ in the
$1$-jet space $\R\times T^*L$ with the contact $1$-form $\alpha=
dz+\lambda$, where $\lambda$ is the Liouville form on $T^*L$ which
is locally given by $\sum_i p_idq_i$.  The Reeb vector field is given by
$R_\alpha=\bdry_z$, and we can take the boundary of the tubular
neighborhood of the zero section to be
$\Sigma=\{(z,p,q)~|~z^2+|p|_q^2=\varepsilon^2\}$ after choosing
a Riemannian metric on $L$.
Then $R_\alpha$ is positively transverse to $\Sigma$ (with the
boundary orientation) for $z>0$, negatively transverse to $\Sigma$
for $z<0$, and tangent to $\Sigma$ for $z=0$. The set
$\Gamma=\{(z,p,q)~|~z=0, |p|_q=\varepsilon\}$ is the unit cotangent
bundle of $L$, and is a $(2n-1)$-dimensional contact manifold.  One
can see this for example by observing that the Liouville vector
field $\sum_i p_i\bdry_{p_i}$ for $(T^*L,\lambda)$ is transverse to
$\Gamma$. If we set $N(L)=\{(z,p,q) ~|~ z^2< \varepsilon^2, |p|_q^2
<\varepsilon^2\}$, then $(M-N(L),\xi|_{M-N(L)},\Gamma)$ is a concave
sutured manifold.
\end{example}

\begin{prop} \label{prop: concave to sutured}
Let $\mathcal{M}=(M,\Gamma ,V(\Gamma ),\xi )$ be a concave sutured
contact manifold. Then there is an inclusion of $\mathcal{M}$ into a
convex sutured contact manifold
$\mathcal{M'}=(M',\Gamma',U(\Gamma'),\xi')$, so that the contact
manifold with convex boundary $(M_{sm},\Gamma,\xi)$, obtained by
smoothing the corners of $\mathcal{M}$, is isotopic to the contact
manifold with convex boundary $(M'_{sm},\Gamma',\xi')$, obtained by
smoothing the corners of $\mathcal{M}'$. Here $\Gamma$ and $\Gamma'$
are isotopic contact submanifolds and $M'-M\subset
(0,1]\times(-1,1)\times \Gamma$.
\end{prop}

\begin{proof}
On $V(\Gamma)=([-1,1] \times [-2,2]- (0,1] \times (-1,1)) \times
\Gamma$ the adapted contact form is $\alpha=Cdt+\beta$, where $C$ is
a positive constant. Without loss of generality we can write $\beta
=e^{-\tau} \beta_0$, where $\beta_0$ is a $1$-form on $\Gamma$. (The
minus sign in $e^{-\tau}\beta_0$ is due to the fact that the
Liouville vector field on $R_\pm(\Gamma)$ points in the negative
$\tau$-direction.) We now describe how to extend $\alpha$ to the
product $[0,1] \times [-1,1]\times \Gamma$. To that end, we look for
a form of type $f(\tau ,t)dt +g(\tau,t)\beta_0$, where $f,g
\colon [-1,1]\times [-2,2] \rightarrow \R$, and $f=C$ and $g=e^{-\tau}$
outside of $[0,1] \times [-1,1]$.

Let $g$ be a positive Morse function on $[0,1]\times [-1,1]$, whose
level sets are obtained from perturbing the foliation by intervals
$\{ \tau \} \times [-1,1]$, $\tau \in [0,1]$, by adding a pair of
(canceling) critical points --- a saddle $h$ and a source $e$
--- as in Figure~\ref{modification}.
\begin{figure}[ht]
\begin{overpic}[height=1.5in]{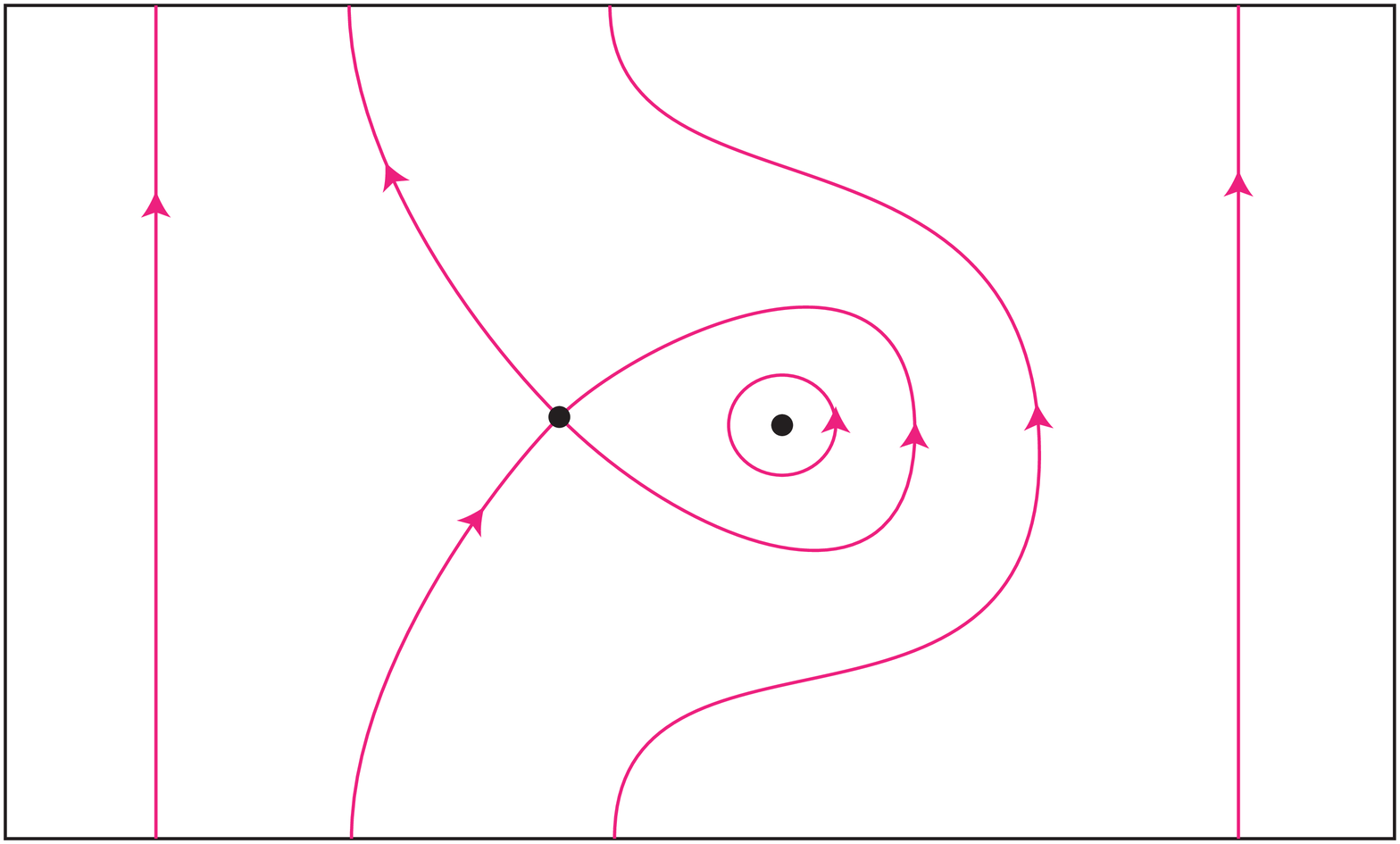}
\put(33.2,29.25){\tiny{$h$}} \put(58.5,33.5){\tiny{$e$}}
\end{overpic}
\caption{The level sets of $g$ on $[0,1]\times[-1,1]$. The arrows
indicate the direction of $X_g$.} \label{modification}
\end{figure}
Two of the separatrices of $h$ go to $[0,1]\times\{\pm 1\}$ and
decompose $[0,1]\times[-1,1]$ into two components; we assume that
$e$ is on the component which contains $(1,0)$.  We choose $g$ so
that ${\bdry g\over \bdry \tau}<\varepsilon$ whenever
$\frac{\partial g}{\partial \tau} \geq 0$. (This happens at those
points in Figure~\ref{modification} where the arrows on the level
sets point downwards.) Next choose a positive function $f$ on
$[0,1]\times [-1,1]$ so that $\frac{\partial f}{\partial \tau}\geq
0$ on $[0,1]\times[-1,1]$ and ${\partial f\over \partial \tau}$ is a
large positive constant where $\frac{\partial g}{\partial \tau} \geq
0$.

On $[0,1]\times[-1,1]\times\Gamma$, with $\alpha$ defined as above,
we compute
$$d\alpha =\frac{\partial f}{\partial \tau} d\tau \wedge dt +dg\wedge
\beta_0 +gd\beta_0.$$ The contact condition for $\alpha$ is
$$g^{n-1} \left(g \frac{\partial f}{\partial \tau} - f\frac{\partial
g}{\partial \tau} \right)>0,$$ and the requirements on ${\bdry
f\over \bdry \tau}$ and ${\bdry g\over \bdry \tau}$ yield the
contact condition.

Let $X_g$ be the Hamiltonian vector field with respect to the
symplectic form $d\tau \wedge dt$. Note that $X_g$ is tangent to the
level sets of $g$. The Reeb vector field $R_\alpha$ is parallel to
$\frac{\partial f}{\partial \tau} R_0 + X_g,$ where $R_0$ is the
Reeb vector field for $\beta_0$ on $\Gamma$. Indeed, we compute
that:
\begin{eqnarray*}
\imath_{\frac{\partial f}{\partial \tau} R_0 + X_g} d\alpha &=& {\bdry
f\over \bdry \tau}\cdot \imath_{X_g} (d\tau \wedge dt) +\imath_{\frac{\partial
f}{\partial \tau} R_0} (dg \wedge
\beta_0) +g\cdot \imath_{\frac{\partial f}{\partial \tau} R_0} d\beta_0 \\
&=&\frac{\partial f}{\partial \tau} dg -\frac{\partial f}{\partial
\tau}dg +0=0.
\end{eqnarray*}

Let $\delta$ be an arc in $[0,1] \times [-1,1]$ which connects the
source $e$ to the point $(1,0)$ and is transversal to $X_g$. Let $D$
be a small disk of radius $r$ about $e$, whose boundary is a level
set of $g$, and let $N_\varepsilon$ be an $\varepsilon$-neighborhood
of $\delta$, with $\varepsilon \ll r$. Consider the manifold $M''$,
obtained from $M$ by adding $([0,1] \times [-1,1] -int(D\cup
N_\varepsilon )) \times \Gamma$.  See Figure~\ref{modification2}.
\begin{figure}[ht]
\begin{overpic}[height=1.5in]{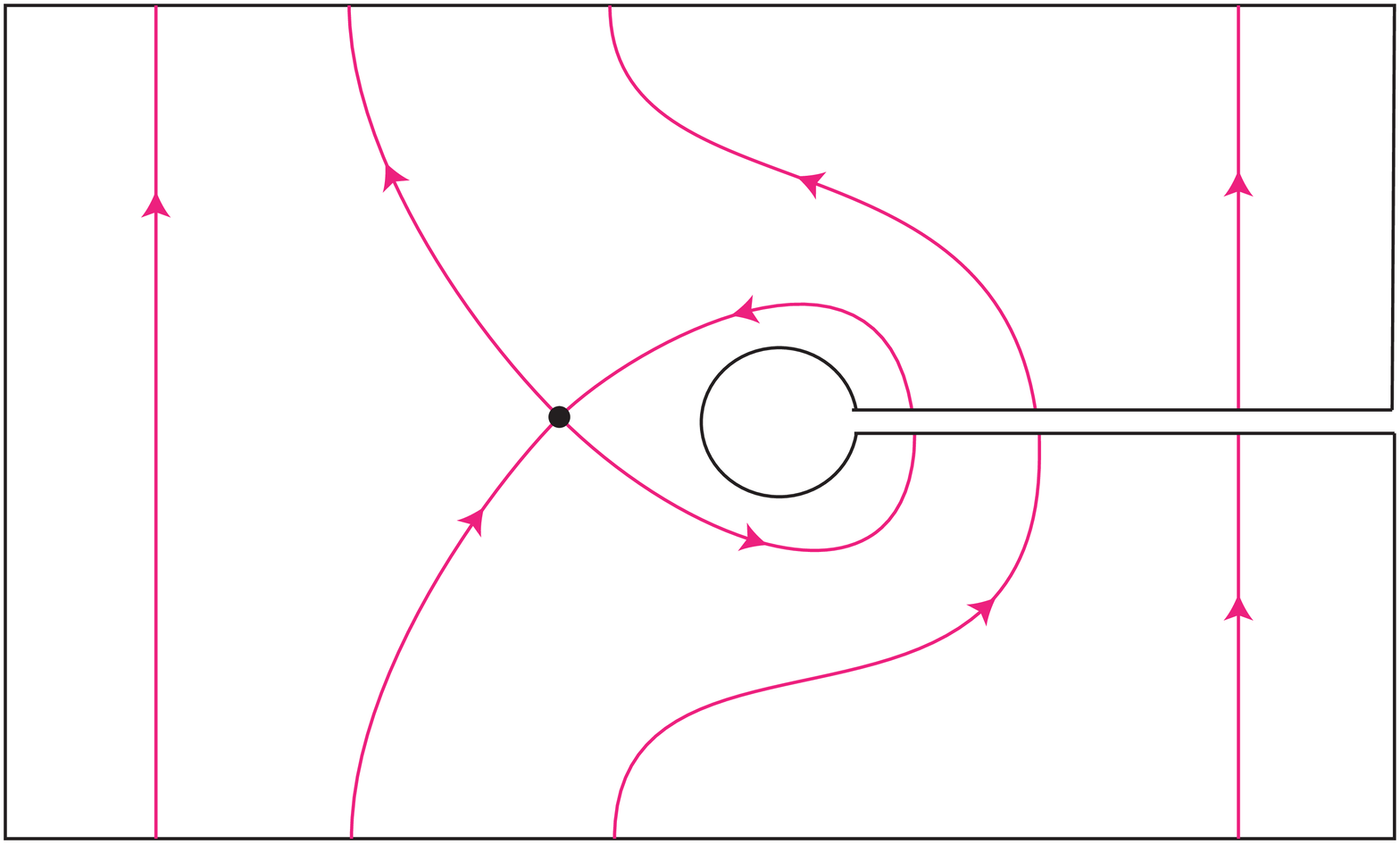}
\end{overpic}
\caption{Excavation of $D\cup N_\varepsilon$ from
$[0,1]\times[-1,1]$.} \label{modification2}
\end{figure}
The contact form on $M''$ is the restriction of $\alpha$, defined
above. We then modify $M''$ slightly so that the corners along
$\tau=1$ are smoothed and the horizontal boundary is transverse to
$R_\alpha$.  Note that $R_\alpha$ is tangent to $(\partial D
-N_\varepsilon ) \times \Gamma$ and the orbits connect from
$R_-(\Gamma)$ to $R_+(\Gamma)$; we may also need to make a slight
modification so that the flow lines of the Reeb vector field from
$R_-(\Gamma)$ to $R_+(\Gamma)$ have constant length near the
vertical boundary. The resulting manifold $(M' ,\alpha|_{M'})$ is a
(convex) sutured contact manifold whose vertical boundary contains
$(\partial D -N_\varepsilon ) \times \Gamma$.

Finally, the isotopy of $(M_{sm},\Gamma,\xi)$ to
$(M'_{sm},\Gamma',\xi')$ follows from observing that there is a
$1$-parameter family of convex submanifolds which connect between
$\bdry M_{sm}$ and $\bdry M'_{sm}$ inside $M$.  We use
Lemma~\ref{lemma: two defns of convex} and find submanifolds which
are (positively or negatively) transverse to $R_\alpha$ except at
some contact submanifold $\{(\tau,t)\}\times\Gamma$, where
$(\tau,t)\in [0,1]\times[-1,1]$.
\end{proof}

The only periodic orbits of $R_\alpha$ that are contained in $M'-M$
are periodic orbits of $R_0$ contained in $\{ h\}\times \Gamma$.
When $\dim M=3$, this construction gives a collection of hyperbolic
orbits (one for each component of $\Gamma$) which are parallel to
the suture $\Gamma$.

\subsection{Gluing sutured contact manifolds}
\label{subsection: gluing}

The procedure of gluing sutured contact manifolds, together with
compatible Reeb vector fields, was first described in \cite{CH} when
$\dim M=3$. Here we describe the sutured gluing so that it is also
applicable to higher dimensions.

Let $(M',\Gamma', U(\Gamma'), \xi')$ be a sutured contact manifold of
dimension $2m+1$ and let $\alpha'$ be an adapted contact form. Let
\[
\pi: U(\Gamma')=[-1,0]\times[-1,1]\times\Gamma'\to [-1,0]\times
\Gamma',
\]
be the projection onto the first and third factors.  If we think of
$[-1,0]\times\Gamma'$ as a subset of $R_+(\Gamma')$ (resp.\
$R_-(\Gamma')$), then we denote the projection by $\pi_+$ (resp.\
$\pi_-$). By definition, the horizontal components $(R_\pm (\Gamma')
,\beta'_\pm =\alpha' \vert_{R_{\pm } (\Gamma')} )$ are Liouville
manifolds. We denote by $Y'_\pm$ their Liouville vector fields. The
contact form $\alpha'$ is $dt+\beta'_\pm$ on the neighborhoods $R_+
(\Gamma') \times [1-\varepsilon ,1 ]$ and $R_- (\Gamma')\times [-1,
-1+\varepsilon ]$ of $R_+ (\Gamma')=R_+ (\Gamma')\times \{ 1\}$ and
$R_- (\Gamma')=R_- (\Gamma')\times \{ -1\}$, found using the Reeb
flow. Also, we may assume that the Reeb vector field $R_{\alpha'}$
is given by $\bdry_t$ on $U(\Gamma')$, after scaling the contact
form.

Take a $2m$-dimensional submanifold $P_+\subset R_+(\Gamma')$ with
smooth boundary\footnote{This is slightly different from what
appears in \cite{CH}, where it is assumed that $\partial P_+$ has
corners along $\bdry (\bdry P_+)_\bdry = \bdry (\bdry P_+)_{int}$.}
so that:
\begin{itemize}
\item $\bdry P_+$ is the union of $(\bdry P_+)_\bdry \subset
\bdry R_+(\Gamma')$ and $(\bdry P_+)_{int}\subset int(R_+(\Gamma'))$
and
\item $\bdry P_+$ is positively transverse to the Liouville vector
field $Y'_+$ on $R_+ (\Gamma')$.
\end{itemize}
Similarly take $P_-\subset R_-(\Gamma')$, $(\bdry P_-)_\bdry$, and
$(\bdry P_-)_{int}$ with $Y'_-$ positively transversal to $\partial
P_-$. See Figure~\ref{suturedgluing}.  Whenever we refer to $(\bdry
P_\pm)_{int}$ and $(\bdry P_\pm)_\bdry$, we assume that closures are
taken as appropriate.

Suppose we have a pair $P_+, P_-$ so that $\pi ((\bdry
P_-)_\bdry)\cap \pi ((\bdry P_+)_\bdry )=\emptyset$ and there is a
diffeomorphism $\phi$ which sends $(P_+,\beta'_+ |_{P_+} )$ to
$(P_-,\beta'_- |_{P_-} )$ and takes $(\bdry P_+)_{int}$ to $(\bdry
P_-)_\bdry$ and $(\bdry P_+)_\bdry$ to $(\bdry P_-)_{int}$.  We will
refer to the triple $(P_+,P_-,\phi)$ as the {\em gluing data}. For
the purposes of gluing, it suffices to require that
$\beta'_+|_{P_+}$ and $\phi^*(\beta'_-|_{P_-})$ be homotopic on
$P_+$, via a homotopy which is constant in a neighborhood of $\bdry
P_+$. In that case, there is a $1$-parameter family of adapted
contact $1$-forms $(\alpha')^\sigma$, $\sigma\in[0,1]$, on
$(M',\Gamma',U(\Gamma'))$ so that $(\alpha')^0=\alpha'$,
$(\alpha')^\sigma= C^\sigma dt+(\beta')^\sigma_\pm$ on
$R_\pm(\Gamma')$, $(\beta')^\sigma_\pm=\beta'_\pm$ on
$R_\pm(\Gamma')-int(P_\pm)$, and
$(\beta')^1_+|_{P_+}=\phi^*((\beta')^1_-|_{P_-})$. This is made
possible by the flexibility theorem of Giroux~\cite{Gi1}. (Note
that, when $\dim M'=3$, we only need $\beta'_+|_{P_+}$ and
$\phi^*(\beta'_-|_{P_-})$ to match up on $\bdry P_+$, since we can
linearly interpolate between primitives of positive area forms on a
surface.)

\begin{figure}[ht]
\begin{overpic}[height=1.5in]{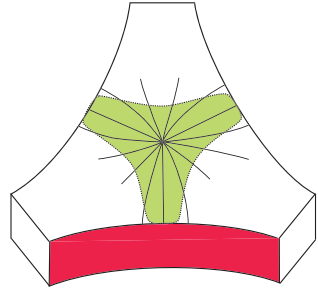}
\end{overpic}
\caption{The diagram shows $P_+\subset R_+(\Gamma')$. The line field
represents $Y'_+\subset \ker\beta'_+$, and the vertical annuli
represent the vertical boundary of $M'$.} \label{suturedgluing}
\end{figure}

Topologically, we construct the sutured manifold $(M,\Gamma)$ from
$(M',\Gamma')$ and the gluing data $(P_+,P_-,\phi)$ as follows: Let
$M=M'/\sim$, where
\begin{enumerate}
\item $x\sim \phi(x)$ for all $x\in P_+$;
\item $x\sim x'$ if $x,x'\in \pi^{-1}(\Gamma')$ and $\pi(x)=\pi(x')
\in \Gamma'$.
\end{enumerate}
In words, (2) says that we collapse the annular neighborhood of
$\Gamma'$ onto $\Gamma'$. Then
\begin{align*}
R_+(\Gamma)= \overline{(R_+(\Gamma')- P_+)}/\sim, & \quad \text{i.e. }
(\bdry P_+)_{int} \text{ is identified with } \pi_+(\bdry P_-)_{\bdry},\\
R_-(\Gamma)= \overline{(R_-(\Gamma')-P_-)}/\sim, & \quad \text{i.e.} (\bdry P_-)_{int}
\text{ is identifiled with } \pi_-(\bdry P_+)_{\bdry},
\end{align*}
and
\[
\Gamma=\overline{(\Gamma'-\pi(\bdry P_+\sqcup \bdry P_-))}/\sim. \]

In \cite[Definition 3.1]{Ga}, Gabai defined the notion of a {\em
sutured manifold decomposition} for sutured $3$-manifolds , which is
the inverse construction of our sutured gluing.

\begin{fact}
Suppose $\dim M=3$. Let $P \subset (M, \Gamma)$ be the surface
obtained by identifying $P_+$ and $P_-$. Then $P$ gives rise to a
sutured manifold decomposition
$$(M, \Gamma) \stackrel{P} \leadsto (M', \Gamma').$$
\end{fact}

\s\n {\bf Construction of $(M_n, \alpha_{n})$.} For the purposes of
studying holomorphic curves, we want to stretch in both the
$\tau$-and $t$-directions.  The construction of the contact manifold
will depend on the parameter $n$, and the resulting glued-up sutured
contact manifold will be written as
$$(M_n,\Gamma_n,U(\Gamma_n),\xi_{n}=\ker(\alpha_{n})).$$

{\em Step 1: gluing top and bottom.} Let $(M^{(0)},\alpha^{(0)})=
(M^{(0)}_n,\alpha^{(0)}_n)$ --- we will
often suppress $n$ to avoid cluttering the notation --- be the
contact manifold obtained from the completion $((M')^*,(\alpha')^*)$
by removing the Side, i.e., $M^{(0)}=M'\cup
(R_+(\Gamma')\times[1,\infty))\cup (R_-(\Gamma')\times
(-\infty,-1])$. Then construct $(M^{(1)},\alpha^{(1)})$ from
\begin{equation} \label{eqn: m1}
M^{(0)}- (P_+\times[n,\infty))- (P_-\times(-\infty,-n]),
\end{equation} by taking closures and identifying:
\begin{itemize}
\item $P_+\times\{n\}$ with $P_-\times\{-n\}$;
\item $(\bdry P_+)_{int}\times[n,\infty)$ with $(\bdry
P_-)_{\bdry}\times [-n,\infty)$;
\item $(\bdry P_+)_\bdry \times (-\infty,n]$ with $(\bdry
P_-)_{int}\times (-\infty,-n]$;
\end{itemize}
all via the identification $(x,t)\mapsto (\phi(x),t-2n)$. Let us
write $P_+^c= \overline{R_+(\Gamma')-P_+}$ and
$P_-^c=\overline{R_-(\Gamma')-P_-}$. Next take $n'\gg 0$ and truncate 
the Top and Bottom of $(M^{(1)},\alpha^{(1)})$ to obtain the (compact) 
sutured manifold $(M^{(2)},\Gamma^{(2)},U(\Gamma^{(2)}))$ with contact 
form $\alpha^{(2)}$ so that $M^{(2)}$ contains
$$M'\cup (P_+^c\times [1,n'])\cup (P_-^c\times[-n',-1]),$$
the Reeb vector field $R=R_{\alpha^{(2)}}$ is transverse to the
horizontal boundary, and the vertical boundary $E$ is foliated by
interval orbits of $R$ with fixed action $\geq 3n'$.

{\em Step 2: Extending the side} Let $\rho\colon  E\to B$ be the fibration whose 
fibers are the interval orbits of $R$, so that $B$ is diffeomorphic to $\Gamma$.  
The base $B$ is a union of finitely many codimension zero submanifolds $B_i$
so that there are local sections $s_i\colon  B_i\to \rho^{-1}(B_i)$ for
which $s_i(B_i)$ are $(2n-1)$-dimensional contact submanifolds.  Let
$(x,t)$ be coordinates on $\rho^{-1}(B_i)$ so that $R=\bdry_t$, $x$
is a local coordinate system for $B_i$, and $t=0$ corresponds to
$s_i(B_i)$. We consider the extension
$$\tilde\rho\colon  [0,\infty)\times E\to [0,\infty)\times B$$
with first coordinate $\tau$ so that $\{0\}\times E$ is identified
with $E\subset M^{(2)}$ and $\tilde\rho(\tau,x,t)=(\tau,\rho(x,t))$.
We can extend the contact form $\alpha^{(2)}$ to a $t$-invariant
contact form on $[0,\infty)\times E$ which is given by
$dt+e^\tau\beta_0(x)$, where $(\tau,x,t)$ are coordinates on
$[0,\infty)\times \rho^{-1}(B_i)$.

At this point we are not guaranteed the existence of a global
section $s\colon B\to E$ which is contact when $\tau=0$.  However, given
any section $s\colon B\to E$, for sufficiently large $\tau=\tau_0$, we
claim that the submanifold $\{\tau_0\}\times s(B)$ is contact.
Indeed, any section $s$ can locally be written as $(x,t)\mapsto (x,
f(x))$, and pulling back $dt+e^\tau\beta_0(x)$ yields
$df(x)+e^\tau\beta_0(x)$. If $\tau_0\gg 0$, the term
$e^\tau\beta_0(x)$ dominates $df(x)$, and the section becomes
contact. Attaching
\begin{equation}
\label{eqn: V3} V=[0,\tau_0]\times E
\end{equation}
to $M^{(2)}$ gives us $(M_n,\alpha_n)$. The horizontal boundary
which is positively (resp.\ negatively) transverse to $R$ will be
called $R_+(\Gamma_n)$ (resp.\ $R_-(\Gamma_n)$). 

We now verify that
$R_\pm(\Gamma_n)$ are Liouville manifolds. The $1$-form $\alpha_n$
restricts to the primitive of a symplectic form on $R_\pm
(\Gamma_n)$, since $R$ is transverse to $R_\pm(\Gamma_n)$. Without
loss of generality the ends of $R_\pm(\Gamma_n)$ are of the form
$[0,\tau_0]\times\bdry E$ with local contact form $dt+df(x)+e^\tau
\beta_0(x)$.  As before, when $\tau_0\gg 0$, $e^\tau\beta_0(x)$
dominates $df(x)$, and the Liouville vector field corresponding to
$df(x)+e^\tau\beta_0(x)$ approaches one parallel to $\bdry_\tau$. It
now follows that the resulting manifold $(M_n, \Gamma_n,
U(\Gamma_n), \xi_n, \alpha_n)$ is a sutured contact manifold.

Now we describe the completion $M_n^*$ of $M_n$. Let
\begin{equation}
\label{eqn: V3star} V^*=[0,\infty)\times \R\times B
\end{equation}
be the completion of $V=[0,\tau_0]\times E$, obtained by extending
to (T), (B), and (S).  Then $M_n^*$ is obtained from $M^{(1)}$ by
attaching $V^*$.

{\em Step 3: interval-fibered extension} Let $$S=(R_+(\Gamma')\times\{n\})\cup 
(R_-(\Gamma')\times
\{-n\})\subset M^{(1)},$$ and let $S_\infty\subset M^{(1)}$ be the
noncompact, possibly disconnected surface obtained from $S$ by
attaching all the $P_+^c\times\{(2k+1)n\}$ and
$P_-^c\times\{(-2k-1)n\}$, where $k$ ranges over all the positive
integers. Note that it is possible for $S_\infty$ to have finitely many
noncompact components and countably many compact components.
There is an embedding
\[
\eta_n\colon  S_\infty\times[-n+1,n-1]\hookrightarrow M^{(1)},
\]
which maps $(x,t)$ to the time-$t$ translation of $x\in S_\infty$
along $\bdry_t$, such that
\begin{equation} \label{eqn: interval fibered 1}
M'_e=M^{(1)}-(S_\infty\times[-n+1,n-1])
\end{equation}
is obtained from $M'$ by attaching an interval-fibered product which
is diffeomorphic to $(S_\infty-S)\times[-1,1]$.

We will call $M_e'$ an {\em infinite interval-fibered extension}
(i.e., an exhaustion of interval-fibered extensions) of
$(M',\Gamma')$. More explicitly,
\begin{equation} \label{eqn: interval fibered 2}
M'_e= M' \cup \bigcup_{k>0} \left(P^c_+ \times [2kn-1, 2kn
+1]\right) \cup\bigcup_{k>0} \left(P_-^c\times [-2kn-1, -2kn
+1]\right),
\end{equation}
where the gluing maps are given as before by $(x,t) \mapsto
(\phi(x), t-2n)$ for $x\in P_+$.

We can write $\overline{S_\infty \setminus S} = S_+ \cup S_-$, where
$S_+$ is the subsurface obtained by gluing together the
$P_+^c\times\{(2k+1)n\}$ pieces and $S_-$ is the subsurface obtained
by gluing together the $P_-^c\times\{(-2k-1)n\}$ pieces. Let us
denote by $(\bdry P_+)_{int}\subset S_+$ the union of connected
components of $(\bdry P_+)_{int}\times \{(2k+1)n\}$ which are on the
boundary of $S_+$, i.e., when $k=1$; similarly define $(\bdry
P_-)_{int}\subset S_-$. Then we can write $M'_e$ more abstractly as
$$M' \cup (S_- \times [-1,1]) \cup (S_+ \times [-1,1]),$$ where we
glue $(\partial P_+)_{int} \times [-1,1] \subset S_+ \times [-1,1]$
to $(\partial P_-)_{\partial} \times [-1,1] \subset M'$ by
$(\phi,id)$, and $(\partial P_-)_{int} \times [-1,1] \subset S_-
\times [-1,1]$ to $(\partial P_+)_{\partial} \times [-1,1] \subset
M'$.

Summarizing, we have the following:

\begin{lemma}
Suppose $n>0$. Given a sutured contact manifold
$(M',\Gamma',U(\Gamma'),\alpha')$ and gluing data $(P_+,P_-,\phi)$,
there exists an inclusion of sutured contact manifolds
$$(M',\Gamma',U(\Gamma'),\alpha')\hookrightarrow
(M_n,\Gamma_n,U(\Gamma_n),\alpha_{n}),$$ where
$(M_n,\Gamma_n,U(\Gamma_n))$ is homeomorphic to
$(M,\Gamma,U(\Gamma))$ and the completion of $M_n$ is
$M_n^*=M^{(1)}\cup V^*$.  Here $V^*$ is a fibered piece given by
Equation~(\ref{eqn: V3star}) and $M^{(1)}$ admits a noncompact
embedding of $S_\infty\times [-n+1,n-1]$ so that
$R_{\alpha_{n}}=\bdry_t$ on $(S_\infty-S) \times [-n+1,n-1]$ and
$S\times ([-n+1,-\varepsilon]\cup [\varepsilon, n-1])$, and
$M^{(1)}-(S_\infty\times[-n+1,n-1])$ is an infinite interval-fibered
extension of $(M',\Gamma',U(\Gamma'),\alpha')$ which is independent
of $n$.
\end{lemma}

\s\n {\bf Almost complex structures.} We now discuss the
gluing/extension of almost complex structures under sutured manifold
gluing.

We first define an almost complex structure $J'_{\kappa}$ on
$\R\times M'$ which is tailored to the sutured contact manifold
$(M', \Gamma', U(\Gamma'), \alpha')$. Consider the neighborhood
$U(\Gamma')=[-1,0]\times[-1,1]\times \Gamma'$ with coordinates
$(\tau,t,x)$, where we may assume that $\beta'_+=e^\tau
\overline\beta_0$ and $\overline\beta_0=\beta'_+|_{\{0,0\}\times
\Gamma'}$. Choose a diffeomorphism
$$H_\kappa\colon [-1,0]\times \Gamma'\stackrel\sim\to
[0,\kappa]\times \Gamma',$$
$$(\tau,x)\mapsto (h_\kappa(\tau),x),$$
where $h_\kappa\colon [-1,0]\stackrel\sim\to[0,\kappa]$, $h_\kappa(-1)=0$,
$h_\kappa(0)=\kappa$, $h_\kappa'(\tau)=1$ in a neighbourhood of $\tau=-1,0$, and
$h_\kappa$ is linear outside a biger neighnourhood of. $\tau=-1,0$. Then choose the
projection $(J'_\kappa)_0$ of $J'_\kappa$ to $R_\pm(\Gamma')$ so
that:
\begin{itemize}
\item $(J'_\kappa)_0$ is adapted to $\overline\beta_0$ on
$H_\kappa([-1,0]\times\Gamma')=[0,\kappa]\times \Gamma'$;
\item $(J'_\kappa)_0$ is independent of $\kappa$ on
$R_\pm(\Gamma')-((-1,0]\times\Gamma')$;
\item $\phi_*$ takes $(J'_\kappa)_0$ along $\bdry (\bdry P_+)_{int}$
to $(J'_\kappa)_0$ along $\bdry (\bdry P_-)_{int}$, so that they
agree when projected to the base $B$ of the fibration $\rho\colon E\to B$.
\end{itemize}
On $M'-U(\Gamma')$, choose $J'_{\kappa}$ to be independent of
$\kappa$.

We then extend $J'_{\kappa}$ to an almost complex structure
$J_{\kappa,n}$ on $\R\times M^{(1)}$ which satisfies the following:
\begin{enumerate}
\item $J_{\kappa,n}$ is adapted to the symplectization
$(\R\times M^{(1)},d(e^s\alpha^{(1)}))$.
\item $J_{\kappa,n}$ is $\bdry_t$-invariant on each connected
component of $P_+^c\times[2n-1,\infty)$, $P_-^c\times (-\infty,
-2n+1]$, $\bdry M^{(1)}$, and $S\times( [-n+1,-\varepsilon]\cup
[\varepsilon,n-1])$.
\item The extension to the interior of $S\times
[-\varepsilon,\varepsilon]$ is arbitrary, but is independent of $n$.
\end{enumerate}
The almost complex structure $J_{\kappa,n}$ on $P_+^c\times
[2n-1,\infty)$ is defined by specifying the projection
$(J_{\kappa,n})_0$ of $J_{\kappa,n}$ to $P_+^c$ so that
$(J_{\kappa,n})_0$ agrees with $(J'_\kappa)_0$ along $\bdry
P_+^c-\bdry P_+$ and with $\phi_*(J'_\kappa)_0$ along $\bdry
P_+^c\cap \bdry P_+$. The extension of $(J_{\kappa,n})_0$ to the
interior of $P_+^c$ is arbitrary, provided it is compatible with
$d\beta'_+$.  In particular, $(J_{\kappa,n})_0$ does not need to
agree with $(J'_\kappa)_0$ on all of $P_+^c$. The almost complex
structure $(J_{\kappa,n})_0$ is defined similarly on $P_-^c$.

Next we extend $J_{\kappa,n}$ to $V^*=[0,\infty)\times \R\times B$,
as follows: On each $[0,\infty)\times \rho^{-1}(B_i)$ with
coordinates $(\tau,x,t)$ and contact form $dt+e^\tau\beta_0(x)$,
choose an $e^\tau\beta_0$-adapted almost complex structure $J_0$ on
$[0,\infty)\times s_i(B_i)=\{t=0\}$.
This determines $J_{\kappa,n}$ which projects to $J_0$.
By construction, we may also assume that the sections $s_i(B_i)$ and
$s_j(B_j)$ differ by $t=\mbox{const}$ on the overlap
$\rho^{-1}(B_i\cap B_j)$; hence the contact form on
$[0,\infty)\times \rho^{-1}(B_i\cap B_j)$ is $dt+e^\tau\beta_0(x)$
with respect to either coordinate chart. This means that we can
choose a $J_0$ on all of $[0,\infty)\times B$ and a $J_{\kappa,n}$
which projects to $J_0$ on all of $V^*$.

We now verify that $J_{\kappa,n}$ is tailored to
$(M_n^*,\alpha_n^*)$. Conditions (A$_0$) and (A$_1$) are easily
satisfied.  It remains to verify (A$_2$), namely $J_{\kappa,n}$ is
$d\beta_\pm$-positive, where $\beta_\pm$ is the restriction of
$\alpha_n$ to $R_\pm (\Gamma_{n})$. The reason this needs
verification is that the adjustment in the vertical direction
implies that the $t$-variable undergoes a coordinate change of the
type $(t,y)\mapsto (t+f(y),y)$, where $y$ is a coordinate on
$R_\pm(\Gamma')$.  By pulling back, we see that
$\beta_\pm(y)=df(y)+\beta'_\pm(y)$, and $d\beta_\pm=d\beta'_\pm$.
Hence the $d\beta_\pm$-positivity is inherited from the
$d\beta'_\pm$-positivity.

\begin{rmk} \label{rmk: tau is plurisubharmonic}
Let $J_0$ be the projection of $J_{\kappa,n}$ onto $[0,\infty)\times
B$. Since the $1$-forms $\beta_0(x)$ patch to give a contact
$1$-form $\beta_0$ on $B$, it follows that $([0,\infty)\times
B,d(e^\tau\beta_0))$ is a (positive) symplectization and $J_0$ is
adapted to the symplectization.  Hence $\tau$ is a plurisubharmonic
function with respect to $J_0$.
\end{rmk}

\subsection{Gluing along convex submanifolds}
\label{subsection: gluing along convex submanifolds} In this
subsection we discuss gluing along convex submanifolds. In
particular, we carefully construct a contact $1$-form which is
suited to counting holomorphic curves.

Let $(M',\Gamma', U(\Gamma'),\xi')$ be a sutured contact manifold of
dimension $2m+1$ and $\alpha'$ be an adapted contact $1$-form.  Let
$S_1$ and $S_2$ be two disjoint components of $\partial M'$ and let
$S_i^\pm =S_i \cap R_\pm (\Gamma')$. A neighborhood of $S_i^+$ in
$(M',\alpha')$ is given by $(S_i^+ \times [1-\varepsilon ,1], dt
+\beta' )$, where $t\in [1-\varepsilon ,1]$ and $S_i^+ =S_i^+ \times
\{ 1\}$. Similarly, we have a neighborhood $(S_i^- \times [-1,
-1+\varepsilon], dt +\beta')$ of $S_i^- =S_i^- \times \{ -1\}$.
Suppose there is a diffeomorphism $h \colon S_1^+ \cup S_1^- \rightarrow
S_2^- \cup S_2^+$ which takes $(S_1^+ ,\beta' \vert_{S_1^+} )$ to
$(S_2^- ,\beta' \vert_{S_2^-} )$ and $(S_1^- ,\beta' \vert_{S_1^-}
)$ to $(S_2^+ ,\beta' \vert_{S_2^+} )$, and which can be extended to
a (piecewise smooth) homeomorphism from $S_1$ to $S_2$. Also suppose
that when we compose $h|_{S_1^+}$ and $h^{-1}|_{S_2^+}$ with the
identifications of $\partial S_i^-$ and $\partial S_i^+$ by the flow
of $\partial_t$ in $U(\Gamma')$, we get the identity on $\partial
S_1^+$.

Instead of gluing directly using $h$, we insert layers as follows:
Fix $n> 0$. Then let $(M'_n,\alpha'_n)$ be the contact manifold
obtained by gluing the products $(S_1^+ \times [0,n],
dt+\beta'\vert_{S_1^+} )$ and $(S_1^- \times [0,n],
dt+\beta'\vert_{S_1^-} )$ to $(M',\alpha' )$ by identifying $S_1^+$
with $S^+_1\times\{0\}$, $S^-_2$ with $S_1^+\times\{n\}$, $S_2^+$
with $S_1^-\times\{0\}$, and $S_1^-$ with $S_1^-\times\{n\}$.

We now construct the contact manifold $(M_n,\alpha_{n,f,g})$ by
filling in some of the boundary components of $(M'_n,\alpha'_n)$.
Here $f,g$ are smooth functions $[0,1]\sqcup\dots\sqcup[0,1]\to \R$,
where there is a copy of $[0,1]$ for each component $V$ of
$S_1\cap\Gamma$. Moreover, $f,g$ depend on $n$. Consider a boundary
component of $(M'_n,\alpha'_n)$ of the form $V\times S^1$, where $V$
corresponds to a connected component of $S_1\cap\Gamma$ and
$S^1=\R/2\pi\Z$ has coordinate $\theta$. The contact form
$\alpha'_n$ on $V\times S^1$ is given by $a_n d\theta+\beta'_0$, where
$a_n$ is a constant $> {n\over \pi}$ and $\beta'_0=\alpha'|_{V}$. We
then fill $V\times S^1$ with $V \times D^2$, where we are using
polar coordinates $(r,\theta)$ on $D^2$ and $S^1=\{r=1\}$. We
require the contact form $\alpha_{n,f,g}$ on $V\times D^2$ to be of
the form
$$f(r) d\theta +g(r)\beta'_0,$$  with boundary condition
$(f(1),g(1))=(a,1)$.  The contact condition is equivalent to
$$f'g-g'f = (f',g')\cdot (g,-f)>0,$$ which in words says that the
path $\{(f(r),g(r)), r\in[0,1]\}$, is transverse to radial rays in
the $(f,g)$-plane and rotates clockwise around the origin. The Reeb
vector field $R=R_{\alpha_{n,f,g}}$ is given by $R=\frac{1}{f'g-g'f}
(f' R_0 -g'\partial_\theta)$, where $R_0$ is the Reeb vector field
for $\beta'_0$.

We now choose specific $f$ and $g$ so that an orbit $\gamma$ of $R$
which passes through $M_n-M'$ has action $A(\gamma)>n$.  Let
$B_0,B_1$ be large positive constants so that $B_0-aB_1\gg 0$. Then
set:
\begin{itemize}
\item $(f(r),g(r))=(ar^2, B_0-aB_1r^2)$ for
$r\in[0,1-2\varepsilon]$;
\item $0<f'(r)$ and $g'(r)< 0$ for
$r\in[1-2\varepsilon,1-\varepsilon]$;
\item $f(r)=a$, $g'(r)< 0$ for $r\in[1-\varepsilon, 1]$;
\item $(f(1),g(1))=(a,1)$;
\item $f(r)=a$, $g(r)=e^{1-r}$ for $r\in[1-\varepsilon/2,1]$.
\end{itemize}
The last condition is to ensure the smooth gluing of
$fd\theta+g\beta'_0$ with $\alpha_n'$.

If $\gamma$ passes through $M_n'-M'$, then we claim that
$A(\gamma)>n$ by construction. Suppose $\gamma$ lies in $V\times
S^1$. Then we compute $R={1\over B_0}R_0+{B_1\over B_0}\bdry_\theta$
for $0\leq r\leq 1-2\varepsilon$.  Since ${B_0\over B_1}\gg a$, it
takes at least $2\pi a$ units of time to travel once around the
$\theta$-direction; hence $A(\gamma)>n$. (When $r=0$, then
$R={1\over B_0}R_0$ and $\gamma$ is tangent to $V\times\{0\}$.  If
$B_0$ is sufficiently large, then $A(\gamma)>n$.) On the other hand,
for $r\in[1-2\varepsilon,1]$, the coefficient in front of
$\bdry_\theta$ in $R$ is less than ${1\over f}\approx {1\over a}$;
hence $A(\gamma)>n$ for sufficiently small $\varepsilon$.


Summarizing the above discussion, we have:

\begin{lemma}\label{lemma: convexgluing}
Let $(M,\xi)$ be a compact contact manifold of dimension $2m+1$ and
$(S, \Gamma_S) \subset (M, \xi)$ be a convex submanifold. If
$(M',\xi')$ is obtained from $(M,\xi)$ by cutting along $S$, then,
for any $n> 0$ and appropriate $f=f(n),g=g(n)$, $(M,\xi)$ is
contactomorphic to $(M_n,\ker\alpha_{n,f,g})$, where
$\alpha_{n,f,g}$ is obtained from a contact $1$-form $\alpha'$
adapted to $(M',\Gamma', U(\Gamma'),\xi')$ by attaching (i) layers
$(S_1^+ \times [0,n], dt+\beta'\vert_{S_1^+} )$ and $(S_1^- \times
[0,n], dt+\beta'\vert_{S_1^-} )$ and (ii) $(V\times
S^1,fd\theta+g\beta'_0)$. The Reeb vector field
$R=R_{\alpha_{n,f,g}}$ satisfies the following:
\begin{itemize}
\item Every orbit of $R$ which intersects $M_n-M'$ has action
larger than $n$;
\item $R$ is tangent to $V \times \{0\}$, positively transverse to
$S^+_i \times \{ t\}$ and $S^-_i \times \{ t\}$ for all $t\in
[0,n]$, and transverse to $\theta=const$ on $V\times (D^2-\{0\})$.
\end{itemize}
\end{lemma}

When $\dim M=3$, the dividing set $V \times \{0\}$ is a periodic
orbit of $R$.

We define the tailored almost complex structure $J=J_{n,f,g}$ on
$\R\times M_n^*$ as follows: Choose a tailored $J'$ on $\R\times
(M')^*$ so that its restrictions to $S_1^+$ and $S_2^-$, and also
its restrictions to $S_1^-$ and $S_2^+$, agree via $h$. We then
extend $J'$ to $J$ on $\R\times S_1^\pm\times [0,n]$ so that $J$ is
invariant in both the $s$- and $t$-directions.  Finally, we extend
$J$ so that it is $\alpha_{n,f,g}$-adapted on $\R\times V\times
D^2$.

\section{Compactness results}
\label{section: compactness results}

Let $(M,\alpha)$ be a sutured contact manifold, and let
$(M^*,\alpha^*)$ denote its completion as defined in Section
\ref{sec:completion}.  Let $J$ be an almost complex structure on
$\R\times M^*$ tailored to $(M^*,\alpha^*)$, as defined in
\ref{subsection: almost complex structure}.  In this section we show
that the SFT compactness theorem for holomorphic curves in the
symplectization of a closed contact manifold, proved in \cite{BEHWZ}
and \cite{CM}, extends to the case of $J$-holomorphic curves in
$\R\times M^*$.  At the end of this section, we extend the compactness
theorem for embedded contact homology \cite{Hu1} to $\R\times M^*$ in
the case $\dim(M)=3$.

\subsection{Convergence of stable Riemann surfaces}
We begin by reviewing some notation and classical results about the
convergence of stable Riemann surfaces, following \cite{BEHWZ}.

A {\em marked Riemann surface\/} is a triple $\mathbf{S}=(\Sigma, j,
\mathbf{m})$ consisting of a closed Riemann surface $(\Sigma,j)$ and a
finite ordered set $\mathbf{m}\subset \Sigma$ of ``punctures'' or
``marked points''.  (The surface $\Sigma$ does not need to be
connected.) Two marked Riemann surfaces $\mathbf{S}=(\Sigma, j,
\mathbf{m})$ and $\mathbf{S}'=(\Sigma', j', \mathbf{m}')$ are said to
be {\em equivalent} if there exists a diffeomorphism $\varphi \colon
\Sigma \stackrel\sim\to \Sigma'$ such that $\varphi_*j=j'$ and
$\varphi(\mathbf{m})= \mathbf{m}'$ in an order-preserving way. The
surface $\mathbf{S}$ is called {\em stable} if, on each connected
component $\Sigma_0$ of $\Sigma$, we have
$2g(\Sigma_0)+\mu(\Sigma_0)\geq 3$. Here $g(\Sigma_0)$ is the genus of
$\Sigma_0$ and $\mu(\Sigma_0)$ is the number of marked points on
$\Sigma_0$.  A {\em nodal} Riemann surface is a quadruple
$\mathbf{S}=(\Sigma,j,\mathbf{m},D)$, where $(\Sigma,j,\mathbf{m})$ is
a marked Riemann surface as before, and
$D\subset\Sigma\setminus\mathbf{m}$ is a finite set partitioned into
unordered pairs $\{(d_i',d_i'')\}$.  A stable nodal Riemann surface is
defined as above, where the set of marked points is taken to be
$\mathbf{m}\sqcup D$.  From a nodal surface $\mathbf{S}=(\Sigma, j,
\mathbf{m}, D)$ one can form a singular surface $\hat{\Sigma}_D=
\Sigma / (d_i' \sim d_i'')$.

Let $\mathbf{S}=(\Sigma,j,\mathbf{m})$ be a stable marked Riemann
surface. Then on $\dot{\Sigma}= \Sigma \setminus \mathbf{m}$, there is
a unique complete, finite volume hyperbolic metric $h^{j,\mathbf{m}}$
which is compatible with $j$. Denote its injectivity radius by
$\rho$. Given $\epsilon>0$, we define the ``thick part'' and ``thin
part''
\begin{align*}
& {\rm Thick}_{\epsilon}(\mathbf{S})= \{ x \in \dot\Sigma ~|~
\rho(x) \ge \epsilon \}, \\
& {\rm Thin}_{\epsilon}(\mathbf{S})= \overline{\{ x \in \dot\Sigma
~|~ \rho(x) < \epsilon \}}.
\end{align*}
If $\epsilon < \log (1+ \sqrt{2})$, then each component of ${\rm
Thin}_{\epsilon}(\mathbf{S})$ is conformally equivalent to a
punctured disk or to a finite cylinder.  Each cylindrical component
$C$ of ${\rm Thin}_{\epsilon}(\mathbf{S})$ contains a unique closed
geodesic $\Gamma_C$. The thick and thin parts of complete, finite
volume hyperbolic metrics for stable nodal Riemann surfaces are defined
similarly, except that we take $\dot\Sigma=
\Sigma\setminus(\mathbf{m}\cup D)$.

\begin{defn}\label{dm}
A sequence of marked Riemann surfaces $\mathbf{S}_n = (\Sigma_n,
j_n, \mathbf{m}_n)$ {\em converges} to a nodal Riemann surface
$\mathbf{S} = (\Sigma, j, \mathbf{m}, D)$ if the following hold:

\begin{itemize}
\item There exist a smooth surface $\Sigma^D$, diffeomorphisms $\varphi_n
\colon \Sigma^D \stackrel\sim\to \Sigma_n$ and an ordered set
$\mathbf{m}^D \subset \Sigma^D$ such that $\varphi_n(\mathbf{m}^D)=
\mathbf{m}_n$ (as ordered sets).

\item There exist disjoint circles
$\Gamma_1, \ldots, \Gamma_k \subset \Sigma^D  \setminus
\mathbf{m}^D$ and a map $\varphi \colon \Sigma^D \to \hat{\Sigma}_D$
such that $\varphi$ is a diffeomorphism between $\Sigma^D \setminus
\bigcup \Gamma_i$ and $\Sigma \setminus D$, and
$\varphi(\mathbf{m}^D)= \mathbf{m}$ (as ordered sets).

\item $\varphi_n(\Gamma_i) \subset \Sigma_n$ are
closed geodesics for the metric $h^{j_n, \mathbf{m}_n}$ and are
contained in the thin part (defined using some $\epsilon<\log(1+\sqrt{2})$).

\item $\varphi_n^*j_n \to \varphi^* j$  in $C^{\infty}_{loc}(\Sigma^D
\setminus \bigcup \Gamma_i)$ or, equivalently, $\varphi_n^* (h^{j_n,
\mathbf{m}_n}) \to \varphi^* (h^{j, \mathbf{m}})$ in
$C^{\infty}_{loc}(\Sigma^D \setminus (\bigcup \Gamma_i \cup
\mathbf{m}^D))$.

\item Given a point $c_i \in \Gamma_i$, the geodesic arcs
$\delta_i^n$ for the metric $\varphi^*(h^{j_n, \mathbf{m}_n})$ which
intersect $\Gamma_i$ orthogonally at $c_i$ and whose endpoints are
contained in the thick part of $\Sigma^D$ for the metric
$\varphi^*(h^{j_n, \mathbf{m}_n})$, converge uniformly as $n \to
\infty$ to a continuous arc in $\Sigma^D$ which passes through $c_i$
and is a geodesic in $\Sigma^D \setminus (\bigcup \Gamma_i \cup
\mathbf{m}^D)$ for the metric $\varphi^* (h^{j, \mathbf{m}})$.
\end{itemize}
\end{defn}

\begin{thm}\label{classical}
Any sequence of stable marked Riemann surfaces $\mathbf{S}_n= (\Sigma_n,
j_n, \mathbf{m}_n)$ with fixed $2g(\Sigma_n)+\mu(\Sigma_n)$ has a
subsequence which converges to a nodal Riemann surface $\mathbf{S}=
(\Sigma, j, \mathbf{m}, D)$.
\end{thm}

\begin{fact}\label{lunghezza}
Let $g_n$ be a sequence of Riemannian metrics which converges
uniformly to a Riemannian metric $g$. Let $l_n$ be the length
functional for $g_n$ and $l$ be the length functional for $g$. Then
for any $\delta >0$ there exists $n_0 \in \N$ such that for all
$n\geq n_0$ and for all arcs $\gamma$ we have
\[ (1 - \delta) l(\gamma) \le l_n(\gamma) \le (1+\delta)  l(\gamma). \]
\end{fact}

The proof of Fact~\ref{lunghezza} is an easy exercise.

\begin{prop} \label{radius}
Let $\mathbf{S}_n = (\Sigma_n, j_n, \mathbf{m}_n)$ be a sequence of
Riemann surfaces which converges to a nodal Riemann surface
$\mathbf{S}= (\Sigma, j, \mathbf{m}, D)$, in the sense of
Definition~\ref{dm}.  Then for all $\epsilon, \delta >0$ there is
$n_0 \in \N$ such that
\begin{align*}
\varphi^{-1}({\rm Thick}_{\epsilon}(\mathbf{S})) & \subset
\varphi_n^{-1}({\rm
Thick}_{(1- \delta)\epsilon} (\mathbf{S}_n)), \\
\varphi^{-1}({\rm Thin}_{\epsilon}(\mathbf{S})) & \subset
\varphi_n^{-1}({\rm Thin}_{(1+ \delta)\epsilon} (\mathbf{S}_n)),
\end{align*}
for all $n \ge n_0$.
\end{prop}

\begin{proof}
Let $h_n$ be the complete, finite volume hyperbolic metric on
$\Sigma^D \setminus \mathbf{m}^D$ which is compatible with
$\varphi_n^*j_n$, and $h$ be the complete, finite volume hyperbolic
metric on $\Sigma^D \setminus (\mathbf{m}^D \cup \Gamma_i)$ which is
compatible with $\varphi^*j$. Also let $\rho_n$ and $\rho$ be the
injectivity radii of $h_n$ and $h$, respectively. Let $z \in
\Sigma^D \setminus (\mathbf{m}^D \cup \Gamma_i)$. In order to
compute the injectivity radii at $z$, it suffices to compute the
shortest geodesic loops based at $z$ (see for example
\cite[Lemma~4.8]{Hum}). Let $\gamma$ be the shortest $g$-geodesic
loop based at $z$, and let $\gamma_n$ be the shortest $g_n$-geodesic
loop based at $z$. By Fact~\ref{lunghezza}, we have
\[ (1 - \delta) l(\gamma) \le  (1 - \delta) l(\gamma_n) \le l_n(\gamma_n), \]
for sufficiently large $n$. Hence $(1-\delta) \rho(z) \le
\rho_n(z)$. We then conclude that $\rho_n(z) \ge (1 -\delta)
\epsilon$ whenever $\rho(z) \ge \epsilon$.

On the other hand, if $l(\gamma) < \epsilon$, then we have
\[ l_n(\gamma_n) \le l_n(\gamma) \le (1+ \delta) l(\gamma), \]
for sufficiently large $n$. We then conclude that $\rho_n(z)
\le (1+\delta) \epsilon$ whenever $\rho(z) < \epsilon$.
\end{proof}

\subsection{Holomorphic curves in $\R\times M^*$.}

Let $J$ be a tailored almost complex structure on $\R\times M^*$ as usual.
Let $(\Sigma,j,\mathbf{m})$ be a marked Riemann
surface.  The notation
\[
F=(a,f)\colon (\Sigma,j,\mathbf{m})\to(\R\times M^*,J)
\]
denotes a $(j,J)$-holomorphic map from the punctured Riemann surface
$\dot\Sigma=\Sigma\setminus\mathbf{m}$ to $M^*$.  If $p\in\mathbf{m}$,
and if $\gamma$ is a Reeb orbit of $\alpha$, we say that $F$ is
``positively asymptotic'' to $\gamma$ at $p$ if $\lim_{z\to
  p}a(z)=+\infty$ and if the restriction of $f$ to a circle around $p$
converges to $\gamma$ as the size of the circle converges to zero.  We
also say that $p$ is a ``positive puncture'' of $F$ asymptotic to
$\gamma$.  We define ``negatively asymptotic'' analogously but with
$\lim_{z\to p}a(z)=-\infty$.

Now let $\underline{\gamma}=(\gamma_1,\ldots,\gamma_k)$ and
$\underline{\gamma}'=(\gamma_1',\ldots,\gamma_l')$ be finite ordered
lists of Reeb orbits, possibly with repetitions.  Let
$\mathcal{M}_g(\underline{\gamma};\underline{\gamma}';J)$ denote the
moduli space of holomorphic maps $F$ as above such that $\Sigma$ has
genus $g$, there are $k+l$ marked points in $\mathbf{m}$, $F$ is
positively asymptotic to $\gamma_i$ at the $i^{th}$ marked point, and
$F$ is negatively asymptotic to $\gamma_j'$ at the $(k+j)^{th}$ marked
point.

We wish to extend the SFT compactness theorem to sequences of
holomorphic curves in these moduli spaces.  To do so, it is sufficient
to show that for any sequence of such curves, the projections to $M^*$
are confined in a compact set.

We first show that a sequence of holomorphic curves cannot escape from
the side (S).

\begin{lemma} \label{lemma: bound on tau}
Let $F \in \mathcal{M}_g(\m{\gamma};\m{\gamma}';J)$ for some $g$,
$\m{\gamma}$, and $\m{\gamma}'$. Then $\tau \circ F(z) \leq 0 $ for
all $z \in \dot\Sigma$.
\end{lemma}

\begin{proof}
Suppose there exists a point $z \in \dot\Sigma$ such that $\tau
\circ F(z)>0$. Then $\tau \circ F$ has a local maximum, which we
assume without loss of generality to be attained at $z$.

The projection of $J$ to $\widehat{R_+(\Gamma)}$ is $J_0$.  Hence
the projection of $F$ to $\widehat{R_+(\Gamma)}$ (when restricted to
$\R\times (M^*-M)$) is a $J_0$-holomorphic map. Since $\tau$
restricted to $\widehat{R_+(\Gamma)}$ is a plurisubharmonic
function, a local maximum of $\tau \circ F$ is forbidden by the
maximum principle.
\end{proof}

The main task in the rest of this section is to show that a sequence
of holomorphic curves cannot escape from the top (T) or bottom (B).
For this purpose we ned to consider somewhat more general holomorphic
curves than the ones in
$\mathcal{M}_g(\overline{\gamma};\overline{\gamma}';J)$, in particular
the restrictions of such curves to certain subsets of the domain.
However our curves $F$ will always have an upper bound on $\tau\circ
F$ as a result of Lemma~\ref{lemma: bound on tau}.  In addition, all
our curves $F$ will have finite {\em Hofer energy\/} $E(F)$; see
\cite[Sec.\ 5.3]{BEHWZ} for the definition of Hofer energy, and
\cite[Prop.\ 5.13]{BEHWZ} for the proof that any curve in
$\mathcal{M}_g(\overline{\gamma};\overline{\gamma}';J)$ (and
consequently the restriction of any such curve to a subset of the
domain) has finite Hofer energy.

\subsection{The Stein case}
\label{subsection: Stein case}

It is easiest to obtain a bound on $|t|$ when
$(\widehat{R_\pm(\Gamma)},J_0,\widehat{\beta}_\pm)$ is a Stein
manifold.  (Recall that we can arrange to be in this situation
when $\widehat{R_\pm(\Gamma)}$ is a surface.)
In this case we have $|t|\le 1$ by the following lemma and
corollary.

\begin{lemma}
\label{lemma: harmonic} Suppose $J_0$ is an integrable complex
structure which makes
$(\widehat{R_\pm(\Gamma)},J_0,\widehat\beta_\pm)$ into a Stein
manifold. If $F\colon  (\Sigma,j,\mathbf{m})\rightarrow (\R\times M^*,J)$
is a holomorphic map, then $t \circ F$ is a harmonic function on the
open set $\{ z \in \dot\Sigma : |t \circ F(z)|>1 \}$.
\end{lemma}

\begin{proof}
We prove the lemma for the case when $(t \circ F)(z) >1$; the
argument for $(t \circ F)(z) < -1$ is identical. The symplectization
of the top (T) is written as $\R \times (1,\infty)\times
\widehat{R_+(\Gamma)}$, with coordinates $s$ on $\R$ and $t$ on
$(1,\infty)$. On $(1,\infty)\times \widehat{R_+(\Gamma)}$, we may
take the contact form to be $\alpha = dt + \beta$, where
$\beta=\widehat\beta_+$. Since $(\widehat{R_+(\Gamma)}, J_0, \beta)$
is Stein, $\beta\circ J_0=df$, where $f$ is the strictly
plurisubharmonic function. With these conventions in place, we
compute the Laplacian of $t \circ F$:
\begin{align*}
  dd^{\C}(t \circ F) & = d(d(t \circ F) \circ j)= d(F^*(dt \circ J))=
  F^* d((\alpha-\beta)\circ J)\\
  & = F^* d(ds- (\beta \circ J))= - F^* d (\beta \circ J).
\end{align*}

We now claim that $\beta\circ J=df$. First we observe that
$\beta\circ J$ and $df$ both evaluate to zero on $\bdry_s$ and
$\bdry_t$. Next we compare $(\beta\circ J)(X)$ and $(\beta\circ
J_0)(X)=df(X)$ for any vector $X$ tangent to
$\widehat{R_+(\Gamma)}$. Since $J(X) = J_0(X)+
v_0(X)\partial_s+v_1(X) \partial_t$ by the definition of $J$ and
$\beta(\partial_s)=\beta(\partial_t)=0$, it follows that
$(\beta\circ J)(X)=(\beta\circ J_0)(X)$.

Finally, since $\beta \circ J$ is exact, we conclude that $dd^{\C}(t
\circ F) = 0$.
\end{proof}

\begin{cor}
  Suppose $J_0$ is an integrable complex structure which makes
  $(\widehat{R_\pm(\Gamma)},J_0,\widehat\beta_\pm)$ into a Stein
  manifold. If
  $F\in\mathcal{M}_g(\underline{\gamma};\underline{\gamma}';J)$, then
  $|(t \circ F)(z)| \le 1$ for all $z \in \dot\Sigma$.
\end{cor}

\begin{proof}
If there is a point $z \in \dot\Sigma$ such that $|(t \circ F)(z)|
>1$, then there is a local maximum for $t \circ F$, which we may
assume to be attained at $z$. But this is forbidden by the maximum
principle because $t \circ F$  is harmonic in a neighborhood of $z$
by Lemma~\ref{lemma: harmonic}.
\end{proof}

The non-Stein case is less nice because we do not necessarily have
$|t|\le 1$.  However we can still obtain a theoretical upper bound on
$|t|$, as the rest of this section will explain.

\subsection{Bubbling lemma}
In this subsection we adapt the usual bubbling argument to our
noncompact setting; cf.\ \cite[Lemma 5.11]{BEHWZ}.

Let $\mathbf{S}= (\Sigma, j,
\mathbf{m})$ be a marked Riemann surface, and let
\[
F=(a,f)\colon  (\Sigma,j,\mathbf{m})\to (\R\times M^*,J)
\]
be a holomorphic map as above.  Below, we write
$b= t \circ f \colon \dot{\Sigma} \to \R$. 
When the image of $F$ is
contained in the symplectization $\R \times (1, +\infty) \times
\widehat{R_+(\Gamma)}$ of the Top (or in the symplectization $\R
\times (-\infty, -1) \times \widehat{R_-(\Gamma)}$ of the Bottom),
we can write $f=(b,v)$, where $v \colon \dot{\Sigma} \to
\widehat{R_{\pm}(\Gamma)}$.  

On $\R\times M^*$, we will always use the Riemannian metric
\begin{equation}
\label{eqn: metric} g= ds \otimes ds + \alpha^* \otimes \alpha^* + d
\alpha^* ( \cdot, J \cdot) - d\alpha^*(J \cdot,\cdot),
\end{equation}
where $s$ is the $\R$-coordinate.  (The last term is added to
symmetrize the tensor, since we are only taking $J$ to be
$d\alpha$-positive.) With respect to this metric, $\| \nabla t \|$
is uniformly bounded.

Also recall the following (by now well-known) topological lemma.

\begin{lemma}[Hofer's lemma] \label{hofer}
Let $(X, d)$ be a complete metric space, $f \colon X \to \R$ be a
non-negative continuous function, $x_0 \in X$, and $\delta>0$. Then
there exist $x \in X$ and a positive number $\varepsilon \le \delta$
such that
\[ d(x_0,x) < 2 \delta, \quad \sup_{B_{\varepsilon}(x)} f \le 2 f(x),
\quad \varepsilon f(x) \ge \delta f(x_0),\] where $B_\varepsilon(x)$
is an open ball of radius $\varepsilon$ about $x$.
\end{lemma}

Let us write $\mathbb{D}_r=\{z\in \C~|~ |z|<r\}$, and
$\mathbb{D}=\mathbb{D}_1$.  Then we have the following:

\begin{lemma}[Bubbling] \label{rescaling}
Consider a sequence of $J$-holomorphic maps
\[ F_n = (a_n, f_n) \colon {\mathbb D} \to (\R \times M^*, J) \]
satisfying $E(F_n)<C$ and $\tau \circ F_n <C'$ for some constants $C',
C>0$. Suppose that $\| \nabla F_n(0) \| \to \infty$ as $n \to
\infty$. Then after passing to a subsequence, there exists a sequence
of points $x_n \in {\mathbb D}$ converging to $0$, and sequences of
positive numbers $c_n,R_n\to\infty$ such that $|x_n| + c_n^{-1}R_n <1$
and the rescaled maps
$$ F_n^0 \colon  {\mathbb D}_{R_n} \to (\R \times M^*, J),$$
$$ z \mapsto F_n(x_n+c^{-1}z)$$
converge in $C^{\infty}_{loc}(\C)$ to one of the following:
\begin{enumerate}
\item a nonconstant holomorphic map $F^0 \colon \C \to (\R \times M^*,
  J)$, after translating in the $s$-direction, or
\item a nonconstant holomorphic map $F^0 \colon \C \to
(\R \times \R \times \widehat{R_{\pm}(\Gamma)}, J)$, after
translating in the $s$- and $t$-directions.
\end{enumerate}
In both cases the limit map satisfies the condition $E(F^0) \le C$.
\end{lemma}

The gradients $\| \nabla F_n(0) \| $ are computed with respect to
the standard Euclidean metric on $\mathbb{D}$ and the metric on
$\R\times M^*$ given by Equation~(\ref{eqn: metric}).

\begin{proof}
Choose a sequence $\delta_n >0$ such that $\delta_n \to 0$ and
$\delta_n \| \nabla F_n(0) \| \to \infty$. Applying Hofer's lemma to
$\|\nabla F_n\|$, we obtain new sequences $x_n \in {\mathbb D}$ and
$0 < \varepsilon_n \le \delta_n$ such that $x_n \to 0$ and
\[ \sup_{|x-x_n| \le \varepsilon_n} \| \nabla F_n(x) \| \le
2 \| \nabla F_n(x_n) \|, \quad \varepsilon_n \| \nabla F_n(x_n) \|
\to \infty. \] Set $c_n=  \| \nabla F_n(x_n) \|$ and $R_n=
\varepsilon_n \| \nabla F_n(x_n) \|$. For sufficiently large $n$ we
have $|x_n|+ c_n^{-1}R_n <1$. Hence there exist rescaled maps
\[ F^0_n(z)= \left ( a_n^0(z), f_n^0(z) \right ) =
\left ( a_n(x_n+c_n^{-1}z)- a_n(x_n), f_n(x_n+ c_n^{-1}z) \right ),
\] defined on ${\mathbb D}_{R_n}$. The
sequence $\{F_n^0\}$ satisfies the following:
\begin{itemize}
\item $a^0_n(0)=0$,
\item $\| \nabla F_n^0(0) \|=1$,
\item $\displaystyle\sup_{z\in{\mathbb D}_{R_n}} \| \nabla F_n^0(z) \| \le 2$,
\item $E(F_n^0) \le E(F) \le C$.
\end{itemize}

\s We now consider two cases:

\s\n {\em Case 1.} If there is a constant $C>0$ such that $|
b_n(x_n)| \le C$, then the maps $F_n^0$ are uniformly bounded, in
the sense that for any compact set $K_1 \subset \C$ there is a
compact set $K_2 \subset \R \times M^*$ such that $ F_n^0(K_1)
\subset K_2$ for all $n$ sufficiently large. This is a consequence
of the uniform bounds on $\| \nabla F_n^0 \|$ and on $\tau \circ
F_n^0$. The Gromov--Schwarz lemma \cite[Lemma~5.1]{BEHWZ} implies
that all the derivatives of $F_n^0$ are bounded. Hence we can apply
the Arzel\`a--Ascoli theorem and extract a subsequence which
converges in $C^{\infty}_{loc}(\C)$ to a finite energy plane
\[ F^0 \colon \C \to (\R \times M^*, J). \]
(In the rest of the paper, the Gromov--Schwarz lemma and the
Arzel\`a--Ascoli theorem will be used repeatedly without specific
mention.) The limiting map $F^0$ is nonconstant since $\| \nabla
F_n^0(0) \|=1$ for all $n$.

\s\n {\em Case 2.} Suppose that $b_n(x_n)$ is unbounded. Then,
without loss of generality, we can assume that $\lim \limits_{n \to
+ \infty} b(x_n)= + \infty$ and that there exists a sequence $R_n'
\le R_n$ such that
\[
\lim \limits_{n \to \infty} R_n' = + \infty \quad  \text{and} \quad
F_n^0(\mathbb{D}_{R_n'}) \subset \R \times (1, \infty) \times
\widehat{R_+(\Gamma)} \subset \R \times M^*.
\]
Therefore $F_n^0$ can be viewed as a map
\[
F_n^0 = (a_n^0, b_n^0, v_n^0) \colon
\mathbb{D}_{R_n'} \to \R \times \R \times \widehat{R_+(\Gamma)}.
\]
If we define
\[
\widetilde{F}_n^0(z) = (a_n^0(z)-a_n^0(0), b_n^0(z)- b_n^0(0),
v_n^0(z)),
\]
then the uniform bound on the gradient implies that for any compact
set $K \subset \C$ there is a positive constant $C$ such that
$\widetilde{F}_n^0(K) \subset [-C,C] \times [-C,C] \times
\widehat{R_+(\Gamma)}$. Hence there is a subsequence which converges
in $C^{\infty}_{loc}(\C)$ to a nonconstant finite energy plane
\[
\widetilde{F}^0 \colon \C \to (\R \times \R \times
\widehat{R_+(\Gamma)},J).
\]
This completes the proof of Lemma~\ref{rescaling}.
\end{proof}

\subsection{Bound on the $t$-coordinate}
In this section we discuss the bound on the $t$-coordinate and the
removal of singularities.

\subsubsection{Gradient bound for a single curve}

We start this subsection with the following useful lemma.

\begin{lemma} \label{noarea} Let $(M^*, \alpha^*)$ be the completion
  of a sutured contact manifold $(M,\alpha)$, and let $J$ be a
  tailored almost complex structure on $\R\times M^*$.
\begin{enumerate}
\item If $F=(a,f) \colon \C \to (\R \times M^*, J)$ is a finite energy
holomorphic map with bounded gradient and $\int_\C f^* d \alpha^*=0$,
then $F$ is constant.
\item If $F=(a,f) \colon \C^\times=\C-\{0\} \to (\R \times M^*, J)$ is a finite
energy holomorphic map with bounded gradient, $R_{\alpha^*}$ has no
closed orbits, and $\int_{\C^\times} f^* d \alpha^*=0$, then $F$ is constant.
\end{enumerate}
\end{lemma}

In (2) the gradient is computed using the flat metric on
$\C^\times$, viewed as an infinite cylinder.

\begin{proof}
  (1) The first statement is basically \cite[Lemma~28]{Ho1}, which
  goes through without modification to our noncompact case. By the
  zero $d\alpha^*$-energy condition, $\op{Im}(F)$ is contained in
  $\R\times \gamma$, where $\gamma$ is a Reeb orbit of
  $R_{\alpha^*}$. Let $\widetilde\gamma$ be the universal cover of
  $\gamma$ if $\gamma\simeq S^1$ or $\R$, and let
  $\widetilde\gamma\simeq \R$ be the extension of $\gamma$ to $M^*$ if
  $\gamma$ is an interval.  Then $F$ factors through a holomorphic map
  $\phi\colon \C \to \C=\R\times \widetilde\gamma$. Note that $\nabla
  F$ is bounded if and only if ${\bdry\phi\over \bdry z}$ is bounded
  with respect to the flat metric on both $\C$'s. It then follows that
  ${\bdry\phi\over \bdry z}$ is bounded and hence constant. Therefore
  $\phi(z)=c_0+c_1z$ for some constants $c_0,c_1$, and the
  corresponding $F$ does not have finite Hofer energy unless $c_1=0$.

(2) If $R_{\alpha^*}$ has no closed orbits, then $F$ factors through a
holomorphic map $\phi\colon  \C^\times\to \C=\R\times\widetilde\gamma$,
where $\widetilde\gamma\simeq \R$. First observe that any
holomorphic function $\phi(z)$ on $\C^\times$ can be written as a
Laurent series $\sum_{n\in\Z} a_n z^n$, $a_n\in \C$, where
$\phi_0(z)= \sum_{n\geq 1} a_n z^n$ is a holomorphic function on
$\C$ and $\phi_\infty(z)=\sum_{n\leq -1} a_nz^n$ is a holomorphic
function on $\C^\times\cup \{\infty\}$, and both $\phi_0$ and
$\phi_\infty$ have infinite radius of convergence. Next observe that
the boundedness of $\nabla F$ is equivalent to the boundedness of
${\bdry\over \bdry w} (\phi\circ g(w))= {\bdry \phi\over \bdry
z}(e^w)\cdot e^w={\bdry \phi\over \bdry z}(z)\cdot z$, where $g\colon 
\R\times [0,2\pi]\to \C^\times$ sends $w\mapsto z=e^w$, and we are
using the flat metric on $\R\times[0,2\pi]$ and $\C=\R\times
\widetilde\gamma$. It follows that ${\bdry\phi\over \bdry z}$ (and
hence ${\bdry \phi_0\over \bdry z}$) is bounded for $|z|$ large.
Hence $\bdry\phi_0\over \bdry z$ is constant and $\phi_0(z)=c_1z$.
Similarly, $\phi_\infty(z)=c_{-1}z^{-1}$. We then conclude that
$\phi(z)= c_{-1}z^{-1}+c_0+c_1z$.  The image of $\phi$ contains a
neighborhood of the point at infinity, which contradicts the finite
Hofer energy condition of $F$.
\end{proof}

The following proposition is analogous to
\cite[Proposition 27]{Ho1}, and its proof only needs some minor
changes.

\begin{prop}[Gradient bound for a single curve] \label{gradientbound}
Let $F \colon (\Sigma, j,\mathbf{m}) \to (\R \times M^*, J)$ be a
finite energy holomorphic map with bounded $\tau \circ F$.  Then
\[ \sup_{z \in \dot{\Sigma}} \rho(z) \| \nabla F(z) \| < + \infty, \]
where $\rho$ denotes the injectivity radius of the complete, finite
volume hyperbolic metric $h$ on $ \dot{\Sigma}$ which is compatible
with $j$ and $\| \nabla F(z)\|$ is measured with respect to $h$ on
$\dot{\Sigma}$ and the Riemannian metric on $\R\times M^*$ defined
in Equation~(\ref{eqn: metric}).
\end{prop}

\begin{rmk} \label{rmk: commensurate} Near a puncture, $\rho(z)\|
  \nabla F(z)\|$, calculated with respect to a complete, finite volume
  hyperbolic metric (i.e., a cusp), is commensurate to $\|\nabla
  F(z)\|$, calculated with respect to a flat metric on a
  half-cylinder. \end{rmk}

\begin{proof}
  We argue by contradiction. Suppose there is a sequence $z_n \in
  \dot{\Sigma}$ such that $$\rho(z_n) \| \nabla F(z_n) \| \to \infty$$
  as $n \to \infty$. By passing to a subsequence we may assume that
  $z_n$ converges to a puncture in $\mathbf{m}$.  Next, there exist
  holomorphic charts $\psi_n \colon \mathbb{D} \stackrel\sim\to
  {\mathcal D_n} \subset \dot{\Sigma}$ such that $\psi_n(0)=z_n$ and
  \[ C_1 \rho(z_n) \le \| \nabla \psi_n(z) \| \le C_2 \rho(z_n) \] for
  all $z\in \mathbb{D}$.  Here $C_1$ and $C_2$ are two positive
  constants that do not depend on $z_n$ and $\nabla$ is calculated
  with respect to the standard Euclidean metric on $\mathbb{D}$ and
  the hyperbolic metric on $\dot\Sigma$.  (This follows from
  Remark~\ref{rmk: commensurate}.) Setting
\[
\widetilde{F}_n = (\tilde{a}_n, \tilde{f}_n)= (a \circ \psi_n, f \circ
\psi_n),
\]
we have $\| \nabla \widetilde{F}_n(0) \| \to +\infty$ as $n \to +
\infty$.

We now apply Lemma~\ref{rescaling} to obtain the bubbling off of a
finite energy plane $\widetilde{F}^0=(\tilde{a}^0, \tilde{f}^0)$. In
both Cases (1) and (2) of Lemma~\ref{rescaling}, we have
\[
0\leq \int_{\C} (\tilde{f}^0)^* d \alpha\leq  \lim_{n \to \infty} \int_{\mathbb{D}}
(\tilde{f}_n)^* d \alpha = \lim_{n \to \infty} \int_{{\mathcal D}_n}
f^* d \alpha = 0,
\]
because the size of ${\mathcal D}_n$ is going to zero as $n$ goes to
infinity. Moreover, $\|\nabla \widetilde{F}^0\|$ is bounded by
construction. Hence $\widetilde{F}^0$ is a constant map by
Lemma~\ref{noarea}. This contradicts the property that $\| \nabla
\widetilde{F}_n^0(0) \|=1$ for all $n$.
\end{proof}

\subsubsection{Bound on $b$ for a single curve}

\begin{prop} \label{prop: bound on b for single F}
Let $F \colon \dot{\mathbb D}=\mathbb{D}-\{0\} \to (\R \times M^*,
J)$ be a finite energy $J$-holomorphic map such that $\tau \circ F$
is bounded. Then $b= t \circ F$ is bounded.
\end{prop}

\begin{proof}
Let us rewrite $F$ as
$$F=(a,f)\colon [0,\infty)\times S^1 \to (\R\times M^*,J),$$
with coordinates $(r,\theta)$ for $[0,\infty)\times S^1$. Here we
are using the flat metric on the half-cylinder and the metric on
$\R\times M^*$ given by Equation~(\ref{eqn: metric}). The gradient
bound (Proposition~\ref{gradientbound}) and Remark~\ref{rmk:
commensurate} imply a uniform bound on $|b(r,\theta) -
b(r,\theta')|$ for all $r,\theta,\theta'$.

Arguing by contradiction, suppose that $b$ is not bounded. Without
loss of generality, we may assume that $\limsup \limits_{r\to
\infty} b(r,\theta) = \infty$. By the bound on $\|\nabla F\|$, there
are increasing sequences $\kappa_n\to\infty$, $r_n^{(i)}\to\infty$,
$i=1,2,3,4$, such that:
\begin{itemize}
\item $r_n^{(1)}< r_n^{(2)}< r_n^{(3)}<r_n^{(4)}$;
\item $r_n^{(i+1)}-r_n^{(i)}\to \infty$, $i=1,2,3$;
\item $b(r_n^{(i)},0)=i\kappa_n$, $i=1,2,3,4$; and
\item $b(r,\theta)\geq 1$ for all $(r,\theta)\in [r_n^{(1)},r_n^{(4)}]
\times S^1$, i.e., $f([r_n^{(1)},r_n^{(4)}]\times S^1)$ is contained
in the Top.
\end{itemize}
Hence we may view $F|_{[r_n^{(1)},r_n^{(4)}]\times S^1}$ as a map:
$$F_n\colon [r_n^{(1)},r_n^{(4)}]\times S^1\to (\R\times \R\times
\widehat{R_+(\Gamma)},J).$$

Modulo translations in the $r$-, $s$- and $t$-directions, we can
extract a convergent subsequence of $F_n$.  However, we need to
exercise some care in order to ensure that the limiting curve is
nonconstant.

First suppose that there is a constant $c>0$ such that
$\displaystyle\sup_{r\in [r_n^{(2)}, r_n^{(3)}]}\|\nabla F_n\|\geq
c$ for all $n$. (Note that we still have an upper bound on $\|\nabla
F_n\|$.) Then, after translating in the $r$- and $\theta$-directions
and restricting the domain, we may view $F_n$ as
\[
\widetilde F_n\colon  [-R_n,R_n]\times S^1\to (\R\times \R\times
\widehat{R_+(\Gamma)},J),
\]
where $\|\nabla \widetilde F_n(0,0)\|>c$ and $R_n\to\infty$. (Note
that we are using $[r_n^{(2)}, r_n^{(3)}]\subset [r_n^{(1)},
r_n^{(4)}]$ to give ourselves extra room on both sides.)  By our
assumption that $\tau\circ F$ is bounded, we can pass to a subsequence
so that after translating in the $s$- and $t$-directions, $\widetilde
F_n$ converges in $C^\infty_{loc}$ to
\[
\widetilde F_\infty\colon  \R\times S^1\to (\R\times\R\times
\widehat{R_+(\Gamma)},J).
\]
Since $\|\nabla \widetilde
F_\infty(0,0)\|\geq c$, it follows that $\widetilde F_\infty$ is
nonconstant. Also $\|\nabla \widetilde F_\infty\|$ is bounded by
construction.  Since $\widetilde F_\infty$ has zero $d\alpha$-energy
as argued in Proposition~\ref{gradientbound}, we can apply
Lemma~\ref{noarea}(2) to obtain a contradiction.

On the other hand, suppose there is a positive sequence
$\varepsilon_n\to 0$ such that $$\displaystyle\sup_{r\in [r_n^{(2)},
r_n^{(3)}]} \|\nabla F_n\|=\varepsilon_n.$$ By shrinking the
interval $[r_n^{(2)},r_n^{(3)}]$ if necessary, we may assume that
the distance between $F_n(r_n^{(2)},0)$ and $F_n(r_n^{(3)},0)$ is
$1$ and the diameter of $Z_n=F_n([r_n^{(2)}, r_n^{(3)}]\times S^1)$
is between $1$ and $2$. Such ``long and thin'' tubes in $\R\times
\R\times \widehat{R_+(\Gamma)}$ can be eliminated by the
isoperimetric inequality and the monotonicity lemma. Here the area
is calculated with respect to the metric given in
Equation~(\ref{eqn: metric}). More precisely, by the gradient bound,
$\gamma_n^{(2)}=F_n(\{r=r_n^{(2)}\})$ has length $\leq 2\pi
\varepsilon_n$. Now recall the following well-known isoperimetric
inequality (see for example~\cite[Proposition~A.1]{Hum}):

\begin{lemma}[Isoperimetric inequality]
Let $(M,g)$ be a Riemannian manifold with bounded geometry.  Then
there exist constants $\varepsilon>0$ and $C>0$ satisfying the
following: for every  $0<r<\varepsilon$ and geodesic ball $B_r(x)$
of radius $r$, if $\gamma$ is a closed curve in $B_r(x)$ of length
$l(\gamma)$, then there is a surface $S\subset B_r(x)$ with boundary
$\gamma$ such that
$$\mbox{Area}(S) \leq C (l(\gamma))^2.$$
Here the area and length are calculated with respect to the metric
$g$.
\end{lemma}

Continuing the proof of Proposition~\ref{prop: bound on b for single
  F}: by the isoperimetric inequality, there is a surface $S_n^{(2)}$
which bounds $\gamma_n^{(2)}$ and has area $\leq K_0 \varepsilon_n^2$,
where $K_0$ does not depend on $n$. The same can be said about
$\gamma_n^{(3)}=F_n(\{r=r_n^{(3)}\})$.

We now claim that
\begin{equation} \label{eqn: upper bound for zn}
C\cdot\mbox{Area}(Z_n)\leq \mbox{Area} (S_n^{(2)}\cup S_n^{(3)})
\leq 2K_0\varepsilon_n^2,
\end{equation}
for some positive constant $C$ which is independent of $n$.  The
first inequality follows from noting that:
\begin{enumerate}
\item[(i)] $C_1\cdot\int_{S}\omega\leq \mbox{Area} (S)$ for any
  surface $S$ (Wirtinger's inequality),
\item[(ii)] $\int_{Z_n} \omega = \int_{S_n^{(2)}\cup S_n^{(3)}}\omega$ 
(since $Z_n\cup S_n^{(2)}\cup S_n^{(3)}$ is nullhomologous due to
the fact that $Z_n$ is thin), and
\item[(iii)] $C_2 \cdot
\mbox{Area}(Z_n)\leq \int_{Z_n} \omega$ (since $J$ tames $\omega$ and
$Z_n$ is holomorphic).
\end{enumerate}
Here $\omega=d(e^s \alpha)$ is the
symplectization $2$-form and (i) and (iii) work because each $Z_n$,
after translation, is contained in $0\leq s\leq 2$ by the diameter
bound.

On the other hand, since $\varepsilon_n\to 0$ and the distance
between $F_n(r_n^{(2)},0)$ and $F_n(r_n^{(3)},0)$ is fixed, there is
a constant $\delta>0$, independent of $n$, such that a ball
$B_\delta(x_n)$ of radius $\delta$ centered at some point $x_n\in
Z_n$ does not intersect the boundary of $Z_n$. Then by the
monotonicity lemma,
\begin{equation}
\mbox{Area}(Z_n\cap B_\delta(x_n))\geq K_1 \delta^2
\end{equation}
for some constant $K_1>0$ which is independent of $n$. This
contradicts Inequality~(\ref{eqn: upper bound for zn}) for
sufficiently small $\varepsilon_n$. This concludes the proof of
Proposition~\ref{prop: bound on b for single F}.
\end{proof}

\subsubsection{Removal of singularities}

We now state some corollaries of the bound on the $t$-coordinate.

\begin{cor}[Removal of singularities for Top/Bottom] \label{removal}
Every finite energy holomorphic map
\[ F=(a,f,v) \colon \dot{\mathbb{D}}= \{ z \in \C ~|~ 0 < |z| < 1 \}
\to (\R \times \R \times \widehat{R_+(\Gamma)},J) \] with $\tau
\circ f$ bounded, extends to a finite energy holomorphic map
$$\overline{F} \colon \mathbb{D} \to (\R \times \R \times
\widehat{R_+(\Gamma)}, J).$$
\end{cor}

\begin{proof}
Since $b$ is bounded by Proposition~\ref{prop: bound on b for single
F}, the usual argument for a symplectization applies: either $F$
approaches a closed orbit of the Reeb vector field as $|z|\to 0$, or
the singularity is removable. Since there are no closed orbits on
$\R\times\R\times\widehat{R_+(\Gamma)}$, the result follows.
\end{proof}

\begin{cor} \label{removal-global}
Let $F=(a, f) \colon (\Sigma, j,\mathbf{m})  \to (\R \times M^*, J)$
be a finite energy $J$-holomorphic map with $\tau\circ F$ bounded. Then the set
of punctures $\mathbf{m}$ can be written as $\mathbf{m}^+ \sqcup
\mathbf{m}^- \sqcup \mathbf{m}^r$, where:
\begin{itemize}
\item for any $z_+ \in \mathbf{m}^+$ we have $\lim \limits_{z \to z_+}
  a(z) = + \infty$ and $\limsup \limits_{z \to z_+} |b(z)| < +
  \infty$;
\item for any $z_- \in \mathbf{m}^-$ we have $\lim \limits_{z \to z_-}
  a(z) = - \infty$ and $\limsup \limits_{z \to z_-} |b(z)| < +
  \infty$;
\item for any $z_r  \in \mathbf{m}^r$ the singularity is removable.
\end{itemize}
\end{cor}

\subsection{Bounds for sequences of holomorphic curves}

To extend the SFT and ECH compactness theorems to our situation, we
need uniform bounds on the $t$ coordinate for sequences of holomorphic
curves.

\subsubsection{Gradient bound for a sequence}

We start with the following lemma which gives a gradient bound
for a sequence of holomorphic maps. The proof is similar to the
proof of Proposition~\ref{gradientbound} and to \cite[Section
10.2.1]{BEHWZ}.

\begin{lemma}\label{gradbound}
Let $F_n= (a_n, f_n) \colon (\Sigma_n, j_n,\mathbf{m}_n) \to (\R
\times M^*, J)$ be a sequence of $J$-holomorphic maps such that
there exists $C>0$ with $E(F_n)<C$ and $|\tau\circ F_n|<C$. Then we
can remove finite sets $\mathbf{m}_n^0$ from $\Sigma_n \setminus
\mathbf{m}_n$ so that the sequence
$$F_n \colon (\Sigma_n \setminus (\mathbf{m}_n \cup  \mathbf{m}_n^0),j_n)
\to (\R\times M^*,J)$$ satisfies the bound
\begin{equation} \label{gradientbound.eq}
{\rho_n(x)} \| \nabla F_n (x) \| \le C, \quad \forall x \in \Sigma_n
\setminus (\mathbf{m}_n \cup  \mathbf{m}_n^0),
\end{equation}
where the norm of gradient is computed with respect to the unique
complete, finite volume hyperbolic metric which is compatible with
$j_n$ on $ \Sigma_n \setminus (\mathbf{m}_n \cup  \mathbf{m}_n^0)$,
and with respect to the metric on $\R \times M^*$ given by
Equation~(\ref{eqn: metric}).
\end{lemma}

\begin{proof}
Suppose there is a sequence $z_n \in \Sigma_n \setminus
\mathbf{m}_n$ such that $\rho_n(z_n) \| \nabla F_n(z_n) \| \to
\infty$ for $n \to \infty$.  There exist holomorphic charts $\psi_n
\colon \mathbb{D} \stackrel\sim\to {\mathcal D_n} \subset \Sigma_n
\setminus \mathbf{m}_n$ such that $\psi_n(0)=z_n$ and
\[ C_1 \rho_n(z_n) \le \| \nabla \psi_n(z) \| \le C_2 \rho_n(z_n) \]
for all $z\in \mathbb{D}$.  Here $C_1$ and $C_2$ are two positive
constants that do not depend on $z_n$ and $\nabla$ is calculated
with respect to the standard Euclidean metric on $\mathbb{D}$ and
the complete hyperbolic metric on $\Sigma_n \setminus \mathbf{m}_n$.
Setting
\[ \widetilde{F}_n = (\tilde{a}_n, \tilde{f}_n)=
(a_n \circ \psi_n, f_n \circ \psi_n), \]
we have $\| \nabla \widetilde{F}_n(0) \| \to +\infty$ as $n \to +
\infty$.

By Lemma~\ref{rescaling} we obtain the bubbling off of a nonconstant
finite energy plane $\widetilde{F}^0 \colon \C \to (\R \times M^*,
J)$ or $\widetilde{F}^0 \colon \C \to (\R \times \R\times
\widehat{R_\pm(\Gamma)}, J)$. The latter cannot happen because a
nonconstant finite energy plane in $\R \times \R \times
\widehat{R_{\pm}(\Gamma)}$ would extend to a nonconstant holomorphic
sphere by Corollary~\ref{removal}. (Note that there are no closed
orbits in $\R\times \widehat{R_\pm(\Gamma)}$.) This is a
contradiction since the symplectic form on $\R \times \R \times
\widehat{R_{\pm}(\Gamma)}$ is exact. Also observe that the finite
energy plane $\widetilde{F}^0$ is positively asymptotic to a closed
Reeb orbit because there are no nonconstant holomorphic spheres in
$(\R \times M^*, J)$.

In order to achieve the gradient bound given by
Equation~(\ref{gradientbound.eq}), we add marked points in the
bubbling neighborhoods as in \cite[Subsection 10.2.1]{BEHWZ}. Since
there is a uniform lower bound on the areas of finite energy planes,
we only need a finite set $\mathbf{m}^0_n$.
\end{proof}

\subsubsection{Bound on $b_n$, assuming topological complexity bound}

We now prove the following bound on $b_n$, provided we have bounds
on the energy and genus (and number of marked points).

\begin{prop} \label{noside} Let $F_n = (a_n, f_n) \colon (\Sigma_n,
  j_n,\mathbf{m}_n) \to (\R \times M^*, J)$ be a sequence of
  holomorphic maps with uniform upper bounds on $|\tau \circ F_n|$,
  the energy $E(F_n)$, and the ``topological complexity''
  $g(\Sigma_n)+|\mathbf{m}_n|$. Then there is a uniform upper bound on
  $|b_n|= |t \circ f_n|$.
\end{prop}

\begin{proof}
Let $F_n$ be a sequence as in the hypothesis of
Proposition~\ref{noside}.  Arguing by contradiction, suppose the
functions $b_n$ are not uniformly bounded. Without loss of
generality we may assume that
$\displaystyle\lim_{n\to\infty}(\sup_{\dot{\Sigma}_n} b_n)= +
\infty$ for $n \to \infty$. By Lemma \ref{gradbound} we can add
marked points $\mathbf{m}_n^0$ to $\dot{\Sigma}_n=\Sigma_n\setminus
\mathbf{m}_n$ to obtain the gradient bound given by
Equation~(\ref{gradientbound.eq}) for the sequence $F_n$.

By Theorem~\ref{classical}, there is a subsequence of
$\mathbf{S}_n'= (\Sigma_n, j_n, \mathbf{m}_n \cup \mathbf{m}_n^0)$
which converges to a nodal surface $\mathbf{S}=(\Sigma, j,
\mathbf{m}, D)$. Fix $\epsilon < \frac 14 \log(1+ \sqrt{2})$ (i.e.,
$1/4$ of the constant required for the thick-thin decomposition),
and consider the cover
\[
\Sigma-(\mathbf{m}\cup D)= C_0 \cup \ldots \cup C_k,
\]
where $C_i$ is either a connected component of ${\rm
  Thick}_{2\epsilon}(\mathbf{S})$ or a ``connected component'' of
${\rm Thin}_{3 \epsilon}(\mathbf{S})$. (Here any two components of
${\rm Thin}_{3\epsilon}(\mathbf{S})$, whose corresponding marked
points in $D$ are identified, are regarded as part of the same
``connected component'' of ${\rm Thin}_{3 \epsilon}(\mathbf{S})$.)
Similarly consider the cover
\[
\Sigma_n-(\mathbf{m}_n\cup \mathbf{m}_n^0)=C_0^n\cup\ldots \cup C_k^n,
\]
where $C_i^n$ is a component of ${\rm
Thick}_{\epsilon}(\mathbf{S}'_n)$ or a ``connected component'' of
${\rm Thin}_{4\epsilon}(\mathbf{S}'_n)$, and $C_i^n$ corresponds to
$C_i$. By Proposition \ref{radius}, for sufficiently large $n$,
$C_i$ is contained in a component $C^n_i$ for all $i$.

Now define
\[
\Delta_n(C_i^n)= \sup_{C_i^n} b_n - \inf_{C_i^n} b_n.
\]
Since $\lim \limits_{n \to \infty} (\sup \limits_{\dot\Sigma_n} b_n)
= + \infty$ and the ends of $F_n$ are asymptotic to cylinders over
Reeb orbits in $M$, it follows that $$\lim_{n\to \infty}
(\sup_{\dot\Sigma_n} b_n -\inf_{\dot\Sigma_n}b_n) = +\infty.$$ Now,
since each covering has the same finite number of components, there
must be one
--- which we call $C_0^n$ without loss of generality
--- for which $\lim \limits_{n \to \infty} \Delta_n(C_0^n) = +
\infty$. By Lemma \ref{gradbound} and Proposition \ref{radius},
$\|\nabla F_n\|$ is uniformly bounded on ${\rm
Thick}_{2\epsilon}(\mathbf{S})$.  Since the variation of $b_n$ is
bounded on the thick part due to a bound on the diameter, $C_0$ must
be a connected component of ${\rm Thin}_{3 \epsilon}(\mathbf{S})$.

By reparametrizing the component $C^n_0$ using a standard flat
cylinder, we can write $F_n$ on $C^n_0$ as:
$$F_n\colon  [0,r_n]\times S^1 \to (\R \times \R \times
\widehat{R_+(\Gamma)},J),$$ where $\|\nabla F_n\|$ is uniformly
bounded by Lemma~\ref{gradbound}, in view of Remark~\ref{rmk:
commensurate}. This uniform bound implies that $\op{Im} F_n$ has
bounded diameter (independent of $n$) when restricted to any circle
$\{r=\mbox{const}\}$.

The rest of the proof is as in Proposition~\ref{prop: bound on b for
single F}. There exist $\kappa_n\to\infty$ and $r_n^{(i)}\to
\infty$, $i=1,2,3,4$, such that:
\begin{itemize}
\item $0<r_n^{(1)}< r_n^{(2)}< r_n^{(3)}<r_n^{(4)}<r_n$;
\item $r_n^{(i+1)}-r_n^{(i)}\to \infty$, $i=1,2,3$;
\item $b_n(r_n^{(i+1)},0)-b_n(r_n^{(i)},0)=\kappa_n$, $i=1,2,3$;
\footnote{Note that,
unlike the corresponding condition for Proposition~\ref{prop: bound
on b for single F}, we are taking the difference of the $b_n$
values.} and
\item $b_n(r,\theta)\geq 1$ for all $(r,\theta)\in [r_n^{(1)},r_n^{(4)}]
\times S^1$, i.e., $f([r_n^{(1)},r_n^{(4)}]\times S^1)$ is contained
in the Top.
\end{itemize}
If $\sup\|\nabla F_n\|$ is bounded below by $c>0$ on
$[r_n^{(2)},r_n^{(3)}]\times S^1$, then, after restricting the
domain of $F_n$ and translating in the $r$- and $\theta$-directions,
we obtain:
$$ \widetilde F_n\colon  [-R_n,R_n]\times S^1\to (\R \times \R \times
\widehat{R_+(\Gamma)},J),$$ where $\|\nabla \widetilde
F_n(0,0)\|\geq c$ and $R_n\to \infty$. The limit curve
$$\widetilde F_\infty\colon  \R\times S^1\to (\R \times \R
\times \widehat{R_+(\Gamma)},J),$$ is a nonconstant holomorphic
curve. By Corollary~\ref{removal}, we can extend this function to a
nonconstant holomorphic sphere in $(\R \times \R \times
\widehat{R_+(\Gamma)},J)$ and obtain a contradiction.  On the other
hand, if $\sup \|\nabla F_n\|\to 0$ on $[r_n^{(2)},r_n^{(3)}]\times
S^1$, then we can eliminate the ``long and thin'' tubes in $\R\times
\R\times \widehat{R_+(\Gamma)}$ as in Proposition~\ref{prop: bound
on b for single F}.
\end{proof}

\begin{cor}
\label{cor:SFTcompactness}
Let $M^*$ be the completion of a sutured contact manifold $M$ and let
$J$ be a tailored almost complex structure on the symplectization
$\R\times M^*$, as usual.  Then the SFT compactness theorem
\cite[Theorem 10.1]{BEHWZ} holds for $J$-holomorphic curves in
$\R\times M^*$ whose punctures are asympotic to Reeb orbits.
\end{cor}

\begin{proof}
  We need to show that any sequence in
  $\mathcal{M}_g(\underline{\gamma};\underline{\gamma}';J)$ has a
  subsequence which converges to a holomorphic building in the sense
  of \cite{BEHWZ}.  By Lemma~\ref{lemma: bound on tau}, there is a
  uniform upper bound on $\tau$ for the curves in the sequence.  By
  \cite[Prop.\ 5.13]{BEHWZ} there is a uniform upper bound on the
  Hofer energy of the curves in the sequence.  By
  Proposition~\ref{noside} there is then a uniform upper bound on $t$.
  Thus the projections of all the holomorphic curves in the sequence
  to $M^*$ are contained in a compact set, and the rest of the
  argument in \cite{BEHWZ} carries over.
\end{proof}

\subsubsection{Bound on $b_n$ in dimension four}

We turn now to the compactness theorem for ECH.  For this purpose we
will prove the bound on $b_n$ without any constraints on the genus,
but assuming that $\R\times M^*$ has dimension four. The proof is
based on a version of Gromov compactness due to Taubes which uses
currents and does not assume any genus bounds; see
\cite[Proposition~3.3]{T3} and \cite[Lemma~9.8]{Hu1}.

We recall some basic terminology from ECH.  An {\em orbit set\/} is a
finite set of pairs $\{(\gamma_i,m_i)\}$, where the $\gamma_i$'s are
distinct embedded Reeb orbits, and the $m_i$'s are positive integers.
In the terminology of \cite{Hu1}, a {\em flow line} from the orbit set
$\{(\gamma_i,m_i)\}$ to the orbit set $\{(\gamma_j',m_j')\}$ is a
finite energy holomorphic curve $F\colon (\Sigma,j,\mathbf{m})\to
(\R\times M^*,J)$ such that:
\begin{enumerate}
\item $F$ is an embedding, except perhaps for repeated $\R$-invariant
cylinders which do not intersect the other components of $F$.
\item $F$ has positive punctures at covers of $\gamma_i$ with total
  multiplicity $m_i$, negative punctures at covers of $\gamma_j'$ with
  total multiplicity $m_j'$, and no other punctures.
\end{enumerate}

\begin{prop} \label{prop: ECH top and bottom bound} Suppose $\dim
  (\R\times M^*)=4$. Let $F_n = (a_n, f_n) \colon (\Sigma_n,
  j_n,\mathbf{m}_n) \to (\R \times M^*, J)$ be a sequence of flow
  lines from $\{(\gamma_i,m_i)\}$ to $\{(\gamma_j',m_j')\}$. Then
  there are uniform upper bounds on $|\tau\circ f_n|$ and $|b_n|= |t
  \circ f_n|$.
\end{prop}

\begin{proof}
  The bound on $\tau$ follows from Lemma~\ref{lemma: bound on tau}.
  To prove the bound on $t$, suppose on the contrary that there is a
  sequence of flow lines $F_n$ with $b_n$ unbounded.  Without loss of
  generality there exist $x_n\in \dot\Sigma_n$ such that $b_n(x_n)\to
  +\infty$.  Now consider the restriction
\[
F_n'\colon \Sigma_n'\to \R\times [1,\infty)\times\widehat{R_+(\Gamma)}
\]
of $F_n$ where
\[
\Sigma_n'=\{x\in \dot\Sigma_n~|~ f_n(x)\in
[1,\infty)\times\widehat{R_+(\Gamma)}\}.
\]
Let $C_n'$ be the
holomorphic subvariety obtained by translating $F_n'(\Sigma_n')$ by
$a_n(x_n)$ in the $s$-direction and by $b_n(x_n)$ in the
$t$-direction.  (From now on, we will not distinguish between
holomorphic maps and their images, viewed as currents.) We then set
\[
C_n=C_n'\cap ([-c_n,c_n]\times[-d_n,d_n]\times \widehat{R_+(\Gamma)}),
\]
where $c_n,d_n\to \infty$ and $0<d_n\ll b_n(x_n)$.  Note that $C_n$
passes through $\{(0,0)\}\times \widehat{R_+(\Gamma)}$.  We may
assume without loss of generality that $\int_{C_n} d\alpha^*\to 0$.

By the Gromov compactness theorem via
currents~\cite[Proposition~3.3]{T3}, we can pass to a subsequence so
that $C_n$ converges weakly as currents in $(\R\times \R\times
\widehat{R_+(\Gamma)},J)$ to a proper $J$-holomorphic subvariety $C$,
so that, for any compact set $K\subset \R\times \R\times
\widehat{R_+(\Gamma)}$,
\begin{equation}
\sup_{x\in C_n\cap K} \mbox{dist}(x,C)+\sup_{x\in C\cap
K}\mbox{dist}(x,C_n)\to 0,
\end{equation}
as $n\to \infty$. More precisely, for any compact set
$K\subset\R\times \R\times\widehat{R_+(\Gamma)}$, we can pass to a
subsequence so that the intersections of the curves $C_n$ with $K$
converge to a $J$-holomorphic subvariety in $K$, using the fact that
there is a uniform upper bound on the integral of the symplectic form
$d(e^s\alpha^*)$ over $C_n\cap K$.  An exhaustion argument then gives
a subsequence converging on all of $\R\times
\R\times\widehat{R_+(\Gamma)}$ as above.

We claim now that $d\alpha^*|_C=0$.  To see this, let $p\in C$ and let
$\varphi\colon \R\times M\to [0,1]$ be a compactly supported smooth
function with $\varphi(p)=1$.  Since $\int_{C_n}d\alpha^*\to 0$ and
$d\alpha^*|_{C_n}\ge 0$ on all of $C_n$, we have $\int_{C_n}\varphi
d\alpha^*\to 0$.  Since $C_n$ converges to $C$ as functionals on
compactly supported $2$-forms, we obtain $\int_C\varphi d\alpha^*=0$.
Since $d\alpha^*|_C\ge 0$ on all of $C$, we conclude that
$d\alpha^*|_C$ vanishes on a neighborhood of $p$.  This proves the
claim.

It follows now that $C$ is supported on $\R\times \gamma$, where
$\gamma$ is a Reeb orbit.  Note that $\gamma$ is not a closed orbit,
and instead is a line. Now $C$ covers all of $\R\times \gamma$ by the
properness of $C$, and the fact that holomorphic maps are open. On the
other hand, $\R\times\gamma$ has infinite Hofer energy, while there is
a uniform upper bound on the Hofer energy of $C_n$ by \cite[Prop.\
5.13]{BEHWZ}. This contradicts the weak convergence of $C_n$ to $C$.
\end{proof}

\begin{cor}
\label{cor:ECHCompactness}
Suppose $\dim (\R\times M^*)=4$.  Then the ECH compactness theorem
\cite[Lemma~9.8]{Hu1} holds for $J$-holomorphic curves in the
symplectization of the completion of a sutured contact manifold,
provided that we choose the almost complex structure $J$ on $\R
\times M^*$ to be tailored to $(M^*, \alpha^*)$ in the sense of
Section~\ref{subsection: almost complex structure}.
\end{cor}

\section{Definition of the sutured contact homology 
and sutured ECH}
\label{section: defn of sutured CH}

We now use the Gromov compactness established in the previous section
to define the sutured contact homology and sutured ECH and prove
Theorem~\ref{thm: well-definition}.

\subsection{Definition of sutured contact homology}
\label{sec:DSCH}

Let $(M,\Gamma, U(\Gamma), \xi)$ be a sutured contact manifold and
$\alpha$ be an adapted contact form for $\xi$. Let $(M^*, \alpha^*)$
be the completion of $(M,\alpha)$ and $J$ be an almost complex
structure on $\R\times M^*$ which is tailored to $(M^*,\alpha^*)$.
Since all the periodic orbits of $R_{\alpha^*}$ are contained in $M$,
by performing a small perturbation of $\alpha^*$ supported in $M$ we may
assume that $\alpha^*$ is {\em nondegenerate}, i.e., all the
periodic orbits of $R_{\alpha^*}$ are nondegenerate. 

We define the {\em sutured contact homology}
$HC(M,\Gamma,\alpha,J)$ to be the contact homology of
$(M^*,\alpha^*,J)$ as follows: A periodic orbit of the Reeb vector
field $R_{\alpha^*}$ is said to be {\em good} it does not cover a
simple orbit $\gamma$ an even number of times, where the first
return map $\xi_{\gamma(0)}\rightarrow \xi_{\gamma(T)}$ has an odd
number of eigenvalues in the interval $(-1,0)$. Let
$\mathcal{P}(\alpha)$ be the set of good periodic orbits $\gamma$ of
$R_{\alpha^*}$.  The contact homology chain complex
$\mathcal{A}(\alpha,J)$ is the free supercommutative $\Q$-algebra
with unit generated by elements of $\mathcal{P}(\alpha)$, where the
grading and the boundary map $\bdry\gamma$ are defined in the usual
way (as in \cite{EGH}) with respect to the $\alpha^*$-adapted almost
complex structure $J$. The homology of $\mathcal{A}(\alpha,J)$ is
the sutured contact homology algebra $HC(M,\Gamma,\alpha,J)$.

It follows from Corollary~\ref{cor:SFTcompactness} that the necessary
Gromov compactness holds to show that the differential $\bdry$ is
well-defined and $\bdry^2=0$.  Namely, if $\gamma$ is a periodic
orbit, then there are only finitely many collections of negative ends
with total action less than that of $\gamma$.  Hence $\bdry\gamma$
counts holomorphic curves in the quotients by the $\R$-action of index
$1$ moduli spaces $\mathcal{M}_0(\gamma;\gamma'_1,\dots,\gamma'_l)$,
where we range over finitely many $(\gamma'_1,\dots,\gamma'_l)$.  If
these moduli spaces are cut out transversely, then it follows from
Corollary~\ref{cor:SFTcompactness} that $\bdry\gamma$ is a finite
count of holomorphic curves.  Similarly, the proof that $\bdry^2=0$
involves considering the boundaries of quotients by the $\R$-action of
index $2$ moduli spaces
$\mathcal{M}=\mathcal{M}_0(\gamma;\gamma'_1,\dots,\gamma'_l)$, where
for any given $\gamma$ there are only finitely many possibilities for
$\underline{\gamma}'$.  If these moduli spaces are cut out
transversely, then it follows from Corollary~\ref{cor:SFTcompactness}
that $\bdry^2$ counts points in the boundary of a compact
$1$-manifold.

\s\n {\bf Disclaimer.}  Already for closed contact manifolds, it is
usually not possible to choose $J$ so that all of the above moduli
spaces are cut out transversely.  This problem arises because of
multiply covered holomorphic curves of negative index.  Thus in
general the differential $\bdry$ needs to be defined as a count of
points in some abstract perturbation of the moduli space of index $1$
holomorphic curves.  Even in a lucky situation where all relevant
moduli spaces of holomorphic curves are cut out transversely, one
still needs some abstract perturbations to define the chain homotopies
needed to prove that the contact homology is independent of the choice
of contact form and almost complex structure.  This problem arises
because in a generic $1$-parameter family of data there can be
holomorphic buildings with repeated index $-1$ curves.

The necessary abstract perturbations to solve the above problems in
the closed case are a work in progress by Hofer-Wysocki-Zehnder (see
\cite{Ho3} for an overview), and are expected to carry over directly
to the sutured case.  But strictly speaking Theorem~\ref{thm:
  well-definition} should be regarded as a conjecture until this work
has been completed.

\bigskip

On the other hand, transversality for {\em somewhere injective\/}
holomorphic curves in $\mathcal{M}_g(\m{\gamma};\m{\gamma}';J)$ can be
achieved by taking $J$ to be generic inside $M$, while keeping it
tailored.  In fact, the transversality argument in \cite{Dr} carries
over directly to the sutured case. In particular, it suffices to
perturb $J$ arbitrarily near the periodic orbits in order to attain
transversality for somewhere injective curves.

\subsection{Invariance of the contact homology algebra}
\label{subsection: ac}

Modulo the above disclaimers, we now prove the following proposition,
which will complete the proof of Theorem~\ref{thm:
  well-definition}(1).  Below we suppress the (not yet defined)
abstract perturbations from the discussion.

\begin{prop}
Let $(M,\Gamma,\xi)$ be a sutured contact manifold.  
\begin{enumerate}
\item The contact homology algebra $HC(M,\Gamma,\alpha,J)$ does not
  depend on the choice of adapted contact form $\alpha$ with
  $\operatorname{Ker}(\alpha)=\xi$ or tailored almost complex
  structure $J$, and so we can denote it by $HC(M,\Gamma,\xi)$.
\item More generally, a one-parameter family of contact strutures
  $\{\xi_t\mid t\in[0,1]\}$ which are the kernels of a one-parameter
  family $\{\alpha^\lambda\mid \lambda\in[0,1]\}$ of adapted contact
  forms on $(M,\Gamma,U(\Gamma))$ induces an isomorphism
  $HC(M,\Gamma,\xi^0)\simeq HC(M,\Gamma,\xi^1)$ which depends only on
  the homotopy class of the path $\{\xi_t\}$.
\end{enumerate}
\end{prop}

\begin{proof}
  Let $\alpha^0$ and $\alpha^1$ be two contact $1$-forms which are
  adapted to $(M,\Gamma,U(\Gamma))$, and are connected by a
  $1$-parameter family $\alpha^\lambda$, $\lambda\in[0,1]$, of adapted
  contact $1$-forms; also let $(\alpha^\lambda)^*$ be the completion
  of $\alpha^\lambda$ to $M^*$. Note that we are not assuming that
  $\ker\alpha^0=\ker\alpha^1$, only that they are isotopic. Let
  $J^\lambda$, $\lambda\in[0,1]$, be an almost complex structure on
  $\R\times M^*$ which is tailored to $(M^*,(\alpha^\lambda)^*)$. In
  particular, the projection $J_0^\lambda$ of $J^\lambda$ to
  $(\widehat{R_\pm(\Gamma)},\widehat{\beta}_\pm^\lambda)$ is
  $\widehat{\beta}_\pm^\lambda$-adapted.  Here
  $\widehat{\beta}^\lambda_\pm$ is the completion of the Liouville
  $1$-form $\alpha^\lambda|_{R_\pm(\Gamma)}$ so that the Liouville
  vector field $Y^\lambda=\bdry_\tau$ for $\tau\geq 0$; let us also
  write $(\beta^\lambda_\pm)_0$ for the restriction of
  $\widehat{\beta}^\lambda_\pm$ to $\bdry R_\pm(\Gamma)$.  We now
  define an isomorphism
  $HC(M,\Gamma,\alpha^0,J^0)\stackrel{\simeq}{\to}
  HC(M,\Gamma,\alpha^1,J^1)$.

  \s\n {\bf Step 1.} First consider the case when
  $\widehat{\beta}^\lambda_\pm$ and $J^\lambda_0$ are independent of
  $\lambda$ on the region where $\tau\geq 0$.  We then define a chain
  map
\[
\Phi\colon  \mathcal{A}(\alpha^0,J^0) \to
\mathcal{A}(\alpha^1,J^1).
\]
as follows.
Let $\phi\colon \R\to[0,1]$ be a
smooth nonincreasing function with $\phi(s)=1$ for $s\leq  -N$ and
$\phi(s)=0$ for $s\geq N$, where $N\gg 0$. On $\R\times M^*$ with
coordinates $(s,y)$, define the almost complex structure $\tilde{J}$
so that $\tilde{J}(s,y)=J^{\phi(s)}(s,y)$. Let
$\mathcal{M}_g(\m\gamma;\m\gamma';\tilde J)$ be the moduli space of
genus $g$ finite energy holomorphic maps $F\colon 
(\Sigma,j,\mathbf{m})\to (\R\times M^*,\tilde J)$ with positive ends
$\m\gamma$ which are periodic orbits of $R_{(\alpha^0)^*}$ and
negative ends $\m\gamma'$ which are periodic orbits of
$R_{(\alpha^1)^*}$. Then the chain map $\Phi(\gamma)$ counts
elements of index zero moduli spaces
$\mathcal{M}=\mathcal{M}_0(\gamma;\gamma_1',\dots,\gamma_k';\tilde
J)$. Note that the almost complex structure $\tilde
J$ is tamed by the symplectic form $d(e^s\alpha^{\phi(s)})$,
provided $|{d\phi\over ds}|$ is sufficiently small for all $s$.
Moreover, $\tilde J$ is $\alpha^0$-adapted for $s\geq N$ and
$\alpha^1$-adapted for $s\leq -N$.

We claim that all the curves in
$\mathcal{M}_0(\gamma;\ldots;\tilde{J})$, when projected to $M^*$, are
contained inside a compact subset of $M^*$, so that they satisfy the
Gromov compactness needed to show that $\Phi$ is a well-defined chain
map. Since the projection $J_0^{\phi(s)}$ of $\tilde{J}$ is
$s$-invariant on $\tau\geq 0$, it follows that no such curve enters
the region $\tau\geq 0$. Now, if there is a sequence of curves $F_n\in
\mathcal{M}_0(\gamma;\ldots;\tilde{J})$ and $z_n\in \dot\Sigma$ such
that $t\circ F_n(z_n)\to\infty$, then an argument similar to the proof
of Proposition~\ref{noside} implies the existence of a nonconstant
finite energy holomorphic map to $\R\times \R\times R_\pm(\Gamma)$,
either with respect to $\tilde J$ or with respect to $J^0$ or $J^1$.
In any case, since there are no periodic orbits inside $\R\times
\R\times R_\pm(\Gamma)$, we have a contradiction.

Arguing as usual, we can prove that $\Phi$ has a homotopy inverse
$\Psi$, so that $\Phi$ induces an isomorphism on homology.

\s\n {\bf Step 2.} Next suppose that $J^0$ and $J^1$ do not agree on
the ends $\tau\geq 0$.   We define an intermediate almost complex
structure $J^2$ together with a $1$-form $(\alpha^2)^*$ on $M^*$ so
that there are isomorphisms $HC(\alpha^0,J^0)\simeq
HC((\alpha^2)^*,J^2)$ and $HC((\alpha^2)^*, J^2)\simeq
HC(\alpha^1,J^1)$.

The proof of Lemma~\ref{lemma: interpol} and Corollary~\ref{cor:
  plurisub} shows that there exist an almost complex structure $J_0^2$
and $1$-forms $\widehat\beta^2_\pm$ on $\widehat{R_\pm(\Gamma)}$ which
satisfy the following:
\begin{itemize}
\item Where $\tau\gg 0$, the $1$-form $\widehat\beta^2_\pm$
agrees with $(\beta^0_\pm)_0$, and the almost complex structure
$J_0^2$ is $(\beta^0_\pm)_0$-adapted.
\item Where $\tau\leq 0$ we have $J_0^2=J_0^1$ and
  $\widehat\beta^2_\pm=\widehat\beta^1_\pm$;
\item Where $\tau\geq 0$, some increasing function $u(\tau)$ is
  plurisubharmonic with respect to $J_0^2$;
\item $\widehat\beta^2_+=\widehat\beta^2_-$ for $\tau\geq 0$, where
$\widehat{R_\pm(\Gamma)}-int(R_\pm(\Gamma))$ are naturally
identified.
\end{itemize}
In particular, no holomorphic map from a punctured Riemann surface to
$([0,\infty)\times \bdry
R_\pm(\Gamma), J_0^2)$ has a local maximum of $\tau$ in
the interior of the domain.

The $1$-form $(\alpha^2)^*$ on $M^*$ is defined as follows:
\begin{itemize}
\item $(\alpha^2)^*=\alpha^1$ on $M$;
\item $(\alpha^2)^*=Cdt+\widehat\beta^2_\pm$ on $M^*-int(M)$.
\end{itemize}
The almost complex structure $J^2$ on $\R\times M^*$ is chosen so that:
\begin{itemize}
\item Conditions (A$_0$) (with respect to the $1$-form
$(\alpha^2)^*$) and (A$_1$) from Section~\ref{subsection: almost
complex structure} hold;
\item the projection of $J^2$ to $\widehat{R_\pm(\Gamma)}$
is $J_0^2$;
\item $J^2=J^1$ on $\R\times\{\tau\leq 0\}$.
\end{itemize}

We then apply Step 1 to obtain a chain map
\[
\Phi_1\colon  \mathcal{A}(\alpha^0,J^0)\to
\mathcal{A}((\alpha^2)^*,J^2),
\]
which is a quasi-isomorphism.

On the other hand, since $J_0^1$ and $J_0^2$ agree on $R_\pm(\Gamma)$
and $\tau\circ F$ does not attain a local maximum for any holomorphic
curve $F$ where $\tau>0$, it follows that every holomorphic curve
counted in $\bdry$ for $J_0^2$ lies inside $\R\times \{\tau\leq 0\}$.
This implies that
$\mathcal{A}((\alpha^2)^*,J^2)=\mathcal{A}(\alpha^1,J^1)$ as chain
complexes.  Hence we obtain an isomorphism
\begin{equation} \label{eqn:iso}
HC(M,\Gamma,\alpha^0,J^0)\stackrel{\simeq}{\to} HC(M,\Gamma,\alpha^1,J^1).
\end{equation}

\s\n {\bf Step 3.}  To complete the proof of the proposition, we need
to show that the isomorphism \eqref{eqn:iso} is canonical when
$\xi^\lambda$ is independent of $\lambda$, and otherwise depends only
on the homotopy class of the path $\{\xi^\lambda\}$.

First consider the situation where $M$ is closed and $\alpha_0$,
$\alpha_1$ are contact $1$-forms which are homotopic through contact
$1$-forms $\alpha_\rho$, $\rho\in[0,1]$.  We can use the homotopy to
construct a cobordism $(\R\times M,J)$, which gives rise to the chain
map $\Phi\colon
\mathcal{A}(\alpha_0,J_0)\to\mathcal{A}(\alpha_1,J_1)$, where $J_i$ is
adapted to $\alpha_i$. Now, if there are two homotopies $\alpha_\rho$,
$\alpha'_\rho$ from $\alpha_0$ to $\alpha_1$ which are homotopic, then
there is a homotopy of cobordisms from $(\R\times M,J)$ to $(\R\times
M,J')$, and the usual chain homotopy argument implies that the induced
isomorphisms $\Phi$, $\Phi'$ agree.  In other words, the map
$\Phi\colon HC(M,\alpha_0,J_0)\to HC(M,\alpha_1,J_1)$ only depends on
the homotopy class of paths connecting $\alpha_0$ and $\alpha_1$;
however, the map will likely depend on the choice of homotopy
class. On the other hand, when we have two contact $1$-forms
$\alpha_0$ and $\alpha_1$ for the same contact structure $\xi$, we can
write $\alpha_1=f_1\alpha_0$, and there is a canonical homotopy class
of paths from $\alpha_0$ to $\alpha_1$, namely one which has the form
$\alpha_\rho=f_\rho\alpha_0$. Hence, the identification $\Phi\colon
HC(M,\alpha_0,J_0)\to HC(M,f_1\alpha_0,J_1)$ is canonical.

Returning to the sutured case, suppose $\alpha^0$ and $\alpha^1$
are adapted to the sutured contact manifold $(M,\Gamma,
U(\Gamma),\xi)$. We claim that the contact homology algebras
$HC(M,\Gamma,\alpha^0,J^0)$ and $HC(M,\Gamma,\alpha^1,J^1)$ are
canonically isomorphic. Since $\alpha^0$ and $\alpha^1$ are
contact forms for the same contact structure $\xi$, the forms
$\alpha^1$ and $\alpha^0$ are conformally equivalent. Consequently,
$(\beta^1_\pm)_0$ and $(\beta^0_\pm)_0$ differ by a constant
multiple. Any two almost complex structures $J$ constructed in the
proof of Lemma~\ref{lemma: interpol} are connected by a
$1$-parameter family of almost complex structures with the same
properties. Hence there is a $1$-parameter family of chain maps
$(\Phi_1)_\rho\colon  \mathcal(\alpha^0,J^0)\to
\mathcal{A}((\alpha^2)^*_\rho,J^2_\rho)$ where
$\mathcal{A}((\alpha^2)^*_\rho,J^2_\rho)$ and
$\mathcal{A}(\alpha^1,J^1)$ are canonically isomorphic. Then, by the
discussion in the previous paragraph, the induced isomorphisms in
Equation~(\ref{eqn:iso}) agree.
\end{proof}

\subsection{Sutured embedded contact homology}
\label{subsection: sutured ECH}

Suppose now that $(M,\Gamma,\alpha)$ is a sutured contact manifold
where $\dim(M)=3$ and $\alpha$ is nondegenerate.  Let $J$ be a generic
tailored almost complex structure on $\R\times M^*$.  We can now
define the {\em sutured embedded contact homology\/}
$ECH(M,\Gamma,\alpha,J)$ by copying the definition in the closed case
(see e.g.\ \cite[Sec.\ 7]{HT1}) verbatim.  It follows from the
discussion at the end of Section~\ref{sec:DSCH} that for generic
tailored $J$, the moduli spaces of $J$-holomorphic curves needed to
define the ECH differential $\partial$ and prove that $\partial^2=0$
are cut out transversely.  (These curves are all somewhere injective.)
Corollary~\ref{cor:ECHCompactness} implies that the necessary
compactness holds to show that $\partial$ is defined and satisfies
$\partial^2=0$.  The gluing analysis from \cite{HT1,HT2} to complete
the proof that $\partial^2=0$ carries over unchanged.

Recall that part of Conjecture~\ref{conj for ECH} is that
$ECH(M,\Gamma,\alpha,J)$ depends only on $(M,\alpha,\xi)$.  Currently
the only known proof of the analogous statement in the closed case
uses Seiberg-Witten theory; there is no known definition of an
isomorphism in terms of holomorphic curves (due to the presence of
multiply covered curves of negative ECH index in cobordisms).  However
if such an isomorphism could be constructed, then the discussion in
Section~\ref{subsection: ac} would allow it to be extended to the
sutured case.

\section{Variations}
\label{section: variations}

In this section we define some variants of sutured contact homology
and sutured ECH.

\subsection{The ``hat'' versions of contact homology and embedded
contact homology}

Let $(M,\xi)$ be a closed contact $(2n+1)$-dimensional manifold.
Choose a contact form $\alpha$ for $\xi$, and consider a Darboux ball
of the form $B^{2n+1}=D^{2n}\times [-1,1]$ with coordinates
$(x_1,y_1,\dots,x_n,y_n,t)$ and $\alpha=dt+\sum_i{1\over
  2}(x_idy_i-y_idx_i)$ on $B^{2n+1}$. Here
$D^{2n}=\{\sum_i|x_i|^2+|y_i|^2=1\}$.  (One may need to multiply the
contact form by a large positive constant in order for such a Darboux
ball to exist.)  On $B^{2n+1}$ the Reeb vector field is given by
$R_\alpha=\bdry_t$.  In particular, $R_\alpha$ is tangent to $(\bdry
D^{2n})\times[-1,1]$ and transverse to $D^{2n}\times\{-1,1\}$.  Let
$(M(1)',\alpha|_{M(1)'})$ be the concave sutured contact manifold
obtained from $(M,\alpha)$ by removing $B^{2n+1}$. Applying the
concave-to-convex procedure described in Section~\ref{subsection:
  concave to convex} to $(M(1)',\alpha|_{M(1)'})$ then gives a convex
sutured contact manifold $(M(1),\alpha_{1})$.

Recall from Theorem~\ref{thm: mapping cone} that when $\dim(M)=3$ we have
\begin{equation}
\label{eqn:ECHhat}
\widehat{ECH}(M,\xi) \simeq ECH(M(1),\alpha_1).
\end{equation}
By analogy with this, in all odd dimensions we define a ``hat''
version of contact homology by
\begin{equation}
\label{eqn:HChat}
\widehat{HC}(M,\xi)= HC(M(1),\alpha_1).
\end{equation}
(This does not depend on $\alpha$ as shown in Section \ref{subsection: ac}.)

\subsection{A transverse knot filtration} \label{section:
knot filtration}

Let $(M,\xi)$ be a closed contact $3$-manifold and let $K \subset M$
be a null-homologous transverse knot. Since $K$ is transverse, there
exists a contact form $\alpha$ on $M$ such that $\xi=\ker\alpha$ and
$K$ is a closed orbit of $R_\alpha$. In fact, by the Darboux-Weinstein
neighborhood theorem, we can choose $\alpha$ so that there is a
neighborhood $N(K)=D^2\times [-2,2]/(-2\sim 2)$ of $K=\{r=0\}$ in
which $\alpha = dt+cr^2d\theta$.  Here $c$ is a small positive
constant, $(r,\theta,t)$ are cylindrical coordinates on $D^2\times
[-2,2]$, and $D^2=\{r\leq 1\}$. Let $(M(1)',\alpha|_{M(1)'})$ be
defined as in the previous subsection, where
$B^3=D^2\times[-1,1]\subset N(K)$. Define $(M(1),\alpha_1)$ as above,
so that \eqref{eqn:ECHhat} and \eqref{eqn:HChat} hold.

Next we define a related contact manifold $(M_0,\Gamma_0,\xi_0)$,
which is obtained from $(M-N(K), \xi|_{M-N(K)})$ by attaching a
collar. Consider
\[
A=\bdry (M-N(K))\times[-1,0]=\R/2\pi\Z \times
([-2,2]/\sim)\times[-1,0]
\]
with coordinates $(\theta,t,u)$.  We take $M_0=(M-N(K))\cup A$, where
$\bdry (M-N(K))$ is identified with $\bdry(M-N(K))\times\{-1\}$.  We
extend $\alpha$ over $A$ as $dt-cud\theta$. (This is smooth if we
define the smooth structure on $M_0$ using an appropriate chart in the
gluing region.)  If we perturb $\alpha$ near
$\bdry(M-N(K))\times\{0\}$, then the resulting $\xi_0=\ker\alpha_0$
has convex boundary and dividing set $\Gamma_0$ which consists of two
meridians (circles where $t$ is constant).

\begin{prop}
  A nullhomologous transverse knot $K$ in a closed contact 3-manifold
  $(M,\xi)$ induces:
\begin{enumerate}
\item a filtration $\mathcal{F}$ on the chain complex $C(M(1),
  \alpha_1)$ for $\widehat{HC}(M,\xi)$, such that the homology of the
  associated graded complex is isomorphic to $HC(M_0,\Gamma_0,\xi_0)$,
  and:
\item a filtration $\mathcal{F}$ on the chain complex for
  $ECH(M(1),\alpha_1,J_1)\simeq \widehat{ECH}(M,\xi)$, such that the
  homology of the associated graded complex is isomorphic to
  $ECH(M_0,\Gamma_0,\alpha_0,J_0)$, if the almost complex structures
  $J_0$ and $J_1$ are suitably related.
\end{enumerate}
\end{prop}

\begin{proof}
  We will only prove assertion (1) for sutured contact homology;
  assertion (2) for sutured ECH is proved using the same argument.

  A generator of $C(M(1),\alpha_1)$ is a monomial
  $\underline{\gamma}=\gamma_1^{m_1}\dots\gamma_k^{m_k}$, where the
  $\gamma_i$ are closed orbits of $R_{\alpha_1}$, and each $m_i$ is a
  positive integer.  The total homology class of this generator is $A
  = m_1 [\gamma_1] + \dots + m_k [\gamma_k]\in H_1(M)$.  Fix a
  relative homology class $B\in H_2(M,K)$ with $\partial B = [K]$, and
  let $S$ be a Seifert surface for $K$ in the class $B$. Let $K_1 = K
  \cap M(1)$. We view $S$ as a surface in $M(1)$ with boundary on
  $K_1\cup \bdry M(1)$. Since all the closed orbits of $R_{\alpha_1}$
  are contained in $M(1)\setminus K_1$, we can define the filtration
  level of $\underline{\gamma}$ to be its algebraic intersection
  number with $S$, namely
\[
\mathcal{F}(\underline{\gamma})=\underline{\gamma}\cdot S=\sum_{i=1}^k m_i
(\gamma_i\cdot S).
\]
Note that if
$\underline{\gamma}'=(\gamma_1')^{m_1'}\cdots(\gamma'_{l})^{m_{l}'}$
is another generator representing the same homology class $A\in
H_1(M)$, then the filtration difference is given by
\begin{equation}
\label{eqn:FD}
\mathcal{F}(\underline{\gamma})-\mathcal{F}(\underline{\gamma}')=\Sigma\cdot K_1,
\end{equation}
where $\Sigma$ is any $2$-chain in $M$ with
$\partial\Sigma=\sum_{i=1}^k m_i\gamma_i-\sum_{j=1}^{l}m_j'\gamma_j'$.
One can show this by perturbing $\Sigma$ so that it is transverse to
$S$ and then counting points in the boundary of the compact
$1$-manifold $\Sigma\cap S$.

Next we prove that the differential does not increase the filtration
level of the generators.  More generally, for any holomorphic curve
\[
F = (a, f) \colon (\Sigma, j,\mathbf{m}) \to \R \times M(1)^*
\]
which is
positively asymptotic to $\underline{\gamma}$ and negatively asymptotic to
$\underline{\gamma}'$, we have
\[
\mathcal{F}(\underline{\gamma}) \ge \mathcal{F}(\underline{\gamma}').
\]
To prove this, first note that $K_1$ extends to an infinite length
Reeb orbit $\widetilde{K}_1$ in $M(1)^*$.  Now let $\overline{\Sigma}$
be the compact surface with boundary obtained from $\Sigma$ by
performing a real blowup at each puncture. Then the map $f$ extends to
a map $\overline{f} \colon \overline{\Sigma} \to M(1)^*$ whose
restriction to the boundary is
$\sum_im_i\gamma_i-\sum_jm_j'\gamma_j'$.  Moreover $f$ is homotopic
rel boundary to a map $f'$ whose image is contained in $M(1)$.  We
then have
\[
\mathcal{F}(\underline{\gamma})-\mathcal{F}(\underline{\gamma}') =
f'(\overline{\Sigma})\cdot K_1 = f(\overline{\Sigma})\cdot
\widetilde{K}_1 \ge 0,
\]
where the last inequality holds by positivity of intersections of the
holomorphic curve $F$ with the holomorphic plane
$\R\times\widetilde{K}_1$ in $\R\times M(1)^*$.

We now show that the homology of the associated graded complex with
respect to $\mathcal{F}$ is the contact homology
$HC(M_0,\Gamma_0,\xi_0)$. Recall the identification $N(K)=D^2\times
[-2,2]/(-2\sim 2)$.  Consider a small neighborhood
$N(K_1)=D^2_{\varepsilon}\times ([-2,-1]\cup [1,2]) \subset M(1)\cap
N(K)$, where $D^2_{\varepsilon}= \{r\leq \varepsilon\}$. The manifold
$M(1)-N(K_1)$ is almost a convex sutured manifold contactomorphic to
$(M_0,\Gamma_0,\xi_0)$.  The only issue is that, along $(\bdry
D^2_{\varepsilon})\times \{t\}$ with $t\in[-2,1]\cup[1,2]$, the
contact form $\alpha$ restricts to a positive contact form with
respect to the boundary orientation induced from $D^2_{\varepsilon}$,
and hence to a negative contact form with respect to the boundary of
$R_\pm (\Gamma)$.  To remedy this problem we attach a collar
\[
A'=\R/2\pi\Z \times ([-2,-1]\cup[1,2])\times[-1,1]
\]
with coordinates
$(\theta,t,u)$ to $M(1)-N(K_1)$ by identifying $\bdry D^2_{\varepsilon}
\times ([-2,-1]\cup[1,2])$ with $\R/2\pi\Z\times
([-2,-1]\cup[1,2])\times\{-1\}$ and extending via the contact form
$dt-cud\theta$.  Then $(M(1)-N(K_1))\cup A'$ is a
sutured contact manifold, and we leave it as an exercise to prove
that it is contactomorphic to $(M_0,\Gamma_0,\xi_0)$ (modulo the
process of matching up the contact structures on the boundary by a
homotopy).

Finally, let $N(\widetilde{K}_1)$ denote the obvious extension of
$N(K_1)$ to a neighborhood of $\widetilde{K}_1$ in $M(1)^*$.  We then
observe that a holomorphic curve in $\R\times M(1)^*$ does not pass
through $\R\times \widetilde{K}_1$, i.e. does not decrease the
filtration, if and only if its image is contained in $\R\times
(M(1)^*-N(\widetilde{K}_1))$. This follows from intersection
positivity by observing that $N(\widetilde{K}_1)$ is foliated by Reeb
arcs parallel to $\widetilde{K}_1$.  A similar argument shows that the
holomorphic curves that are counted by the contact homology
differential in $\R\times (((M(1)-N(K_1))\cup A')^*$ do not pass
through the ``vertical completion'' of $\R\times A'$, and so are
contained in $\R\times (M(1)^*-N(\widetilde{K}_1))$.  Thus the
differential on the associated graded complex for $M(1)$ counts the
same holomorphic curves as the differential for the contact homology
of $(M(1)-N(K_1))\cup A' \simeq M_0$.
\end{proof}

\begin{rmk}
  Although the filtration defined above depends on the choice of a
  relative homology class $B\in H_2(M,K)$ with $\partial B = [K]$, the
  filtration {\em difference\/} between two generators representating
  the same class $A\in H_1(M)$ does not depend on this choice, by
  equation \eqref{eqn:FD}.
\end{rmk}

\subsection{Invariants of Legendrian submanifolds}
\label{subsection: invts of leg submanifolds}

In this subsection we briefly discuss invariants of Legendrian
submanifolds.  Let $(M,\xi)$ be a closed $(2n+1)$-dimensional
contact manifold and $L\subset M$ be a closed Legendrian
submanifold.  By Example~\ref{example: Legendrian2}, there is a
tubular neighborhood $N(L)$ of $L$ so that
$(M-N(L),\Gamma=S^{n-1}T^*L,\xi|_{M-N(L)})$ is a concave sutured
contact manifold.  Now, by Proposition~\ref{prop: concave to
sutured}, we can modify the concave sutured contact manifold into a
convex sutured contact manifold $(M',\Gamma',\xi')$.  Then we define
\[
HC(M,\xi,L)= HC(M',\Gamma',\xi').
\]
To show that this is well-defined, recall from
Section~\ref{subsection: ac} that the right hand side is independent
of the choices of contact form and the almost complex structure.  We
then have:

\begin{lemma}
  The contact homology algebra $HC(M,\xi,L)$ is an invariant of
  $(M,\xi,L)$, i.e. does not depend on the choice of tubular
  neighborhood of $L$.
\end{lemma}

\begin{proof}
  Observe that the hypersurface $\Sigma$ of $M$, defined in
  Example~\ref{example: Legendrian2}, has the following properties:
  \be
\item[(i)] There is a contact $1$-form for $\xi$, written locally as
\begin{equation} \label{eqn: nbhd}
\alpha=dz+\beta=dz+\sum_{i=1}^n f_i(p,q) dp_i+\sum_{i=1}^n
g_i(p,q)dq_i.
\end{equation}
Here $(z,p=(p_1,\dots,p_n),q=(q_1,\dots,q_n))$ are local
coordinates, $R_\alpha=\bdry_z$, $L=\{z=0,p=0\}$, and
$f_i(0,q)=g_i(0,q)=0$ for all $q$.  In particular, $\xi$ is tangent
to $\{z=0\}$ along $L$.
\item[(ii)] On the $2n$-dimensional submanifold $\{z=0\}$,
let $Y$ be the Liouville vector field satisfying $\imath_Yd\beta=\beta$,
and let $W_q$ be the ``fan'' consisting of all points $(p,q)$ whose
backwards flow along $Y$ converge to $(0,q)$. Then let $\Gamma$ be a
$(2n-1)$-dimensional submanifold of $\{z=0\}$ which is arbitrarily
close to $L$ and such that each $\Gamma\cap W_q$ is ``star-shaped'',
i.e., an $(n-1)$-dimensional sphere which is transverse to $Y$.
\item[(iii)] $\Gamma$ is diffeomorphic to the unit cotangent bundle of $L$
and bounds a $2n$-dimensional submanifold $\Sigma_0\subset \{z=0\}$
which is diffeomorphic to the unit disk bundle of $T^*L$.  Then
$\Sigma\cap \{z>0\}$ (resp.\ $\Sigma\cap \{z<0\}$) is transverse to
$R_\alpha$ and the projection along $R_\alpha$ gives a
diffeomorphism with $int(\Sigma_0)$.
\ee

Condition (ii) implies that $\Gamma$ is a $(2n-1)$-dimensional
contact submanifold and Condition (iii) implies that $\Sigma$ is a
convex hypersurface of $M$.

Now let $\alpha$ be a contact $1$-form for $\xi$, which is defined
in a neighborhood of $L$ and satisfies (i). In particular, $\alpha$
is given by Equation~(\ref{eqn: nbhd}). We describe the Liouville
vector field $Y$ for $\beta$ on $\{z=0\}$ when $|p|$ is arbitrarily
small. For $|p|$ small,
$$d\alpha\approx \sum_i {\bdry g_i\over \bdry p_j} dp_jdq_i+\sum_i
{\bdry f_i\over \bdry p_j}dp_jdp_i,$$ since ${\bdry f_i\over \bdry
q_j}$ and ${\bdry g_i\over \bdry q_j}$ are close to zero. Ignoring
higher order terms, we write $f_i=\sum_j F_{ij}p_j$ and $g_i=\sum_j
G_{ij}p_j$, where $F_{ij}$ and $G_{ij}$ are constants. By the
symplectic condition, $det({\bdry g_i\over \bdry
p_j})=det(G_{ij})>0$. If we write $Y=\sum_i a_i\bdry_{p_i}+\sum_i
b_i\bdry_{q_i}$, then the Liouville condition implies that
$$g_i=\sum_j {\bdry g_i\over \bdry p_j} a_j,$$
or $\sum_j G_{ij}p_j=\sum_j G_{ij}a_j$.  Hence $a_j=p_j$ and $Y$ has
the form:
\begin{equation} \label{eqn for Y}
Y=\sum_i p_i \bdry_{p_i}+ \sum_{i,j} A_{ij} p_j \bdry_{q_j},
\end{equation}
by the invertibility of $G_{ij}$.  Here $A_{ij}$ are constants which
smoothly depend on $F_{ij}$ and $G_{ij}$.

Equation~(\ref{eqn for Y}) implies that the fan $W_q$ and $Y|_{W_q}$
vary continuously as we vary $\beta$ (while preserving the
conditions in (i)), and that the $\bdry_{p_i}$-terms are independent
of $\beta$ (modulo higher order corrections).

Finally, given two convex submanifolds $\Sigma^0$ and $\Sigma^1$ of
the type described in Example~\ref{example: Legendrian2}, there is a
$1$-parameter family of contact $1$-forms $\alpha^{t}$ interpolating
between $\alpha^0$ and $\alpha^1$, all satisfying (i).  Since
$W_q^t$ varies continuously with $\beta^t$, it follows that there is
a family $\Gamma^t$ from $\Gamma^0$ to $\Gamma^1$, all satisfying
(ii).  We can then extend to $\Sigma^t$ from $\Sigma^0$ to
$\Sigma^1$, all satisfying (iii). This implies that $\Sigma^0$ and
$\Sigma^1$ can be connected by a $1$-parameter family of convex
submanifolds $\Sigma^t$.
\end{proof}

Our Legendrian submanifold invariant $HC(M,\xi,L)$, unlike other
invariants such as {\em Legendrian contact homology}, does not
automatically vanish under stabilizations. In fact,
Corollary~\ref{cor: nonvanishing Legendrian invariant} shows that
the invariant does not vanish for example when the ambient manifold
$(M,\xi)$ has an exact symplectic filling.

\begin{example}
Suppose $(M,\xi)=(S^3,\xi)$ is the standard contact $3$-sphere and
$L$ is a Legendrian unknot with Thurston-Bennequin number $tb(L)=-n$
and rotation number $r(L)=n-1$ for $n\geq 1$. (These Legendrian
unknots have maximal rotation number amongst those with the same
$tb$.) Then $(S^3-N(L),\xi|_{S^3-N(L)})$ is a sutured contact solid
torus which is obtained from a product sutured contact manifold
\[
(D^2\times[-1,1], \bdry D^2\times\{0\}, N(\bdry D^2)\times[-1,1],
dt+\beta),
\]
where $\beta$ is a primitive of an area form on $D^2$, by a sutured
manifold gluing. Its contact homology $HC(S^3,\xi,L)$ has been
completely calculated by Golovko~\cite{Go1,Go2}, and in particular is
nonzero.
\end{example}

\begin{q}
Determine the relationship of $HC(M,\xi,L)$ with the Legendrian
contact homology $LCH(M,\xi,L)$ of the Legendrian submanifold
$L\subset (M,\xi)$ as well as the contact homology $HC(M(L),\xi_L)$
of the contact manifold $(M(L),\xi_L)$, obtained from $M$ by
Legendrian surgery along $L$.
(A surgery exact sequence involving $HC(M(L),\xi_L)$ and a variant of
$LCH(M,\xi,L)$ was obtained by
Bourgeois-Ekholm-Eliashberg~\cite{BEE}.)
\end{q}

When $\dim M=3$,
we can also define
\[
ECH(M,\xi,L)= ECH(M',\Gamma',\xi').
\]
This is conjectured to be independent of the choice of $\xi$ (up to
the usual grading shift) and dependent only on the framing of $L$.

\section{First warm-up: neck-stretching in the $t$-direction}
\label{section: warm-up 1}

Before embarking on the proof of Theorem~\ref{thm: sutured gluing},
we treat slightly easier cases in this section and the next.

Consider the situation where we have a sutured contact manifold
$(M',\Gamma',\alpha')$, and there is a diffeomorphism
\[
\phi\colon  (R_+(\Gamma'),\beta'_+)\stackrel\sim\to (R_-(\Gamma'),\beta'_-),
\]
where $\beta'_\pm=\alpha'|_{R_\pm(\Gamma')}$, which is the identity on
$R_+(\Gamma')\cap U(\Gamma')$. Let $(M,\alpha)$ be the contact
manifold with boundary obtained from $M'$ by gluing $R_+(\Gamma')$ and
$R_-(\Gamma')$ via $\phi$.  If we let $\Gamma$ denote the image of
$\Gamma'$ in $M$, then a neighborhood of $\partial M$ is identified
with $[-1,0]\times(\R/\Z)\times\Gamma$ so that $\alpha=C dt + \beta$.

Although $(M,\alpha)$ is not quite a sutured contact manifold in the
sense of Definition~\ref{def:SCM}, we can nonetheless define part of
its contact homology as follows.  First complete $(M,\alpha)$ to
$(M^*,\alpha^*)$ by attaching the side (S) as usual (but not the
top/bottom), and choose a tailored almost complex structure on
$\R\times M^*$.  Define $\mathcal{A}_{[0]}(M,\Gamma,\alpha)$ be the
free supercommutative $\Q$-algebra with unit generated by good Reeb
orbits in $M^*$ which do not intersect $R_+(\Gamma')$; note that these
are the same as the good Reeb orbits in $M'$.  Note that if a
holomorphic curve in $\R\times M^*$ has all positive ends at such Reeb
orbits, then it also has all negative ends at such Reeb orbits,
because all orbits that nontrivially intersect $R_+(\Gamma')$ do so 
positively, therefore they belong to different homology classes.  
Thus the usual construction defines a well-defined
differential on $\mathcal{A}_{[0]}(M,\Gamma,\alpha)$ which has a
well-defined homology $HC_{[0]}(M,\Gamma,\alpha)$.

The goal of this section is to prove the following result:

\begin{thm} \label{thm: warm-up 1}
There is an isomorphism $HC(M',\Gamma')\simeq HC_{[0]}(M,\Gamma)$.
\end{thm}

The idea of the proof is to ``stretch the neck'' in the gluing that
produces $M$ from $M'$, with a parameter $n$ that measures the length
of the neck.  One wants to argue that if $n$ is sufficiently large
then all relevant holomorphic curves in $\R\times M^*$ correspond to
holomorphic curves in $\R\times (M')^*$.  However one cannot choose a
single $n$ that always works; the size of $n$ that is required for
this to work depends on the total symplectic action of the Reeb orbits
involved.  To deal with this issue we will use a direct limit
argument.

We remark that one can also prove a more general version of
Theorem~\ref{thm: warm-up 1} in which one glues only some components
of $R_+(\Gamma')$ to some components of $R_-(\Gamma')$.  This uses the
same argument but more notation.

\subsection{Stretching the neck}
\label{subsection: warm-up 1 stretching the neck}

For the purposes of the neck-stretching, we introduce a sequence of
contact manifolds with boundary $(M_n,\alpha_n)$ and almost complex
structures $J_n$ which are parametrized by $n$: Let $M_n$ be the
manifold diffeomorphic to $M=M'/\phi$, obtained from $M'\sqcup
(R_+(\Gamma')\times [-n,n])$ by identifying $R_+(\Gamma')$ and
$R_+(\Gamma')\times\{-n\}$ by the identity and
$R_+(\Gamma')\times\{n\}$ to $R_-(\Gamma')$ by $\phi$. We take the
$1$-form $\alpha_n$ to agree with $dt+\beta'_+$ on
$R_+(\Gamma')\times[-n,n]$ and with $\alpha'$ on $M'$. Let $J'$ be an
almost complex structure which is tailored to $(M',\alpha')$ and is
taken to itself by $\phi$.  Then define $J_n$ to be $t$-invariant on
$R_+(\Gamma')\times[-n,n]$ and to agree with $J'$ on $M'$.  Also
define $M^*_n$ as the completion of $M_n$, obtained by attaching (S),
but not (T) or (B) since $R_\pm$ have been eliminated.  By counting
$J_n$-holomorphic curves in $\R\times M_n^*$ we can define the contact
homology $HC_{[0]}(M_n,\alpha_n,J_n)$.  The standard continuation
argument shows that this does not depend on $n$ and is canonically
isomorphic to $HC_{[0]}(M,\Gamma)$.

\begin{lemma}\label{lemma: warm-up 1}
Let $\underline{\gamma}^+=(\gamma_1^+, \ldots, \gamma^+_k)$ and
$\underline{\gamma}^-=(\gamma_1^-, \ldots, \gamma^-_l)$ be finite
ordered sets of Reeb orbits in $M'$, possibly taken with
multiplicities. Then given $g$, for all sufficiently large $n$,
\[
\mathcal{M}_g(\underline{\gamma}^+;\underline{\gamma}^-;\R\times (M')^*,J')
=\mathcal{M}_g(\underline{\gamma}^+;\underline{\gamma}^-; \R\times
M_n^*,J_n).
\]
\end{lemma}

\begin{proof}
The proof is almost identical to that of Lemma~\ref{gradbound} and
Proposition~\ref{noside}; the slight difference that the ranges of
the holomorphic maps vary with $n$. Arguing by contradiction,
suppose there is a sequence
$$F_n=(a_n,f_n)\colon  (\Sigma_n,j_n,\mathbf{m}_n)\to (\R\times M_n^*,J_n)$$
in $\mathcal{M}_g(\underline{\gamma}^+;\underline{\gamma}^-;\R\times
M_n^*,J_n)$ whose second component $f_n$ nontrivially intersects
$\widehat{R_+(\Gamma')}\times\{0\}$ for all $n$. (Observe that, if
$f_n$ does not intersect $\widehat{R_+(\Gamma')}\times\{0\}$, then
$F_n$ can be viewed as a holomorphic map in
$\mathcal{M}_g(\underline{\gamma}^+;\underline{\gamma}^-;\R\times
(M')^*,J')$.)  As before, we can restrict to $\tau\leq 0$ by the
strict plurisubharmonicity of $\tau$.

On $\R\times M_n^*$ we use the metric given by Equation~(\ref{eqn:
metric}), and on $(\Sigma_n-\mathbf{m}_n,j_n)$ we use the unique
complete, compatible, finite volume hyperbolic metric $g_n$.  Also
write $\rho_n$ for the injectivity radius of $g_n$.  If there is no
``gradient bound'', i.e., a bound on $\rho_n(x)\|\nabla F_n(x)\|$,
then we obtain the bubbling off of a nonconstant finite energy plane
with image in $(\R\times (M')^*,J')$ or $(\R\times \R\times
\widehat{R_+(\Gamma')}, J')$ by Lemma~\ref{rescaling}.  In the
latter case, we obtain a holomorphic sphere inside $(\R\times
\R\times \widehat{R_+(\Gamma')}, J')$ by the removal of
singularities lemma for the Top/Bottom (Lemma~\ref{removal}), a
contradiction. Hence the bubbling occurs inside $(\R\times
(M')^*,J')$.  Since the area of finite energy holomorphic planes is bounded by below
(see \cite[Lemma 5.11]{BEHWZ}), we
can remove finite sets $\mathbf{m}_n^0$ from $\Sigma_n-\mathbf{m}_n$
to ensure that there is a gradient bound with respect to
$(\dot\Sigma_n=\Sigma_n-(\mathbf{m}_n\cup \mathbf{m}^0_n),j_n)$.

Arguing as in Proposition~\ref{noside}, there is a subsequence of
$F_n$ (again denoted $F_n$ by abuse of notation) for which:
\begin{enumerate}
\item[(i)]
there
is a bound on the gradient,
\item[(ii)]
there is a $\varepsilon$-thin
component $C^n$ of $\dot\Sigma_n$ and an annulus $Z^n\subset C^n$,
such that $f_n(Z^n)\subset R_+(\Gamma')\times[-n,n]$, and
\item[(iii)]
$\max_{x\in Z^n} t\circ f_n(x) -\min_{x\in Z^n} t\circ f_n(x)$ is an
unbounded sequence in $n$, where $t\in[-n,n]$.
\end{enumerate}
This sequence limits
to a nonconstant holomorphic cylinder in $(\R\times\R\times
\widehat{R_+(\Gamma')},J')$, which is a contradiction.
\end{proof}

\subsection{Continuation maps}
\label{subsection: warm-up 1 continuation maps}

Given a contact form $\alpha$, the {\em action} of an oriented curve
$\gamma$ with respect to $\alpha$ will be written as
\[
A_\alpha(\gamma)=\int_\gamma\alpha.
\]
We also write
$\overline{\gamma}=\gamma_1^{m_1}\dots\gamma_k^{m_k}$ and
$A_{\alpha}(\overline{\gamma})=\sum_i m_iA_\alpha(\gamma_i)$.

Let $\mathcal{A}_{\leq K}(M',\alpha',J')$ denote the subcomplex of
$\mathcal{A}(M',\alpha',J')$ generated (as a module) by monomials
$\overline{\gamma}$ with $A_{\alpha'}(\overline{\gamma})\leq
K$. Lemma~\ref{lemma: warm-up 1} implies that given $K$, if $n$ is
sufficiently large then the inclusion
\[
\Phi_{K,n}\colon \mathcal{A}_{\leq K}(M',\alpha',J') \hookrightarrow
\mathcal{A}(M_n,\alpha_n,J_n)
\]
is a chain map.

We now
investigate the dependence of this map on $K$ and $n$.  To start, we have the following key lemma:

\begin{lemma} \label{lemma: warm-up 1 cobordism} For all $n$
  sufficiently large, the canonical isomorphism
  $HC_{[0]}(M_n,\alpha_n,J_n)\simeq
  HC_{[0]}(M_{n+1},\alpha_{n+1},J_{n+1})$ is induced by a chain map
\[
\Psi_n\colon  \mathcal{A}_{[0]}(M_n,\alpha_n,J_n)
\to \mathcal{A}_{[0]}(M_{n+1},\alpha_{n+1},J_{n+1}),
\]
such that if $\gamma$ is a Reeb orbit in $M'$ then
\begin{equation}
\label{eqn: warm-up 1 gamma}
\gamma\mapsto\gamma+\sum_i a_i\overline{\gamma_i},
\end{equation}
where all the orbits of $\overline{\gamma_i}$ are contained in
$M'$ and $A_{\alpha_n}(\gamma)>
A_{\alpha_{n+1}}(\overline{\gamma_i})=A_{\alpha_n}(\overline{\gamma_i})$.
\end{lemma}

In particular, the lemma implies that the chain map $\Psi_n$ is
``triangular'', i.e., is the identity plus lower order terms with
respect to the action.

\begin{proof}
Let us write $\alpha^1=\alpha_n$; on $R_+(\Gamma')\times
[-n,n]$, $\alpha^1=dt+\beta$. (In this subsection we will write
$\beta$ for $\beta'_+$.) There exists an identification
$i_n\colon M_n\stackrel\sim\to M_{n+1}$ so that $M'$ is taken to itself by
the identity and $i_n^*(\alpha_{n+1})=f(t)dt+\beta$ on
$R_+(\Gamma')\times [-n,n]$, where $1\leq f(t)\leq 1+{2\over n-1}$.
If we set $\alpha^0=i_n^*(\alpha_{n+1})$, then $\alpha^0$ and
$\alpha^1$ agree on $M'$. Let $\alpha^s$, $s\in[0,1]$, be the
$1$-parameter family of contact $1$-forms obtained by interpolating
between $\alpha^0$ and $\alpha^1$.
Let us write $\alpha^s=\alpha^0$ for $s\leq 0$ and
$\alpha^s=\alpha^1$ for $s\geq 1$.

Define an almost complex
structure $J$ on $\R\times M_n^*$ such that the following hold: \be
\item For all $s\in \R$, $J|_{\{s\}\times M_n^*}$ takes $\ker \alpha^s$ to
itself, maps $\bdry_s$ to $R_{\alpha^s}$, and is
$d\alpha^s$-positive;
\item $J|_{s\geq 1}=J_n$ and $J|_{s\leq 0} =
i_n^*J_{n+1}$;
\item $J$ is $s$-invariant on $\R\times M'$;
\item the projection of $J|_{\R\times R_+(\Gamma')\times [-n,n]}$ to
$R_+(\Gamma')$ does not depend on $s$ and on $t$. \ee

The cobordism $(\R\times M_n^*,J)$ gives rise to the chain map
$\Psi_n$, obtained in the usual way by counting rigid rational
curves with one positive puncture and an unspecified number of
negative punctures.  The $2$-form $\omega$ that we use below to
control the action is insufficient for verifying the compactness of
the relevant moduli spaces. For compactness, we need a taming form
$d(g(s)\alpha^s)$ for a suitable $g(s)$, whose $J$-positivity is
verified as in Lemma~\ref{lemma: interpol}.  We also restrict the
range from $\R\times M_n^*$ to $\R\times M_n$; this is possible
since the projection of $J$ to $\widehat{R_+(\Gamma')}$ is adapted
to $\beta$.

Consider the $2$-form $\omega= d\alpha^1$.  We claim that $\omega$ is
$J$-nonnegative, i.e., $\omega(v,Jv)\geq 0$ for all tangent vectors
$v\not=0$.  On $M'$, $\alpha^1=\alpha^0$, and the claim is
immediate. On $R_+(\Gamma')\times[-n,n]$, we have $\omega= d\beta$.
If we write $v\in T_{(s, x)}(\R\times R_+(\Gamma')\times[-n,n])$ (for $s \in \R$ and
$x \in R_+(\Gamma')\times[-n,n]$) as
$a\bdry_s+b\bdry_t+w$, where $w\in \ker \alpha^s$, then
$Jv=ah\bdry_t-(b/h) \bdry_s +J_{s,t}w$.  Here $h$ is a function which
is approximately equal to $1$. (This comes from the fact that
$R_{\alpha^s}$ is parallel, but not exactly equal, to $\bdry_t$ on
$\R\times R_+(\Gamma')\times[-n,n]$.) We then compute that
\[
\omega(v,Jv)=d\beta(w,J_{s,t}(w))\geq 0,
\]
by projecting to $R_+(\Gamma')$.

Next let $F$ be a holomorphic curve in $(\R\times M_n^{TB},J)$ with
positive end $\gamma$ and negative ends $\overline{\gamma}'$.
As noted previously, if $\gamma\subset M'$, then all orbits of
$\overline{\gamma}'$ are also contained in $M'$ for homological reasons. By
Stokes' theorem and the $J$-nonnegativity of $\omega$, we have:
\begin{equation} \label{eqn: warm-up 1 action inequality}
A_{\alpha^1}(\gamma)\geq
A_{\alpha^0}(\overline{\gamma}')=A_{\alpha^1}(\overline{\gamma}').
\end{equation}
Note that the first term on the right-hand side of Equation~(\ref{eqn:
  warm-up 1 gamma}) comes from counting a trivial cylinder over
$\gamma$. To obtain strict inequality in (\ref{eqn: warm-up 1 action
  inequality}) when $\overline{\gamma}'\not=\gamma$, first observe
that $F$ is asymptotically a cylinder over $\gamma$ at $+\infty$. If
$F$ is not a cylinder over $\gamma$, then $F$ must have positive
$d\alpha^1$-area, implying the strict inequality in (\ref{eqn: warm-up
  1 action inequality}).  (Branched covers of trivial cylinders do not
contribute to the differential by Fabert \cite{F}.)
\end{proof}

Our next ingredient is the following:

\begin{lemma}\label{lemma: warm-up 1 inclusion of algebras}
Given $K>0$, there exists $n_0> 0$ such that for all $n\geq n_0$,
\[
\Psi_n\colon  \mathcal{A}_{[0]}(M_n,\alpha_n,J_n)\to
\mathcal{A}_{[0]}(M_{n+1},\alpha_{n+1},J_{n+1})
\]
maps
$\gamma\mapsto\gamma$, whenever $A_{\alpha_n}(\gamma)\leq K$.
\end{lemma}

\begin{proof}
This is a variant of the proof of Lemma~\ref{lemma: warm-up 1}.
First note that $A_{\alpha_n}(\gamma) \leq K$ implies that
$\gamma\subset M'$ for sufficiently large $n$.  Suppose there is a
sequence of finite energy, rational holomorphic maps $F_n$ to
$(\R\times M_n,\widetilde J_n)$ with one positive end at
$\gamma$, where $\widetilde J_n$ is the almost complex structure for
the cobordism given in Lemma~\ref{lemma: warm-up 1 cobordism} (called $J_n$ there). 
If $F_n$ intersects $\R\times R_+(\Gamma')\times\{0\}$ for all $n$, the
proof of Lemma~\ref{lemma: warm-up 1} produces a holomorphic sphere
in $\R\times R_+(\Gamma')\times \R$, a contradiction. (Note that, as
$n\to \infty$, the difference between the almost complex structure
$\widetilde J_n$ and the tailored almost complex structure $J_n$ for
$\alpha_n$ becomes arbitrarily small in the $C^{\infty}$ topology.) 
Hence $F_n$ can be viewed as
a map to $(\R\times (M')^*,J')$ for sufficiently large $n$.  Since
$J'$ is $\R$-invariant, the lemma
follows.
\end{proof}

Lemma~\ref{lemma: warm-up 1 inclusion of algebras} implies that the diagram
\begin{equation}
\label{eqn:cd}
\begin{diagram}
\mathcal{A}_{\leq K}(M',\alpha',J') & \rTo^{\Phi_{K,n}} &
\mathcal{A}_{[0]}(M_n, \alpha_{n},J_n)  \\
& \rdTo^{\Phi_{K, n+1}}  &  \dTo_{\Psi_{n }} \\
 &  &
\mathcal{A}_{[0]}(M_{n+1},\alpha_{n+1},J_{n+1})  \\
\end{diagram}
\end{equation}
commutes, provided that $n$ is sufficiently large with respect to $K$.

\subsection{Direct limits}
\label{subsection: warm-up 1 direct limits} $\mbox{}$

\s\n {\bf Definition of $\Phi$.} Suppose $n'\gg n\gg 0$. By
composing $ \Psi_{n'-1}\circ \Psi_{n'-2}\circ \dots\circ \Psi_n$, we
obtain a chain map
\[
\Psi_{n,n'}\colon  \mathcal{A}_{[0]}(M_n,\alpha_n,J_{n})\to
\mathcal{A}_{[0]}(M_{n'},\alpha_{n'},J_{n'}),
\]
where $\gamma\subset M'$ is mapped to $\gamma+\sum_i
a_i\overline{\gamma_i}$, the orbits of $\overline{\gamma_i}$ are
contained in $M'$, and
$A_{\alpha_n}(\overline{\gamma_i})=A_{\alpha_{n'}}(\overline{\gamma_i})<
A_{\alpha_n}(\gamma)$.  It follows from the commutativity of the
diagram \eqref{eqn:cd} that if $K<K'$ and if $n$ is sufficiently large
with respect to $K'$, then the chain map $\Psi_{n,n'}$ fits into the
following commutative diagram of chain complexes:
\begin{equation} \label{equation: warm-up 1 diagram}
\begin{diagram}
\mathcal{A}_{\leq K}(M',\alpha',J') & \rTo^{\Phi_{K,n}} &
\mathcal{A}_{[0]}(M_n, \alpha_{n},J_{n})  \\
\dTo^{i_{K,K'}} &
&  \dTo_{\Psi_{n,n'}} \\
\mathcal{A}_{\leq K'}(M',\alpha',J')  & \rTo^{\Phi_{K',n'}} &
\mathcal{A}_{[0]}(M_{n'},\alpha_{n'},J_{n'}) . \\
\end{diagram}
\end{equation}
Here $i_{K,K'}$ denotes the natural inclusion. Note that the usual
chain homotopy argument shows that $\Psi_{n,n'}$ is chain homotopic to
any continuation map given by a symplectic cobordism from $\alpha_n$
to $\alpha_{n'}$.

By commutativity of the diagram~(\ref{equation: warm-up 1 diagram}),
we can take direct limits to obtain a map
\[
\Phi: \lim_{K\to \infty} HC_{\leq K}(M',\alpha',J') \to
\lim_{n\to\infty} HC_{[0]}(M_n,\alpha_n,J_n)
\]
at the level of homology.  Now observe
that
\[
\lim_{K\to \infty} HC_{\leq K}
(M',\alpha',J') = HC(M',\alpha',J') = HC(M',\Gamma'),
\]
because the analogous statement at the level of chain complexes holds
by definition, and taking homology commutes with direct limits.  On
the other hand,
\[
\lim_{n\to\infty} HC_{[0]}(M_n,\alpha_n,J_n) = HC_{[0]}(M,\Gamma),
\]
because the map $\Psi_{n,n'}$ induces the canonical isomorphism on
homology, so that the direct limit is isomorphic to any single
$HC_{[0]}(M_n,\alpha_n,J_n)$, and canonically isomorphic to
$HC_{[0]}(M,\Gamma)$.  We conclude that $\Phi$ defines a map
\[
\Phi\colon HC(M',\alpha')\to
HC_{[0]}(M,\Gamma).
\]
To complete the proof of Theorem~\ref{thm: warm-up 1}, we will show
that this is an isomorphism.

In the arguments below, we will use, without further notation, the
canonical identifications
\[
\mathcal{A}(M',\alpha',J')\simeq
\mathcal{A}_{[0]}(M_n,\alpha_n,J_n)\simeq
\mathcal{A}_{[0]}(M_{n'},\alpha_{n'},J_{n'})
\]
arising from the fact that an orbit $\gamma$ in
$\mathcal{A}(M',\alpha',J')$ can naturally be viewed as an orbit in
$\mathcal{A}_{[0]}(M_n,\alpha_n,J_n)$ or in
$\mathcal{A}_{[0]}(M_{n'},\alpha_{n'},J_{n'})$.  These are
identifications of $\Q$-vector spaces, but not necessarily of chain
complexes.

\s\n {\bf Injectivity of $\Phi$.} Refer to Diagram~(\ref{equation:
warm-up 1 diagram}).  Suppose that $a$ is a cycle in $\mathcal{A}_{\leq
K}(M',\alpha',J')$ and that $a=\bdry b$ for some $b\in
\mathcal{A}_{[0]}(M_n,\alpha_n, J_{n})$ with $n$ sufficiently large.
Then $\Psi_{n,n'}$ sends $a\mapsto a$ by Lemma~\ref{lemma: warm-up 1
inclusion of algebras}, and $b\mapsto b+\sum_i b_i$ by
Lemma~\ref{lemma: warm-up 1 cobordism}, where $A_{\alpha_{n'}}(b_i)<
A_{\alpha_n}(b)$. (Here $A_{\alpha_n}(b)$ means the maximum over all
the monomials of $b$.) Hence $a=\bdry (b+\sum_i b_i)$ in
$\mathcal{A}_{[0]}(M_{n'},\alpha_{n'},J_{n'})$.  Now, if we let $K'>
A_{\alpha_{n}}(b)$, then, for sufficiently large $n'$, the inclusion
$\Phi_{K',n'}$ is a chain map by Lemma~\ref{lemma: warm-up 1}.
Hence $a=\bdry (b+\sum_i b_i)$ in $\mathcal{A}_{\leq K'}(M',\alpha',
J')$. This proves the injectivity of $\Phi$.

\s\n {\bf Surjectivity of $\Phi$.}  Suppose $a$ is a cycle in
$\mathcal{A}_{[0]}(M_n,\alpha_n, J_{n})$ for some $n$.  By
Lemma~\ref{lemma: warm-up 1 inclusion of algebras}, $\Psi_{n,n'}(a)=
a+\sum_i a_i$ stabilizes for sufficiently large $n'$. As before, for
sufficiently large $K'$, the inclusion $\Phi_{K',n'}$ is a chain map by
Lemma~\ref{lemma: warm-up 1}.  Hence $\Phi_{K',n'}$ sends $a+\sum_i
a_i\mapsto a+\sum_ia_i$.  This proves the surjectivity of $\Phi$.

\subsection{Proof of Theorem~\ref{thm: connected sum}.}
\label{subsection: proof of connected sum theorem} 
Starting with $(M_i,\xi_i)$, let $B_i$ be a standard Darboux ball with
convex boundary in $M_i$, and set $M_i'=M_i-B_i$.  Applying the
convex-to-sutured operation in Lemma~\ref{l:suture}, we obtain sutured
contact manifolds $(M_i'', \Gamma_i''=S^{2n-1}, U(\Gamma_i''),
\xi_i'')$, where $R_\pm(\Gamma_i'')=D^{2n}$.  We then glue $M_1''$,
$M_2''$, and a layer $D^{2n}\times[-N,N]$ so that $R_-(\Gamma_1'')$
and $D^{2n}\times \{N\}$ are identified by a diffeomorphism and
$R_+(\Gamma_2'')$ and $D^{2n}\times\{-N\}$ are identified by a
diffeomorphism.  Without loss of generality we may assume that the
contact $1$-form on $D^{2n}\times \{-N\}$ has the form $dt+\beta$, and
that the contact forms on $M_i''$ agree with $dt+\beta$. Now observe
that all the Reeb orbits of $M''=M_1''\cup M_2''\cup
(D^{2n}\times[-N,N])$ are Reeb orbits of $M_1''$ or Reeb orbits of
$M_2''$. The rest of the proof of Theorem~\ref{thm: connected sum}(1)
is identical to that of Theorem~\ref{thm: warm-up 1}.

We prove Theorem~\ref{thm: connected sum}(2) using a slightly
different argument (which can also be used to give an alternate proof
of Theorem~\ref{thm: connected sum}(1)).  Let $(M'',\alpha_N)$ denote
the version of $M''$ with neck stretching parameter $N$.  (That is we
are using diffeomorphisms to regard the different stretched contact
manifolds as different contact forms on the same 3-manifold.)  Fix
almost complex structures $J^i$ as needed to define the ECH of $M_i''$
for $i=1,2$.  Let $J_N$ be an almost complex structure as needed to
define the ECH of $(M'',\alpha_N)$, which restricts $J^i$ on $M_i''$.
An analogue of Lemma~\ref{lemma: warm-up 1}, modifed for ECH as in
Proposition~\ref{prop: ECH top and bottom bound}, shows that for any
$K$, if $N$ is sufficiently large, then there is a canonical
isomorphism
\begin{equation}
\label{eqn:ECHTensor1}
ECH_{\le K}(M'',\alpha_N,J_N) \stackrel{\simeq}{\longrightarrow}
ECH_{\le K}(M_1''\sqcup M_2'')
\end{equation}
induced by the obvious bijection on generators.  From this description
of the isomorphism it follows that given $K<K'$, if $N$ is
sufficiently large, then the above isomorphisms for $K$ and $K'$ fit
into a commutative diagram
\begin{equation}
\label{eqn:ECHTensor2}
\begin{CD}
ECH_{\le K}(M'',\alpha_N,J_N) @>{\simeq}>>
ECH_{\le K}(M_1''\sqcup M_2'')\\
@VVV @VVV\\
ECH_{\le K'}(M'',\alpha_N,J_N) @>{\simeq}>>
ECH_{\le K'}(M_1''\sqcup M_2'')
\end{CD}
\end{equation}
where the vertical arrows are induced by the inclusions of chain complexes.

Now since the Reeb orbits in $M''$ and their actions do not depend on
the neck stretching parameter $N$, lemmas from \cite{HT3} can be
invoked to show the following:
\begin{enumerate}
\item[(i)] $ECH_{\le K}(M'',\alpha_N,J_N)$ does not depend on $N$,
 i.e.\ for any $N,N'$ there is a canonical isomorphism
\[
ECH_{\le K}(M'',\alpha_N,J_N)\simeq ECH_{\le K}(M'',\alpha_{N'},J_{N'}).
\]
Thus we can denote this homology simply by $ECH_{\le K}(M'')$.  (The
above isomorphism is constructed by choosing a generic homotopy from
$(\alpha_N,J_N)$ to $(\alpha_{N'},J_{N'})$, dividing the homotopy into
a composition of many short homotopies, and taking the composition of
the corresponding continuation isomorphisms from \cite{HT3}.  Note
that the latter continuation maps are defined using Seiberg-Witten
theory and so are only valid in a closed manifold.  To apply them
here, for any given $K$, take a large irrational ellipsoid whose Reeb
orbits have action much larger than $K$, remove a cylinder $Z$ such
that the Reeb flow near $\partial Z$ is diffeomorphic to the Reeb flow
near $\partial M''$, and then glue in $M''$.  )
\item[(ii)] For any given $K$, if $N,N'$ are sufficiently large, then
 the above canonical isomorphism is induced by the obvious bijection
 on generators.  (Here we are again using the ECH analogue of
 Lemma~\ref{lemma: warm-up 1}.)
\item[(iii)]
If $K<K'$ then the inclusion-induced map
\[
ECH_{\le K}(M'',\alpha_N,J_N)\to ECH_{\le K'}(M'',\alpha_N,J_N)
\]
commutes with the canonical isomorphisms in (i) and so induces a well-defined map
\[
ECH_{\le K}(M'') \to ECH_{\le K'}(M'').
\]
\end{enumerate}

It follows from (i) and (ii) that the isomorphism
\eqref{eqn:ECHTensor1} induces a well-defined isomorphism
\[
ECH_{\le K}(M'') \stackrel{\simeq}{\longrightarrow} 
ECH_{\le K}(M_1''\sqcup M_2'').
\]
By (iii) and the commutative diagram \eqref{eqn:ECHTensor2}, the above
isomorphisms fit into a commutative diagram
\[
\begin{CD}
ECH_{\le K}(M'') @>{\simeq}>> 
ECH_{\le K}(M_1''\sqcup M_2'')
\\
@VVV @VVV\\
ECH_{\le K'}(M'') @>{\simeq}>> 
ECH_{\le K'}(M_1''\sqcup M_2'').
\end{CD}
\]
We can then take the direct limit over $K$ to obtain an isomorphism
\[
ECH(M'') \stackrel{\simeq}{\longrightarrow} ECH(M_1'')\otimes ECH(M_2'').
\]
By Theorem~\ref{thm: mapping cone}, $ECH(M_i'')\simeq
\widehat{ECH}(M_i)$ and $ECH(M'')\simeq \widehat{ECH}(M_1\# M_2)$.
This completes the proof of Theorem~\ref{thm: connected sum}(2).

\section{Second warm-up: neck-stretching in the 
$\tau$-direction}
\label{section: warm-up 2}

Let $(M',\Gamma',\alpha')$ be a sutured contact manifold and let
$(W,\beta)$ be a Liouville cobordism from $\bdry_+W$ to $\bdry_-W$, as
defined in Example~\ref{example: interval fibered extension}. Suppose
there is a diffeomorphism which takes $(\bdry R_+(\Gamma'),
\beta_0=\alpha'|_{\bdry R_+(\Gamma')})$ to
$(\bdry_-W,\beta|_{\bdry_-W})$. We also assume that $R_{\beta_0}$ is
nondegenerate. Let us write $N=[0,1]\times[-1,1]\times \Gamma'$ with
coordinates $(\tau,t,x)$.
We construct the interval-fibered extension
$(M,\Gamma=\bdry_+W,\alpha)$ of $(M',\Gamma',\alpha')$  as follows:
The manifold $M$ is obtained from $M'\sqcup N\sqcup(W\times[-1,1])$
by identifying $\{0\}\times [-1,1]\times \Gamma'\subset U(\Gamma')$
and $\{0\}\times [-1,1]\times \Gamma'\subset N$ and by identifying
$\{1\}\times[-1,1]\times\Gamma'$ and $\bdry_-W\times[-1,1]$. We then
define $\alpha$ as follows:
\begin{equation}
\alpha=\left\{ \begin{array}{ll}
\alpha' & \mbox{on } M';\\
dt+\overline\beta & \mbox{on } N\cup (W\times[-1,1]),
\end{array}
\right.
\end{equation}
where $\overline\beta$ is a $1$-form on
$W_N=([0,1]\times\Gamma')\cup W$, which equals $e^\tau\beta_0$ on
$[0,1]\times\Gamma'$ and $e^1\beta$ on $W$.

Let $\kappa>0$.  Choose a diffeomorphism
$$H_\kappa\colon [0,1]\times \Gamma'\stackrel\sim\to
[0,\kappa]\times \Gamma',$$
$$(\tau,x)\mapsto (h_\kappa(\tau),x),$$
where $h_\kappa\colon [0,1]\stackrel\sim\to[0,\kappa]$, $h_\kappa(0)=0$,
$h_\kappa(1)=\kappa$, $h_\kappa'(\tau)=1$ in a neighbourhood of $\tau=0,1$, and
$h_\kappa$ is linear outside a bigger neighbourhood of $\tau=0,1$. If $J'$ is an almost
complex structure on $M'$ which is tailored to $\alpha'$, then we
define its extension $J_\kappa$ on $M$ to be tailored to $\alpha$,
subject to the following conditions on the projection $(J_\kappa)_0$
of $J_\kappa$ to $W_N$:
\begin{enumerate}
\item $(J_\kappa)_0$ is independent of
$\kappa$ on $W$;
\item on $[0,1]\times \Gamma'$, $(J_\kappa)_0$ is the pullback of a
$\beta_0$-adapted almost complex structure on $[0,\kappa]\times
\Gamma'$ via $H_\kappa$.
\end{enumerate}
By sending $\kappa\to \infty$, we are ``stretching the neck'' in the
$\tau$-direction.

In this section we prove the following theorem:

\begin{thm}
\label{thm: interval-fibered extension} An interval-fibered
extension $(M',\Gamma',\xi')\hookrightarrow (M,\Gamma,\xi)$ induces
an isomorphism
$$\Phi\colon  HC(M',\Gamma',\xi')\stackrel\sim\to HC(M,\Gamma,\xi).$$
\end{thm}

The proof of Theorem~\ref{thm: interval-fibered extension} follows
the same outline as the proof of Theorem~\ref{thm: warm-up 1}.

We first observe that the set of Reeb orbits of
$(M',\Gamma',\xi',\alpha')$ and $(M,\Gamma,\xi,\alpha)$ are the
same.  The holomorphic curves are restricted by the following analog
of Lemma~\ref{lemma: warm-up 1}:

\begin{lemma}\label{lemma: warm-up 2}
Suppose $\underline{\gamma}^+$ and $\underline{\gamma}^-$ consist of
orbits in $M'$. Then, for sufficiently large $\kappa$,
$$\mathcal{M}_g(\underline{\gamma}^+;\underline{\gamma}^-;\R\times (M')^*,J')
=\mathcal{M}_g(\underline{\gamma}^+;\underline{\gamma}^-; \R\times
M^*,J_\kappa).$$
\end{lemma}

\begin{proof}
  Arguing by contradiction, suppose there is a sequence of holomorphic
  curves $$F_\kappa=(a_\kappa,f_\kappa)\colon
  (\Sigma_\kappa,j_\kappa,\mathbf{m}_\kappa)\to (\R\times
  M^*,J_\kappa)$$ in
  $\mathcal{M}_g(\underline{\gamma}^+;\underline{\gamma}^-;\R\times
  M^*,J_\kappa)$, whose second component $f_\kappa$ nontrivially
  intersects $(W_N)^{TB}=\R\times W_N$ for all $\kappa$.  (Here the
  superscript `$TB$' indicates that we are extending towards the top
  and bottom.)  We write $f_\kappa=(b_\kappa,v_\kappa)$ when
  $f_\kappa(x)\in (W_N)^{TB}$; here $b_\kappa=t\circ f_\kappa$ and
  $v_\kappa$ is the projection onto $W_N$.

On $\R\times M$ we use the Riemannian metric
$$g_\kappa=ds\otimes ds + \alpha\otimes \alpha +
\omega(\cdot,J_\kappa \cdot) - \omega(J_\kappa\cdot, \cdot),$$ where
$\omega$ is the (not everywhere closed) $2$-form defined by
$$\omega=\left\{ \begin{array}{ll} d\alpha' & \mbox{on } M';\\
d\tilde\tau\wedge \beta_0+d\beta_0 &  \mbox{on }
H_\kappa([0,1]\times\Gamma')=[0,\kappa]\times
\Gamma';\\
d\beta & \mbox{on } W.\end{array}\right.$$ Here $\tilde\tau$ is the
coordinate on $[0,\kappa]$.  

If there is a gradient blow-up for the sequence $F_\kappa$ in the
neck region $\R\times \R\times [0,\kappa]\times\Gamma'$, then the
usual argument gives us a nonconstant finite energy plane in
$\R\times\R\times \R\times \Gamma'$.  However, since there are no
closed orbits in $\R\times \R\times \R\times \Gamma'$, we obtain a
contradiction. Putting in finitely many punctures on
$\Sigma_\kappa-\mathbf{m}_\kappa$ to bound the gradient of
$F_\kappa$ on
$\dot\Sigma_\kappa=\Sigma_\kappa-(\mathbf{m}_\kappa\cup
\mathbf{m}^0_\kappa)$ as usual, we apply similar considerations as
in Proposition~\ref{noside}. There is a connected subsurface
$\overline\Sigma_\kappa$ of $\dot\Sigma_\kappa$ which satisfies the
following:
\begin{itemize}
\item $f_\kappa(\overline\Sigma_\kappa)\subset \R\times (([{1\over 2},
1]\times \Gamma')\cup W)$;
\item $\overline\Sigma_\kappa$ is a union of type $A\cup B$, where $A$
is a possibly empty union of thick and thin components of
$\dot\Sigma_\kappa$ and $B$ is a nonempty union of annular subsets
of thin components of $\dot\Sigma_\kappa$;
\item The annular subsets are of the form $[-R,0]\times S^1$ inside
thin components $[-R,R']\times S^1$ or $[-R, \infty)\times S^1$, or
of the form $[-R'',R''']\times S^1\subset [-R,R']\times S^1$. Here
$R,R'',R'''\to \infty$ as $\kappa\to\infty$;
\item $f_\kappa(\overline\Sigma_\kappa)$ nontrivially intersects
$\R\times W$ and $f(\bdry\overline\Sigma_\kappa)\subset \R\times
[{1\over 2},{1\over 2}+\varepsilon]\times \Gamma'$.
\end{itemize}
We now consider $v_\kappa$ restricted to $\overline\Sigma_\kappa$.
Observe that the finiteness of the $d\alpha$-energy of $F_\kappa$
implies the finiteness of $d\overline\beta$-energy of $v_\kappa$.
Moreover, if $\tilde\beta=f(\tau)\beta_0$ on $[0,1]\times \Gamma'$,
where $f\colon [0,1]\to \R$ is a smooth, monotonically increasing function
which agrees with $e^\tau$ on $[0,{1\over 2}+\varepsilon]$ and
satisfies $f(1)=e^1$, then Stokes' theorem gives an upper bound on
the $d\tilde\beta$-energy of $v_\kappa$ on $[{1\over
2}+\varepsilon,1]\times \Gamma'$. We then have the Hofer energy
bound of $v_\kappa$ on $[{1\over 2}+\varepsilon,1]\times \Gamma'$.
Therefore, $v_\kappa$ converges to a finite energy holomorphic curve
in $W\cup (\R\times \Gamma')$ without any positive ends,
contradicting Stokes' theorem. (Here the $\R$ coordinate corresponds to
$\tilde\tau$.) Hence, for sufficiently large $\kappa$,
$F_\kappa$ does not intersect $\R\times \R\times W$. It follows that
$F_\kappa$ has image inside $\R\times (M')^*$.
\end{proof}

By Lemma~\ref{lemma: warm-up 2}, given $K>0$, there exists
$\kappa>0$ such that all the punctured holomorphic spheres in
$(\R\times M^*,J_\kappa)$ which are asymptotic to $\gamma\in \mathcal{A}_{\leq
K}(M',\alpha',J')$ at the positive end are disjoint from $\R\times
W\times[-1,1]$. Hence we have an inclusion of chain complexes:
$$\Phi_{K,\kappa}\colon \mathcal{A}_{\leq K}(M',\alpha',J') \hookrightarrow
\mathcal{A}(M,\alpha,J_\kappa),$$ for sufficiently large $\kappa$.

We now compare $(M,\alpha,J_{\kappa})$ and $(M,\alpha,J_{\kappa+1})$
for sufficiently large $\kappa$. Observe that the contact forms are
the same, and we are only interpolating between $J_\kappa$ and
$J_{\kappa+1}$.  The almost complex structures differ only on
$\R\times \R\times [0,1]\times \Gamma'$.  We identify
$H_\kappa\colon [0,1]\times \Gamma'\stackrel\sim\to [0,\kappa]\times
\Gamma'$ and use coordinates $\tilde\tau$ on $[0,\kappa]$. Then
$(J_\kappa)_0$ and $(J_{\kappa+1})_0$ agree on $\ker \beta_0$;
however, $(J_\kappa)_0$ sends $\bdry_{\tilde\tau}\mapsto
R_{\beta_0}$ and $(J_{\kappa+1})_0$ sends $\bdry_{\tilde\tau}\mapsto
f(\tilde\tau) R_{\beta_0}$, where we may take $1-{2\over \kappa}\leq
f(\tilde\tau)\leq 1$.  Let $(J_{\kappa+1-s})_0$, $s\in[0,1]$, be an
interpolation between $(J_{\kappa+1})_0$ and $(J_{\kappa})_0$ where
only the function $f(\tilde\tau)$ is varying. Now define the almost
complex structure $J_{\kappa+1-s}$ on $M$ to be tailored to $\alpha$
so that the projection to $W_N$ is $(J_{\kappa+1-s})_0$. We then
define the almost complex structure $\widetilde{J}_\kappa$ on
$\R\times M^*$ so that:
\begin{enumerate}
\item $(\widetilde{J}_\kappa)|_{s\geq 1}=J_{\kappa}$ and
$(\widetilde{J}_\kappa)|_{s\leq 0}=J_{\kappa+1}$;
\item $(\widetilde{J}_\kappa)|_s=J_{\kappa+1-s}$.
\end{enumerate}

The following is the analog of Lemma~\ref{lemma: warm-up 1
cobordism}:

\begin{lemma} \label{lemma: warm-up 2 cobordism}
The cobordism $(\R\times M^*,\widetilde{J}_\kappa)$ gives rise to a
continuation map
$$\Psi_\kappa\colon  \mathcal{A}(M,\alpha,J_\kappa)
\to \mathcal{A}(M,\alpha,J_{\kappa+1}),$$ with the property that, if
$\gamma\subset M'$, then
\begin{equation} \label{eqn: warm-up 2 gamma}
\gamma\mapsto\gamma+\sum_i a_i\overrightarrow{\gamma_i},
\end{equation}
where all the orbits of $\overrightarrow{\gamma_i}$ are contained in
$M'$ and $A(\gamma)>
A(\overrightarrow{\gamma_i})=A_{\alpha}(\ora{\gamma_i})$.
\end{lemma}

\begin{proof}
This is straightforward, since both $J_\kappa$ and $J_{\kappa+1}$
are adapted to $\alpha$.  We easily see that $\omega=d(g(s)\alpha)$
is $\widetilde{J}_\kappa$-nonnegative whenever $g(s)$ is a positive,
monotonically increasing function.
\end{proof}

We also have the following lemma:

\begin{lemma}\label{lemma: warm-up 2 inclusion of algebras}
Given $K>0$, there exists $\kappa_0> 0$ such that for all
$\kappa\geq \kappa_0$,
$$\Psi_\kappa\colon  \mathcal{A}(M,\alpha,J_\kappa)\to
\mathcal{A}(M,\alpha,J_{\kappa+1})$$ maps $\gamma\mapsto\gamma$,
whenever $A_{\alpha}(\gamma)\leq K$.
\end{lemma}

\begin{proof}
Similar to that of Lemma~\ref{lemma: warm-up 2}, with one
difference: If there is a sequence of holomorphic curves
$$F_\kappa=(a_\kappa,f_\kappa)\colon  (\Sigma_\kappa,j_\kappa,
\mathbf{m}_\kappa)\to (\R\times
M^*,\widetilde{J}_\kappa)$$ in
$\mathcal{M}_g(\underline{\gamma}^+;\underline{\gamma}^-;\R\times
M^*,\widetilde{J}_\kappa)$, then there is a restriction of
$F_\kappa$ to a connected subsurface $\overline{\Sigma}_\kappa$ as
before, whose image is contained in $(W_N)^{TB}$.  If we write
$f_\kappa=(b_\kappa,v_\kappa)$, then each $v_\kappa$ is not
necessarily $(J_\kappa)_0$- or $(J_{\kappa+1})_0$-holomorphic.
However, since the sequence $v_\kappa|_{\overline{\Sigma}_\kappa}$
limits to a holomorphic curve in $W\cup (\R\times \Gamma')$, after
possibly taking a subsequence, the proof of Lemma~\ref{lemma:
warm-up 2} still carries over.  (Compare Section~\ref{subsection: ac}.)
\end{proof}

Putting Lemmas~\ref{lemma: warm-up 2}, \ref{lemma: warm-up 2
cobordism}, and \ref{lemma: warm-up 2 inclusion of algebras}
together, the direct limit argument in Section~\ref{subsection:
warm-up 1 direct limits} proves Theorem~\ref{thm: interval-fibered
extension}.

\section{Proof of Theorem~\ref{thm: sutured gluing}}
\label{section: proof} In this section we prove Theorem~\ref{thm:
sutured gluing}, i.e., the inclusion map under sutured manifold
gluing. The proof is a combination of the previous two sections.

\subsection{Stretching the neck}
\label{subsection: stretching the neck}

Keeping the notation from Section~\ref{subsection: gluing}, the main
theorem of this subsection is the following:

\begin{thm}\label{disappearing}
Suppose the orbits of $\underline{\gamma}^+$ and
$\underline{\gamma}^-$ are contained in $M'$. Then there exist
$\kappa>0$ and $n_0=n_0(\kappa)>0$ such that the tailored almost
complex structure $J'_{\kappa}$ on $(M')^*$ satisfies
\[ {\mathcal M}_{g}
(\underline{\gamma}^+; \underline{\gamma}^-;\R \times (M')^*,
J'_{\kappa}) = {\mathcal M}_{g}(\underline{\gamma}^+;
\underline{\gamma}^-; \R \times M^*_n, J_{\kappa,n}),\] for all
$n\geq n_0$.
\end{thm}

\begin{proof}
We analyze the convergence of a sequence of finite energy
holomorphic maps
$$F_n=(a_n,f_n)\colon (\Sigma_n, j_n, \mathbf{m}_n)
\rightarrow (\R\times M^*_n, J_{\kappa,n})$$ in ${\mathcal
M}_{g}(\underline{\gamma}^+; \underline{\gamma}^-;\R \times M^*_n,
J_{\kappa,n})$.

Our first reduction is to restrict the range of $F_n$ from $\R\times
M^*_n$ to $\R\times M^{(1)}_n$. Indeed, by Remark~\ref{rmk: tau is
plurisubharmonic}, any holomorphic map $F_n$ is disjoint from
$\R\times V^*$. From now on, we consider the sequence
$$F_n\colon  (\Sigma_n,j_n,\mathbf{m}_n)\to (\R\times M^{(1)}_n,
J_{\kappa,n}).$$

Recall that $M_e'$ is the infinite interval-fibered extension of
$M'$, obtained from $M'$ by attaching an interval bundle $S'\times
I$ over $S'=S_\infty-S$ (as given in Equations~(\ref{eqn: interval
fibered 1}) and (\ref{eqn: interval fibered 2})), and that
$(M_e')^{TB}$ is the partial completion of $M_e'$, obtained by
attaching just the Top and the Bottom. The theorem now follows from combining
the following Lemmas~\ref{lemma: equiv of moduli spaces} and \ref{stretching}.
\end{proof}

\begin{lemma} \label{lemma: equiv of moduli spaces}
For sufficiently large $\kappa>0$, the almost complex structure
$J'_{\kappa}$ tailored to $(M')^*$ satisfies
$$\mathcal{M}_g(\underline\gamma^+;\underline\gamma^-; \R\times
(M')^*, J'_{\kappa}) =
\mathcal{M}_g(\underline\gamma^+;\underline\gamma^-; \R\times
(M_e')^{TB},J_{\kappa,n}).$$
\end{lemma}

Observe that, by the construction in Section~\ref{subsection:
gluing}, the almost complex structure $J_{\kappa,n}$ does not depend
on $n$ when restricted to $M'_e$.

\begin{proof}[Proof of Lemma~\ref{lemma: equiv of moduli spaces}]
Similar to that of Lemma~\ref{lemma: warm-up 2}. The only difference
is that the region $S'\cup ([-1,0]\times \Gamma')$ analogous to $W$
is not compact, since $M'_e$ is an infinite interval-fibered
extension of $M'$. Hence the sequence
$$v_\kappa\colon  \overline{\Sigma}_\kappa \to S'\cup ([-1,0]\times
\Gamma')$$ may not converge, since $v_\kappa$ can be pushed towards
the ends of $S'$.  However, most of the analysis in \cite[Section 10]{BEHWZ} can be 
carried out for the portion of $\overline{\Sigma}_{\kappa}$ mapped into 
$[-1,0]\times \Gamma'$ by $v_{\kappa}$. In particular there must be a finite set of
disjoint separating curves in $\overline{\Sigma}_\kappa$ which converge to some Reeb 
orbits as negative punctures. We can assume without loss of generality that those curves 
are $\partial \overline{\Sigma}_{\kappa}$, therefore, for $\kappa$ big enough, $\int_{\partial
\overline{\Sigma}_{\kappa}} v_{\kappa}^* \beta <0$ (the negative sign because $\partial
\overline{\Sigma}_{\kappa}$ approaches a Reeb orbit as a negative puncture). Stokes 
Theorem  gives then $\int_{\overline{\Sigma}_{\kappa}} v_{\kappa}^* \beta ,0$, contradicting
the positivity of the symplectic area on holomorphic curves.
\end{proof}

\begin{lemma} \label{stretching}
Given $\kappa>0$, there exists $n_0> 0$ so that for all $n\geq n_0$,
$${\mathcal M}_{g}(\underline{\gamma}^+; \underline{\gamma}^-; \R
\times M^{(1)}_n, J_{\kappa,n})=
\mathcal{M}_g(\underline{\gamma}^+;\underline{\gamma}^-;\R\times
(M'_e)^{TB},J_{\kappa,n}).$$
\end{lemma}

\begin{proof}[Proof of Lemma~\ref{stretching}]
Suppose we are given a sequence $F_n\in {\mathcal
M}_{g}(\underline{\gamma}^+; \underline{\gamma}^-; \R \times
M^{(1)}_n, J_{\kappa,n}).$   If $A$ and $B$ are subsets of a metric
space $(X,d)$, we define the {\em distance from $A$ to $B$} to be
$\sup_{x\in A} d(x,B)$. This ``distance'' is not symmetric, but it is not a problem.
We apply the argument in
Proposition~\ref{noside} and Lemma~\ref{lemma: warm-up 1} to bound
the distance from $\op{Im}(F_n)$ to the interval-fibered extension
$(M'_e,\alpha_{n}, J_{\kappa,n})$. Although the interval-fibered
extension is noncompact, $\R\times M'_e$ has bounded geometry due to
the fact that the almost complex structures on the pieces
$P_+^c\times[2kn-1,2kn+1]$ are isomorphic (and similarly for
$P_-^c\times[-2kn-1,-2kn+1]$), so we can use the same compactness arguments of
Proposition \ref{noside} and Lemma \ref{lemma: warm-up 1}.
\end{proof}

\s\n{\bf The ECH case.}  We have the following analog of
Theorem~\ref{disappearing} in the ECH case:

\begin{thm}\label{disappearing ECH version}
Let $\{(\gamma_i,m_i)\}$ and $\{\gamma'_j, m'_j)\}$ be orbit sets in
$M'$. Then there is some $n_0 \in \N$ and some tailored almost
complex structure $J'$ on $(M')^*$ such that all flow lines in
$(\R\times M^*_n,J_{n})$ from $\{(\gamma_i,m_i)\}$ to $\{\gamma'_j,
m'_j)\}$ are contained in $(\R\times(M')^*,J')$ for all $n\geq n_0$.
\end{thm}

\begin{proof}
The proof of Theorem~\ref{disappearing ECH version} is similar to
that of Proposition~\ref{prop: ECH top and bottom bound}. We can
restrict to $(\R\times M_n^{(1)}, J_{\kappa,n})$ as in the contact
homology case, and apply the Gromov-Taubes compactness theorem in
dimension four to bound the distances of $\op{Im}(F_n)$ to
$(M^{(2)}_n, \alpha^{(2)}_{n}, J_{\kappa,n})$ and $(M'_e,\alpha_{n},
J_{\kappa,n})$.

The analog of Lemma~\ref{lemma: equiv of moduli spaces} is
straightforward and does not involve $\kappa$ since $\dim M'=3$ and
the projection of $J'$ to $J_0'$ on $S_\infty$ makes $S_\infty$ into
a Riemann surface: Let $F$ be a holomorphic map to $\R\times
(M'_e)^{TB}$, whose ends are contained in $\R\times (M')^*$. Also
let $S''=(S_\infty-S)\cup ([-1,0]\times \Gamma')$.  Then consider
the restriction of $F$ to $\R\times S''\times \R$, composed with the
projection to $S''$. It is a holomorphic map between Riemann
surfaces, and hence is an open mapping; on the other hand it is also
proper. We now obtain a contradiction since $S''$ is noncompact. We
conclude that $F$ does not intersect $\R\times S''\times \R$.
\end{proof}

\subsection{Continuation maps and direct limits}
\label{subsection: continuation maps and direct limits}

In this subsection we prove part of Theorem~\ref{thm: sutured
gluing}, namely we define the map $$\Phi\colon  HC(M',\alpha') \to
HC(M,\alpha)$$ and show that $\Phi$ is injective.

By Theorem~\ref{disappearing}, given $K>0$, there are $\kappa>0$ and
$n_0(\kappa)>0$ such that for all $n\geq n_0(\kappa)$ there is an
inclusion of chain complexes:
$$\Phi_{K,\kappa,n}\colon \mathcal{A}_{\leq K}(M',\alpha',J'_\kappa)
\hookrightarrow \mathcal{A}(M_n,\alpha_{n},J_{\kappa,n}).$$

The following lemma is essentially the same as the combination of
Lemmas~\ref{lemma: warm-up 1 cobordism} and ~\ref{lemma: warm-up 1
inclusion of algebras} --- the only difference is the bounded
geometry of the interval-fibered portion --- and its proof will be
omitted.

\begin{lemma} \label{lemma: cobordism}
Given $K>0$ and $\kappa>0$, there exists $n_0(\kappa)>0$ such that
for all $n\geq n_0(\kappa)$ there is a cobordism $(\R\times
M_n^*,J)$ which gives rise to a continuation map
$$\Psi_n\colon  \mathcal{A}(M_n,\alpha_{n},J_{\kappa,n})
\to \mathcal{A}(M_{n+1},\alpha_{n+1},J_{\kappa,n+1}),$$ with the
following properties:
\begin{enumerate}
\item if $A_{\alpha_{n}}(\gamma)\leq K$, then $\Psi_n(\gamma)=\gamma$;
\item if $\gamma\subset M'$, then $\Psi_n(\gamma)=\gamma+\sum_i
a_i\overrightarrow{\gamma_i},$ where all the orbits of
$\overrightarrow{\gamma_i}$ are contained in $M'$ and
$A_{\alpha_{n}}(\gamma)>
A_{\alpha_{n+1}}(\overrightarrow{\gamma_i})$.
\end{enumerate}
\end{lemma}

It follows that given $K>0$ there exist $\kappa>0$ and
$n_0(\kappa)>0$ such that if $n\geq n_0(\kappa)$, then the following
diagram of chain complexes commutes:

\begin{equation}
\begin{diagram}
\mathcal{A}_{\leq K}(M',\alpha',J'_{\kappa}) &
\rTo^{\Phi_{K,\kappa,n}} &
\mathcal{A}(M_n, \alpha_{n},J_{\kappa,n})  \\
& \rdTo^{\Phi_{K,\kappa,n+1}}  &  \dTo_{\Psi_{n }} \\
 &  &
\mathcal{A}(M_{n+1},\alpha_{n+1},J_{\kappa,n+1})  \\
\end{diagram}
\end{equation}

Next consider the continuation maps
$$i_{\kappa,\kappa+1}\colon  \mathcal{A}(M',\alpha',J'_{\kappa})
\to \mathcal{A}(M',\alpha',J'_{\kappa+1}),$$
$$j_{\kappa,\kappa+1}\colon  \mathcal{A}(M_n,\alpha_{n},J_{\kappa,n})
\to \mathcal{A}(M_n,\alpha_{n},J_{\kappa+1,n}),$$ which are defined
as in Lemma~\ref{lemma: warm-up 2 cobordism}. The map
$i_{\kappa,\kappa+1}$ sends $\gamma\mapsto \gamma + \sum_i
a_i\ora{\gamma_i},$ where $A_{\alpha'}(\gamma)>
A_{\alpha'}(\ora{\gamma_i})$.  This is due to the fact that the
contact form $\alpha'$ is the same for the domain and the range.
Similar considerations hold for $j_{\kappa,\kappa+1}$.  We then have
the following lemma:

\begin{lemma} \label{lemma: toad}
Given $K'>K>0$, there exists $\kappa_0>0$ such that for all
$\kappa\geq \kappa_0$ there exists $n(\kappa)$ such that for all
$n\geq n(\kappa)$, the following diagram commutes:
\begin{equation}
\begin{diagram}
\mathcal{A}_{\leq K}(M',\alpha',J'_{\kappa}) &
\rTo^{\Phi_{K,\kappa,n}}
& \mathcal{A}(M_{n},\alpha_{n},J_{\kappa,n}) \\
\dTo^{i_{\kappa,\kappa+1}} & & \dTo_{j_{\kappa,\kappa+1}} \\
\mathcal{A}_{\leq K'}(M',\alpha',J'_{\kappa+1})  &
\rTo^{\Phi_{K',\kappa',n}} &
\mathcal{A}(M_{n},\alpha_{n},J_{\kappa+1,n})  \\
\end{diagram}
\end{equation}
Moreover, if $\gamma\in \mathcal{A}_{\leq K}(M',\alpha',J'_\kappa)$,
then all the maps in the diagram send $\gamma\mapsto\gamma$.
\end{lemma}

\begin{proof}
The proof is similar to that of Lemma~\ref{lemma: warm-up 2
inclusion of algebras}. For sufficiently large $\kappa$, if
$A_{\alpha'}(\gamma)\leq K$, then
$i_{\kappa,\kappa+1}(\gamma)=\gamma$.  The same holds for
$j_{\kappa,\kappa+1}$, provided we choose $n$ to be sufficiently
large in response to $\kappa$.
\end{proof}

\s\n {\bf Definition of $\Phi$.} Suppose $n'>n > 0$. By composing $
\Psi_{n'-1}\circ \Psi_{n'-2}\circ \dots\circ \Psi_n$, we obtain a
chain map
$$\Psi_{n,n'}\colon  \mathcal{A}(M_n,\alpha_{n},J_{\kappa,n})\to
\mathcal{A}(M_{n'},\alpha_{n'},J_{\kappa,n'}),$$ where
$\gamma\subset M'$ is mapped to $\gamma+\sum_i
a_i\overrightarrow{\gamma_i}$ with orbits of
$\overrightarrow{\gamma_i}$ contained in $M'$ and
$A_{\alpha_{n'}}(\overrightarrow{\gamma_i})<
A_{\alpha_{n}}(\gamma)$. Similarly, if $\kappa'>\kappa>0$, then we
can define $i_{\kappa,\kappa'}$ and $j_{\kappa,\kappa'}$ by
composing chain maps of type $i_{\kappa,\kappa+1}$ and
$j_{\kappa,\kappa+1}$. Given $K'>K>0$, there exists $\kappa_0$ such
that if $\kappa'>\kappa\geq \kappa_0$ and $n\geq n(\kappa,\kappa')$,
then the chain maps $\Psi_{n,n'}$, $i_{\kappa,\kappa'}$, and
$j_{\kappa,\kappa'}$ fit into the following commutative diagram of
chain complexes:
\begin{equation} \label{equation: diagram}
\begin{diagram}
\mathcal{A}_{\leq K}(M',\alpha',J'_{\kappa}) &
\rTo^{\Phi_{K,\kappa,n}} &
\mathcal{A}(M_n, \alpha_{n},J_{\kappa,n})  \\
\dTo(0,4)^{i_{\kappa,\kappa'}} & \rdTo^{\Phi_{K,\kappa,n'}}
&  \dTo_{\Psi_{n,n'}} \\
& & \mathcal{A}(M_{n'},\alpha_{n'},J_{\kappa,n'}) \\
& & \dTo_{j_{\kappa,\kappa'}} \\
\mathcal{A}_{\leq K'}(M',\alpha',J'_{\kappa'})  &
\rTo^{\Phi_{K',\kappa',n'}} &
\mathcal{A}(M_{n'},\alpha_{n'},J_{\kappa',n'}).  \\
\end{diagram}
\end{equation}
Now,
$$HC(M',\alpha')= \lim_{K\to \infty} HC_{\leq K}
(M',\alpha',J'_{\kappa(K)}),$$ since the contact form $\alpha'$ does
not vary while $K\to \infty$.  The diagram induces the map
$$\Phi:HC(M',\alpha')\to \lim_{\kappa\to\infty}
HC(M_{n(\kappa)},\alpha_{n(\kappa)},J_{\kappa,n(\kappa)})$$ on the
level of homology. Moreover, the direct limit
$\lim_{\kappa\to\infty}
HC(M_{n(\kappa)},\alpha_{n(\kappa)},J_{\kappa,n(\kappa)})$ is
isomorphic to any single
$HC(M_{n(\kappa)},\alpha_{n(\kappa)},J_{\kappa,n(\kappa)})$.

\s\n {\bf Injectivity of $\Phi$.} Refer to Diagram~(\ref{equation:
  diagram}). Suppose $a$ is a cycle in $\mathcal{A}_{\leq
  K}(M',\alpha',J'_{\kappa})$ and $a=\bdry b$ for some $b\in
\mathcal{A}(M_n,\alpha_{n}, J_{\kappa,n})$ with $n$ sufficiently
large. Note that for homological reasons, all the orbits of $b$ must
be contained in $M'$. Then $\Psi_{n,n'}$ sends $a\mapsto a$ and
$b\mapsto b+\sum_i b_i$ by Lemma~\ref{lemma: cobordism}, where all the
orbits of $b_i$ are contained in $M'$ and $A_{\alpha_{n'}}(b_i)<
A_{\alpha_{n}}(b)$, where the latter means the maximum over all the
monomials of $b$.  Hence $a=\bdry (b+\sum_i b_i)$ in
$\mathcal{A}(M_{n'},\alpha_{n'},J_{\kappa,n'})$. For sufficiently
large $n'$, if we apply $j_{\kappa,\kappa'}$ to $a=\bdry (b+\sum_i
b_i)$, we obtain $a=\bdry (b+\sum_i b_i')$ in
$\mathcal{A}(M_{n'},\alpha_{n'}, J_{\kappa',n'})$ with
$A_{\alpha_{n'}}(b_i')< A_{\alpha_{n'}}(b)$. Now, if we let $K'>
A_{\alpha_{n'}}(b)$, then there is a sufficiently large $n'$ such that
the map $\Phi_{K',\kappa',n'}$ is injective by
Theorem~\ref{disappearing}. Hence $a=\bdry (b+\sum_i b_i')$ in
$\mathcal{A}_{\leq K'}(M',\alpha', J'_{\kappa'})$. This proves the
injectivity of $\Phi$.

\subsection{The inclusion map is well-defined}
\label{subsection: naturality}

In this subsection we prove that the inclusion map
\[
\Phi\colon  HC(M',\Gamma',\xi') \to HC(M,\Gamma,\xi)
\]
does not depend on the choices made to define it.  By this we mean the
following:

\begin{prop} \label{prop: naturality sutured gluing case}
Let $(\alpha')^0$ and $(\alpha')^1$ be two contact forms which are
adapted to the sutured contact manifold
$(M',\Gamma',U(\Gamma'),\xi')$, and let $\alpha^0_{n}$,
$\alpha^1_{n}$ be their extensions to $M_n$. Then there is a
commutative diagram:
\begin{equation} \label{equation: commutative diagram for
naturality}
\begin{diagram}
HC(M',(\alpha')^0) & \rTo^{\Phi^0} &
\lim_{\kappa\to\infty} HC(M_{n(\kappa)},
\alpha^0_{n(\kappa)}, J^0_{\kappa,n(\kappa)})  \\
\dTo^{\Theta'} & &  \dTo_{\Theta} \\
HC(M',(\alpha')^1)  & \rTo^{\Phi^1} & \lim_{\kappa\to\infty}
HC(M_{n(\kappa)},
\alpha^1_{n(\kappa)},J^1_{\kappa,n(\kappa)}),\\
\end{diagram}
\end{equation}
where the $\Phi^i$ are the inclusion maps defined in
Section~\ref{subsection: continuation maps and direct limits} and
$\Theta'$ is the continuation map given in Section~\ref{subsection:
ac}.
\end{prop}

\begin{proof}
Let $(\alpha')^0$ and $(\alpha')^1$ be two contact forms which are
adapted to $(M',\Gamma',U(\Gamma'),\xi')$, and let $\alpha^0_{n}$
and $\alpha^1_{n}$ be their extensions to $M_n$. Also let
$(J'_{\kappa})^0$ and $(J'_{\kappa})^1$ be the almost complex
structures on $M'$ corresponding to $(\alpha')^0$ and $(\alpha')^1$,
as defined in Section~\ref{subsection: gluing}, and let
$J^0_{\kappa,n}$ and $J^1_{\kappa,n}$ be their extensions to $M_n$.
Also write $(\beta')^i_0=(\alpha')^i|_{\bdry R_+(\Gamma')}$ and
$(\beta')^i=(\alpha')^i|_{R_+(\Gamma')}$.

Since $(\alpha')^0$ and $(\alpha')^1$ are contact forms for the same
contact structure $\xi'$, we can write $(\alpha')^0=f \cdot (\alpha')^1$,
where $f$ is constant in a neighborhood of the sutures.  Moreover, we
can write $(\beta')^0_0=C (\beta')^1_0$ for some constant $C$, which
we take to be equal to $1$ for simplicity. Also, if we identify the manifolds $M_n$
using the appropriate diffeomorphisms, then we can write $\alpha^0_n=f_n \alpha^1_n$.

Choose a $1$-parameter family $f^\rho$, $\rho\in[0,1]$, where
$f^0=f$ and $f^1=1$. We then use the family $f^\rho(\alpha')^1$
to construct a symplectic cobordism and an almost complex structure
as in Section~\ref{subsection: ac} and to define a continuation map
\[
\Theta'_\kappa\colon  \mathcal{A} (M',(\alpha')^0,(J'_{\kappa})^0) \to
\mathcal{A} (M',(\alpha')^1,(J'_{\kappa})^1).
\]
Next choose a $1$-parameter family $f_n^\rho$, $\rho\in[0,1]$, where
$f_n^0=f_n$ and $f_n^1=1$ and $f_n^\rho$ extends $f^\rho$. Using
$f_n^\rho\alpha^1_n$, we obtain a continuation map
\[
\Theta^{\kappa}_{n}\colon \mathcal{A}(M_{n},\alpha^0_{n},
J^0_{\kappa,n})\to \mathcal{A}(M_{n},\alpha^1_{n},J^1_{\kappa,n}).
\]

Let $K>0$. Then there exists $K'>0$ such that
\[
\Theta'_\kappa(\mathcal{A}_{\leq K}(M',(\alpha')^0,(J'_{\kappa})^0))
\subset\mathcal{A}_{\leq K'}(M',(\alpha')^1,(J'_{\kappa})^1).
\]
For sufficiently large $\kappa$, there exists $n(\kappa)$ such that if
$n\geq n(\kappa)$ then the following diagram is commutative:
\begin{equation}
\begin{diagram}
\mathcal{A}_{\leq K}(M',(\alpha')^0,(J'_{\kappa})^0) &
\rTo^{\Phi^0_{K,\kappa,n}} &
\mathcal{A}(M_{n}, \alpha_{n}^0,J^0_{\kappa,n})  \\
\dTo^{\Theta'_\kappa} & &  \dTo_{\Theta_{n}^{\kappa}} \\
\mathcal{A}_{\leq K'}(M',(\alpha')^1,(J'_{\kappa})^1)  &
\rTo^{\Phi^1_{K',\kappa,n}} &
\mathcal{A}(M_{n},\alpha_{n}^1,J^1_{\kappa,n}).  \\
\end{diagram}
\end{equation}
The proof follows from combining Step~1 of Section~\ref{subsection:
ac} and Theorem~\ref{disappearing}.

Given $\kappa'> \kappa> 0$ and $n'> n> 0$, let
\[
(\Psi_{n,n'}^{\kappa,\kappa'})^0=j^0_{\kappa,\kappa'}
\circ \Psi^0_{n,n'}\colon  \mathcal{A}(M_n,
\alpha^0_{n},J^0_{\kappa,n})\to \mathcal{A}(M_{n'}, \alpha^0_{n'},
J^0_{\kappa',n'}),
\]
be the continuation map from last section; similarly define
$(\Psi_{n,n'}^{\kappa,\kappa'})^1$.

In order to take direct limits, we need to verify that the diagrams
\begin{equation}
\begin{diagram}
\mathcal{A}_{\leq K}(M',(\alpha')^0,(J'_{\kappa})^0) &
\rTo^{j_{\kappa,\kappa'}^0} &
\mathcal{A}_{\leq K''}(M',(\alpha')^0, (J'_{\kappa'})^0) \\
\dTo^{\Theta'_\kappa} & &  \dTo_{\Theta'_{\kappa'}} \\
\mathcal{A}_{\leq K'}(M',(\alpha')^1,(J'_{\kappa})^1)  &
\rTo^{j_{\kappa,\kappa'}^1} &
\mathcal{A}_{\leq K'''} (M',(\alpha')^1,(J'_{\kappa'})^1) \\
\end{diagram}
\end{equation}
and
\begin{equation}
\begin{diagram}
\mathcal{A}(M_{n}, \alpha_{n}^0,J^0_{\kappa,n}) &
\rTo^{(\Psi_{n,n'}^{\kappa,\kappa'})^0} &
\mathcal{A}(M_{n'}, \alpha_{n'}^0,J^0_{\kappa',n'}) \\
\dTo^{\Theta^\kappa_n} & &  \dTo_{\Theta^{\kappa'}_{n'}} \\
\mathcal{A}(M_{n}, \alpha_{n}^1,J^1_{\kappa,n})  &
\rTo^{(\Psi_{n,n'}^{\kappa,\kappa'})^1} &
\mathcal{A}(M_{n'}, \alpha_{n'}^1,J^1_{\kappa',n'}) \\
\end{diagram}
\end{equation}
commute up to chain homotopy. This follows from the fact that, in
either case, the symplectic cobordisms corresponding to the
compositions (together with their almost complex structures) are
homotopic.  Taking direct limits, we obtain Diagram~(\ref{equation:
  commutative diagram for naturality}).
\end{proof}

\subsection{The ECH case} \label{subsection: ECH case}

In this section we explain how to prove Theorem~\ref{thm: sutured gluing}(2),
assuming the existence of appropriate cobordism maps on sutured ECH,
analogous to the cobordism maps on ECH of closed contact 3-manifolds
defined in \cite{HT3}.

First observe that the ECH setup is much simpler since we do not need
to use the parameter $\kappa$. Let $C(M',\alpha',J')$ be the ECH chain
complex ($\F=\Z/2\Z$-vector space) generated by the orbits sets of
$R_{\alpha'}$ and whose boundary map counts $J'$-holomorphic
curves. Also let $C_0(M_n,\alpha_n,J_{n})$ be the subcomplex of the
ECH chain complex $C(M_n,\alpha_n,J_{n})$ which counts orbit sets
which have zero intersection with $S_\infty$. As before,
$C(M',\alpha',J')$ and the subcomplexes $C_0(M_n,\alpha_n,J_{n})$ for
different $n$ are all isomorphic as $\F$-vector spaces, although not
necessarily as chain complexes.

Fix $n>0$.  By analogy with \cite{HT3}, it is conjectured that given
$K>0$, for sufficiently large $K'$, the cobordism in Lemma~\ref{lemma:
  cobordism} induces a chain map
\[
\Psi^{K,K'}_{n,n+1}\colon C_{\leq K} (M_n,\alpha_n,J_{n}) \to C_{\leq
  K'}(M_{n+1},\alpha_{n+1}, J_{n+1}),\] which depends on some choices,
but which has the following two properties: First,
$\Psi^{K,K'}_{n,n+1}$ is given by some unspecified count of (possibly
broken) holomorphic curves between orbit sets $\ora{\gamma}$ for
$(M_n,\alpha_n)$ and $\ora{\gamma}'$ for $(M_{n+1},\alpha_{n+1})$, in
the cobordism $(\R\times M_n^*,J)$ given in the proof of
Lemma~\ref{lemma: cobordism}.  Second, on the subset $\R\times M'$
where the almost complex structure is cylindrical, trivial holomorphic
cylinders over closed Reeb orbits are always counted in
$\Psi^{K,K'}_{n,n+1}$.

We now claim that the following commutative diagram of chain
complexes exists:
\begin{equation} \label{diagram: ECH case}
\begin{diagram}
C_{\leq K}(M',\alpha',J') & \rTo &
(C_0)_{\leq K''} (M_n,\alpha_n,J_{n})  \\
\dTo & &  \dTo_{\Psi^{K'',K'''}_{n,n+1}} \\
C_{\leq K'}(M',\alpha',J')  & \rTo &
(C_0)_{\leq K'''} (M_{n+1},\alpha_{n+1},J_{n+1})\\
\end{diagram}
\end{equation}
Here we are given $K'> K > 0$; we choose $n=n(K) > 0$, $K''\geq K$ and
$K'''=K'''(n,K'')\geq K'$. First note that $\Psi^{K'',K'''}_{n,n+1}$
is given by some count of holomorphic curves in the cobordism. On the
other hand, Theorem~\ref{disappearing ECH version} shows that, if
$\ora{\gamma}$ is a generator of $(C_0)_{\leq K''}
(M_n,\alpha_n,J_{n})$ which comes from $C_{\leq K}(M',\alpha',J')$,
then no holomorphic subvariety in $(\R\times M_n^*,J)$ which flows
from $\ora{\gamma}$ can cross the ``neck region'', i.e., cross
$S_\infty$, provided $n$ is chosen to be sufficiently
large. Furthermore, once we know that no curve from $\ora{\gamma}$
crosses the ``neck region'', we are now in the symplectization
portion, and we only have trivial cylinders.  Hence
$\Psi^{K'',K'''}_{n,n+1}$ maps $\ora\gamma\mapsto \ora\gamma$ if
$\ora\gamma$ comes from $C_{\leq K}(M',\alpha',J')$.  This proves the
commutativity of Diagram~(\ref{diagram: ECH case}).

For other $\ora\gamma$ in $(C_0)_{\leq K''} (M_n,\alpha_n,J_{n})$,
considerations of $\omega$ in Lemma~\ref{lemma: cobordism}, together
with the fact that $\Psi^{K'',K'''}_{n,n+1}$ is some count of
holomorphic curves, proves that $\Psi^{K'',K'''}_{n,n+1}$ maps
$\ora\gamma$ to $\ora\gamma$ plus terms with lower action.  (Note
that $\omega$ is not the exact symplectic form which gives the exact
symplectic cobordism, but is just some taming form for $J$.)

Arguing as in the contact homology case, we obtain an inclusion:
\[
\lim_{K\to \infty} ECH_{\leq K} (M',\alpha',J') \hookrightarrow
\lim_{n\to\infty} (ECH_0)_{\leq K''(n)} (M_n, \alpha_n, J_{n}),
\]
where $ECH_0$ is the homology for $C_0$.  More precisely, the limit on
the right-hand side is over $n\to\infty $ and $K''(n)$ is a sequence
$\to\infty$ which depends on both $n$ and $K''(n-1)$. The left-hand
side is $ECH(M',\Gamma',\xi')$, and the right-hand side equals
$ECH_0(M,\Gamma,\xi)$, under our conjecture that $ECH(M,\Gamma,\xi)$
does not depend on the choice of contact form or almost complex structure.

\section{Gluing along a convex submanifold}
\label{section: convex gluing}

Let $(M,\Gamma,\xi)$ be a contact manifold with convex boundary and
let $S \subset M$ be a closed convex submanifold with dividing set
$\Gamma_S$. Also let $(M',\Gamma',\xi')$ be the sutured contact
manifold obtained by splitting $M$ along $S$ and applying
Lemma~\ref{l:suture}.

The goal of this section is to prove the following:

\begin{thm} \label{thm: convex gluing restated}
There is a canonical map
\[
\Phi\colon  HC(M', \Gamma', \xi') \to  HC(M,\Gamma, \xi).
\]
\end{thm}

In this section we will treat the case of contact homology; the
proofs for embedded contact homology are similar.

According to Lemma~\ref{lemma: convexgluing}, there is a contact
$1$-form $\alpha'$ which is adapted to the sutured contact manifold
$(M',\Gamma', U(\Gamma'), \xi')$ and an extension to
$(M_n,\alpha_{n,g_0,g_1})$ which is contactomorphic to
$(M,\Gamma,U(\Gamma),\xi)$.  Here $n>0$ and $g_0,g_1$
are functions depending on $n$. In this section we assume that $V\times D^2$ is the
union of all the fillings of $M_n'$, unlike in
Section~\ref{subsection: gluing along convex submanifolds} where it
was assumed to be just one connected component. It is clear that
there is an inclusion
\[
\Phi\colon  \mathcal{A}(M', \Gamma', \alpha',J') \to
\mathcal{A}(M_n,\alpha_{n,g_0,g_1},J_{n,g_0,g_1});
\]
we would like
to prove that $\Phi$ is a chain map.

Our first task is to prove that, given $K>0$, for sufficiently large
$n$ there exist $g_0,g_1$ so that the inclusion
\[
\Phi_K\colon  \mathcal{A}_{\leq K}(M',\alpha',J')\to \mathcal{A}
(M_n,\alpha_{n,g_0,g_1},J_{n,g_0,g_1})
\]
is a chain map, i.e., $\Phi_K\circ \bdry'=\bdry\circ \Phi_K$, where
$\bdry$ and $\bdry'$ are boundary maps for $M_n$ and $M'$.  For this,
it suffices to show the following:

\begin{lemma} \label{lemma: inclusion convex gluing case} Suppose the
  orbits of $\underline\gamma$ are contained in $(M',\alpha')$.  Then
  for sufficiently large $n>0$ there exist $g_0,g_1$ (depending on
  $n$) such that
\[
\mathcal{M}_g(\underline\gamma;\underline\gamma';\R\times M_n^*,J_{n,g_0,g_1})
= \mathcal{M}_g(\underline\gamma;\underline\gamma';\R\times
(M')^*,J'),
\]
if the orbits of $\underline\gamma'$ are contained in
$(M',\alpha')$, and
\[
\mathcal{M}_g(\underline\gamma;\underline\gamma';\R\times
M_n^*,J_{n,g_0,g_1})=\emptyset,
\]
otherwise.
\end{lemma}

\begin{proof}
Let $F = (a,f) \colon (\Sigma,j,\mathbf{m}) \to (\R \times
M_n^*,J_{n,g_0,g_1})$ be an element of
$\mathcal{M}_g(\underline\gamma;\underline\gamma';\R\times
M_n^*,J_{n,g_0,g_1})$. It suffices to show the following: \be
\item There is no $F$ from $\underline\gamma$ to $\underline\gamma'$,
  where some component of $\underline\gamma'$ is not strictly
  contained in $M'$.
\item No $F$ from $\underline\gamma$ to $\underline\gamma'$ with all
  components of $\underline\gamma'$ in $M'$ has $\op{Im}(f)$ which
  nontrivially intersects $V\times D^2$, $S_1^+\times\{{n\over 2}\}$
  or $S_1^-\times\{{n\over 2}\}$. \ee (1) is easy since we can choose
  $n$, $g_0,g_1$ so that all the closed orbits in
  $(M_n,\alpha_{n,g_0,g_1})$ which are not in $(M',\alpha')$ have
  arbitrarily large action, see Lemma~\ref{lemma: convexgluing}.

We now argue (2). First take $n$ sufficiently large so that any
$F\in\mathcal{M}_g(\underline\gamma;\underline\gamma';\R\times
M_n^*,J_{n,g_0,g_1})$ with image inside $\R\times (M'_n)^*$ has
image inside $\R\times (M')^*$. This can be done by
Lemma~\ref{lemma: warm-up 1}.  In addition to $n$, the functions
$g_0,g_1$ will depend on the choice of $B> 0$.  In particular, we
take $B$ so that $\operatorname{Im}(g_0,g_1)$
contains the line segment between $(a,1)$ and $(a,B)$. Let
$U_B\subset V\times D^2$ be the subset consisting of points
$(x,r,\theta)$, where $(g_0(r),g_1(r))$ is contained in this
line segment. Also let $\beta_0'$ be the restriction of $\alpha'$ to
$\bdry R_+(\Gamma')$. On $U_B$,
$\alpha_{n,B}=\alpha_{n,g_0(B),g_1(B)}$ is of the form
$ad\theta+\beta_B$, where $\beta_B$ is a symplectization of
$\beta_0'$ in the $-r$-direction. Alternatively, we write
$M''_{n,B}= M'_n\cup U_B$ and use coordinates $(t,\tau,x)$ on
\[
U_B\simeq (\R/a\R) \times [0,\tau_B]\times V
\]
so that $\alpha_{n,B}=dt+e^\tau\beta_0'(x)$. Let $\widehat{S}_i^+$ be
the extension of $S_i^+$ to $M''_{n,B}$ so that $\bdry
\widehat{S}_i^+\subset \bdry M''_{n,B}$. Let $(J_B)_0$ be an almost
complex structure on $\widehat{S}_i^+$ which is adapted to the
symplectization $d(e^\tau\beta_0'(x))$, and let $J_{n,B}$ be a
tailored almost complex structure on $M_n^*$ whose projection to
$\widehat{S}_i^+$ equals $(J_B)_0$.

We claim that, for sufficiently large $B>0$, all holomorphic maps
$F_B=(a_B,f_B)\in \mathcal{M}_g(\underline\gamma;
\underline\gamma';\R\times M_n^*,J_{n,B})$ are disjoint from
$\R\times U_{B}$. (Note that, by the strict plurisubharmonicity of
$\tau$, $F_{n,B}$ is disjoint from $U_B$ if and only if $F_{n,B}$ is
disjoint from $\bdry M''_{n,B}$.) The argument is similar to that of
Lemma~\ref{lemma: warm-up 2}, only easier. Arguing by contradiction,
suppose there is a sequence $F_{B_i}=(a_{B_i},f_{B_i})$ where
$f_{B_i}$ nontrivially intersects $U_{B_i}$ and $B_i\to\infty$.
Writing $v_{B_i}$ as the projection of $f_{B_i}$ to
$[0,\tau_{B_i}]\times V$ whenever applicable, in the limit as
$B_i\to\infty$ we eventually obtain a finite energy cylinder
$v_\infty\colon  [0,\infty)\times S^1\to [0,\infty)\times V$. However,
this contradicts the energy bound as follows:  First, the $F_{B_i}$
have bounded $d\alpha_{n,B_i}$-energy since $\underline\gamma$ and
$\underline{\gamma}'$ are fixed. On $U_{B_i}$,
$d\alpha_{n,B_i}=d(e^\tau\beta_0')$, and a cylinder over a Reeb
orbit of $\beta_0'$ has unbounded $d(e^\tau\beta_0')$-area, a
contradiction.

Once we know that $F_B$ is disjoint from $\R\times U_B$, by our
choice of $n\gg 0$, $F_{B}$ has image inside $\R\times(M')^*$ by Lemma 
\ref{lemma: warm-up 1}. This
concludes the proof of Lemma~\ref{lemma: inclusion convex gluing
case}.
\end{proof}

\s\n {\bf Case of dimension three.}  We give an alternate, more
straightforward proof of Lemma~\ref{lemma: inclusion convex gluing
case} when $\dim M=3$.

\begin{lemma} \label{lemma: positivity of intersections}
Let $F = (a,f) \colon (\Sigma,j,\mathbf{m}) \to (\R \times
M_n^*,J_{n,g_0,g_1})$ be a holomorphic map which is asymptotic to
$\underline{\gamma}$ at $s\to +\infty$ and asymptotic to
$\underline{\gamma}'$ at $s\to -\infty$. If $\dim M=3$ and the
orbits of $\underline{\gamma}$ and $\underline{\gamma}'$ lie in
$M'$, then the image of $f$ is disjoint from $V\times\{0\}$ if $n$ is
sufficiently large.
\end{lemma}

\begin{proof}
By Lemma~\ref{lemma: convexgluing}, the contact form
$\alpha_{n,g_0,g_1}$ has the property that every connected component
of $V\times\{0\}$ is a periodic orbit of the Reeb flow.  Hence all
intersection points between $V\times\{0\}$ and $C=\op{Im}(f)$ are
positive, by the positivity of intersections in dimension four.
Observe that $V\times\{0\}$ is the oriented boundary of a surface
$S$ which is an extension of $R_+(\Gamma')\subset M'$ to $M_n$, and
$R_+(\Gamma')$ is disjoint from $\underline{\gamma}\cup
\underline{\gamma}'$.  We may assume without loss of generality that
$C\pitchfork S$.  If $C$ has nontrivial intersection with $\bdry S$,
then there is a properly embedded arc $c$ on $S$ which connects from
$\bdry S$ to itself. However, $C$ and $V\times\{0\}=\bdry S$
intersect positively at one endpoint of $c$ and negatively at the
other endpoint, a contradiction.  We conclude that the image of $f$
is disjoint from $V\times \{0\}$.
\end{proof}

We claim that $C=\op{Im}(f)$ is contained in $(M_n')^*$. Assume for
convenience that $V$ is connected. By Lemma~\ref{lemma: positivity of
  intersections}, $C$ is disjoint from $V\times \{0\}$. Let $T_{r=1}$
be the torus $\{r=1\}\subset V\times D^2$. It then follows that $C\cap
T_{r=1}$ is homologous to $\emptyset$ on $T_{r=1}$. On the other hand,
on $T_{r=1}$ the Reeb vector field is parallel to $\bdry_\theta$ and
$C$ must be positively transverse to $\bdry_\theta$ by intersection
positivity. (By a slight perturbation if necessary, we may assume that
$C\cap T_{r=1}$ is an immersion.)  If we take an oriented
identification $T_{r=1}= \R^2/\Z^2$ with orientation on $T_{r=1}$
equal to the boundary orientation of $V\times D^2$ and choose
coordinates $({\theta\over 2\pi},x)$, and we set $\Sigma'=\Sigma-
f^{-1}(V\times D^2)$, then $dx$ is everywhere positive on $f|_{\bdry
  \Sigma'}$. Since $f|_{\bdry\Sigma'}$ is not homologically zero if
$C$ intersects $T_{r=1}$, we conclude that $C$ does not enter $V\times
D^2$. Now we can apply the argument in Lemma~\ref{lemma: warm-up 1} to
show that, for sufficiently large $n$, no $f$ intersects
$S_1^+\times\{{n\over 2}\}$ and $S_1^-\times\{{n\over 2}\}$ as
described in Section~\ref{subsection: gluing along convex
  submanifolds}.  Hence we can view $F$ as sitting inside $\R\times
(M')^*$.

\s Returning to the proof of Theorem~\ref{thm: convex gluing
restated}, we now define two chain maps $\Psi^n_{B,B+1}$ and
$\Psi^{n,n+1}_B$, where $n$ and $B$ are positive integers:

\s\n {\bf The first chain map.} Given contact forms $\alpha_{n,B}$
and $\alpha_{n,B+1}$ on $M_n$, arrange them via an isotopy  so that
the forms agree on $M''_{n,B}$ and the contact structures agree on
$M_n-M''_{n,B}$.  We also assume that $J_{n,B}$ and $J_{n,B+1}$
agree on $M''_{n,B}$, and are induced by $J'$ on $M'$. Then
interpolating between $\alpha_{n,B}$ and $\alpha_{n,B+1}$ and
between the almost complex structures gives us a symplectic
cobordism and a corresponding chain map:
$$\Psi^n_{B,B+1}\colon  \mathcal{A}(M_n, \alpha_{n,B}, J_{n,B})
\to \mathcal{A}(M_{n},\alpha_{n,B+1},J_{n,B+1}).$$ An argument
identical to that of Lemma~\ref{lemma: inclusion convex gluing case}
shows that, given $K>0$, for sufficiently large $n$ there exists
$B_0(n)$ such that for $B\geq B_0(n)$ the following diagram
commutes:
\begin{equation}
\begin{diagram}
\mathcal{A}_{\leq K}(M',\alpha',J') & \rTo^{\Phi_{K,n,B}} &
\mathcal{A}(M_n, \alpha_{n,B},J_{n,B})  \\
& \rdTo^{\Phi_{K,n,B+1}}  &  \dTo_{\Psi^n_{B,B+1}} \\
 &  &
\mathcal{A}(M_{n},\alpha_{n,B+1},J_{n,B+1})  \\
\end{diagram}
\end{equation}
In particular, if $\gamma\in \mathcal{A}_{\leq K}(M',\alpha',J')$,
then
$\Psi^n_{B,B+1}\circ\Phi_{K,n,B}(\gamma)=\Phi_{K,n,B+1}(\gamma)$.

\s\n {\bf The second chain map.} Given $\alpha_{n,B}$ on $M_n$ and
$\alpha_{n+1,B}$ on $M_{n+1}$, we take a diffeomorphism $i\colon 
M_n\stackrel\sim\to M_{n+1}$ which is similar to the one defined in
the paragraph before Lemma~\ref{lemma: warm-up 1 cobordism}: it
takes $M'$ to $M'$ by the identity and sends
$M''_{n,B}\stackrel\sim\to M''_{n+1,B}$, while stretching
$M''_{n,B}-M'$ in the $\bdry_t$-direction (i.e., the Reeb direction)
so that $i^*(dt+\beta')=df+\beta'$ and $|{\bdry f\over \bdry
t}-1|=O({1\over n})$.  Also assume that $J_{n,B}$ and $J_{n+1,B}$
agree with $J'$ on $M'$ and project to the same almost complex
structure on $\widehat{S}_i^+$.  Interpolating between
$\alpha_{n,B}$ and $i^*\alpha_{n+1,B}$, we obtain:
\[
\Psi^{n,n+1}_B\colon  \mathcal{A}(M_n, \alpha_{n,B}, J_{n,B})
\to \mathcal{A}(M_{n+1},\alpha_{n+1,B},J_{n+1,B}).
\]

Given $K>0$, for sufficiently large $n$ there exists $B_0(n)$ such
that for $B\geq B_0(n)$ the following diagram commutes:
\begin{equation}
\begin{diagram}
\mathcal{A}_{\leq K}(M',\alpha',J') & \rTo^{\Phi_{K,n,B}} &
\mathcal{A}(M_n, \alpha_{n,B},J_{n,B})  \\
& \rdTo^{\Phi_{K,n+1,B}}  &  \dTo_{\Psi^{n,n+1}_B} \\
&  & \mathcal{A}(M_{n+1},\alpha_{n+1,B},J_{n+1,B})  \\
\end{diagram}
\end{equation}
Moreover, if $\gamma\in \mathcal{A}_{\leq K}(M',\alpha',J')$, then
$\Psi^{n,n+1}_B\circ\Phi_{K,n,B}(\gamma)=\Phi_{K,n+1,B}(\gamma)$.
First we pick $n$ so that any $F\in \mathcal{M}_g(\underline\gamma;
\underline\gamma';\R\times M_n^*,J_{n,B})$ with image inside
$\R\times (M'_n)^*$ has image inside $\R\times (M')^*$, as in
Lemma~\ref{lemma: warm-up 1 cobordism}. Next, we pick $B_0(n)$ to
bound the $\tau$-direction as in the proof of Lemma~\ref{lemma:
inclusion convex gluing case};

\s\n {\bf Definition of the map $\Phi$.} By repeatedly composing the
maps of type $\Psi^{n}_{B,B+1}$ and $\Psi^{n,n+1}_B$, we obtain the
chain map
\[
\Psi^{n,n'}_{B,B'}=\Psi^{n'-1,n'}_{B'}\circ\dots\circ \Psi^{n,n+1}_{B'}
\circ \Psi^{n}_{B'-1,B'}\circ\dots\circ \Psi^{n}_{B,B+1}.
\]
Here $i_{K,K'}$ is the natural inclusion.  Given $K'> K>0$, there
exist $n'> n> 0$ and $B'=B'(n')> B=B(n)> 0$ so that $\Phi_{K,n,B}$ and
$\Phi_{K',n',B'}$ both map $\gamma\mapsto\gamma$ and the following
diagram commutes:
\begin{equation}
\begin{diagram}
\mathcal{A}_{\leq K}(M',\alpha',J') & \rTo^{\Phi_{K,n,B}} &
\mathcal{A}(M_n, \alpha_{n,B},J_{n,B})  \\
\dTo^{i_{K,K'}} & &  \dTo_{\Psi^{n,n'}_{B,B'}} \\
\mathcal{A}_{\leq K'}(M',\alpha',J')  & \rTo^{\Phi_{K',n',B'}} &
\mathcal{A}(M_{n'},\alpha_{n',B'},J_{n',B'})  \\
\end{diagram}
\end{equation}

Taking direct limits, we have
\[
\Phi\colon \lim_{K\to \infty} HC_{\leq K}(M',\alpha',J') \to
\lim_{n\to \infty} HC(M_n, \alpha_{n,B(n)}, J_{n,B(n)}).
\]
Since the $HC(M',\alpha')=\lim_{K\to
\infty} HC_{\leq K}(M',\alpha')$ and the maps $\Psi^{n,n'}_{B,B'}$
are always isomorphisms, we have defined the map $\Phi$ in
Theorem~\ref{thm: convex gluing restated}.

\s\n {\bf Proof that $\Phi$ is independent of choices.} Let
$(\alpha')^i$, $i=0,1$, be two contact forms which are adapted to
$(M',\Gamma',U(\Gamma'),\xi')$ and let $(J')^i$ be almost complex
structures tailored to $(\alpha')^i$.  Also let $(M_n,\alpha^i_{n,B},
J^i_{n,B})$ be the extensions of $(M',(\alpha')^i,(J')^i)$, as
described earlier. Let $(\beta')^i_0=(\alpha')^i|_{\bdry
  R_+(\Gamma')}$ and $(\beta')^i=(\alpha')^i|_{R_+(\Gamma')}$. As in
the proof of Proposition~\ref{prop: naturality sutured gluing case},
we can write $(\alpha')^0=f(\alpha')^1$ and $(\beta')^0_0=
(\beta')^1_0$. Also, if the manifold $M_n$ is fixed, then we can write
$\alpha^0_n=f_{n,B} \alpha^1_{n,B}$.

We construct a $1$-parameter family $f(\rho)(\alpha')^1$,
$\rho\in[0,1]$, $f(0)=f$, $f(1)=1$, to construct a symplectic
cobordism and a continuation map
$$\Theta'\colon  \mathcal{A}(M',(\alpha')^0,(J')^0)\to
\mathcal{A}(M',(\alpha')^1,(J')^1).$$ Next we extend $f(\rho)$ to
$f_{n,B}(\rho)$, so that $f_{n,B}(0)=f_{n,B}$ and $f_{n,B}(1)=1$.
Using $f_{n,B}(\rho)\alpha^1_{n,B}$, we obtain a continuation map
$$\Theta^{n}_{B}\colon \mathcal{A}(M_n, \alpha^0_{n,B},J^0_{n,B})
\to \mathcal{A}(M_{n},\alpha^1_{n,B},J^1_{n,B}).$$

Let $K>0$. Then there exists $K'>0$ such that
$$\Theta'(\mathcal{A}_{\leq K}(M',(\alpha')^0,(J')^0))\subset
\mathcal{A}_{\leq K'}(M',(\alpha')^1,(J')^1).$$ For sufficiently
large $n$ there exists $B_0(n)$ such that for $B\geq B_0(n)$ the
following diagram commutes:
\begin{equation} \label{diagram: commut1}
\begin{diagram}
\mathcal{A}_{\leq K}(M',(\alpha')^0,(J')^0) & \rTo^{\Phi^0_{K,n,B}}
&
\mathcal{A}(M_n, \alpha^0_{n,B},J^0_{n,B})  \\
\dTo^{\Theta'} & &  \dTo_{\Theta^{n}_{B}} \\
\mathcal{A}_{\leq K'}(M',(\alpha')^1,(J')^1)  &
\rTo^{\Phi^1_{K',n,B}} &
\mathcal{A}(M_{n},\alpha^1_{n,B},J^1_{n,B})  \\
\end{diagram}
\end{equation}

Taking direct limits, we obtain the following commutative diagram:
\begin{equation}
\begin{diagram}
HC(M',(\alpha')^0,(J')^0) & \rTo &
\lim_{n\to \infty}HC(M_n, \alpha^0_{n,B(n)},J^0_{n,B(n)})  \\
\dTo^{\Theta'} & &  \dTo_{\Theta} \\
HC(M',(\alpha')^1,(J')^1)  & \rTo &
\lim_{n\to\infty} HC(M_{n},\alpha^1_{n,B(n)},J^{1}_{n,B(n)})  \\
\end{diagram}
\end{equation}
which proves that the two versions of the map $\Phi$ agree.

\s\n {\em Acknowledgements.}  It is a pleasure to thank Yasha
Eliashberg and Jian He for very helpful discussions.  Part of this
work was done while KH, MH and PG visited MSRI during the academic year
2009--2010; they would like to thank MSRI and the organizers of the
``Symplectic and Contact Geometry and Topology''  and the ``Homology Theories of 
Knots and Links'' programs for their hospitality.

\end{document}